\documentclass{amsart}

\usepackage{wasysym}
\usepackage{amsmath}
\usepackage{amssymb}
\usepackage{amsthm}
\usepackage{tikz}
\usetikzlibrary{matrix,arrows,backgrounds,shapes.misc,shapes.geometric,patterns,calc,positioning}
\usetikzlibrary{calc,shapes}
\usepackage{wrapfig}
\usepackage{epsfig}  
\usepackage{color}	 %
\input{xy}
\xyoption{poly}
\xyoption{2cell}
\xyoption{all}

\usepackage{lscape}

\usetikzlibrary{calc,shapes}
\usetikzlibrary{matrix,arrows,backgrounds,shapes.misc,shapes.geometric,patterns,calc,positioning}
\usetikzlibrary{calc,shapes}

\newcommand{\calf}{\mathcal{F}}
\newcommand{\cala}{\mathcal{A}}

\newcommand{\calg}{\mathcal{G}}
\newcommand{\calh}{\mathcal{H}}
\newcommand{\call}{\mathcal{L}}
\newcommand{\ocalg}{\overline{\mathcal{G}}}

\newcommand{\tcalg}{\widetilde {\mathcal{G}}}
\newcommand{\calgSW}{{}_{SW}\calg}
\newcommand{\calgNE}{\calg^{N\!E}}

\makeatletter
\def\Ddots{\mathinner{\mkern1mu\raise\p@
\vbox{\kern7\p@\hbox{.}}\mkern2mu
\raise4\p@\hbox{.}\mkern2mu\raise7\p@\hbox{.}\mkern1mu}}
\makeatother

\newtheorem{thm}{Theorem}[section]
\newtheorem{prop}[thm]{Proposition}
\newtheorem{lem}[thm]{Lemma}
\newtheorem{cor}[thm]{Corollary}

\theoremstyle{defn}
\newtheorem{defn}[thm]{Definition}
\newtheorem{example}[thm]{Example}

\theoremstyle{remark}

\newtheorem{remark}[thm]{Remark}

\numberwithin{equation}{section}

\DeclareMathSizes{5}{3}{2}{2}

\newcommand{\gi}[3]{G_{#1}, G_{#2}, \ldots, G_{#3}}
\newcommand{\gii}[3]{G'_{#1}, G'_{#2}, \ldots, G'_{#3}}

\newcommand{\gr}[3] {\calg_1 = (G_{#1}, G_{#2}, \ldots, G_{#3})}

\newcommand{\re}[2] {\res_{\calg} ( \calg_{#1}, \calg_{#2})}
\newcommand{\ree}[1] {\res_{\calg} ( \calg_{#1})}
\newcommand{\gt}[4] {\graft_{{#1}, {#2}} ( \calg_{#3}, \calg_{#4})}

\newcommand{\la}[2] {\call ( \calg_{#1} \sqcup \calg_{#2}) }

\newcommand\restr[2]{{
  \left.\kern-\nulldelimiterspace 
  #1 
  \vphantom{\big|} 
  \right|_{#2} 
  }}
  
\newcommand{\e}{\delta}

\def\B{\mathcal{B}}

\newcommand{\za}{\alpha}

\newcommand{\zd}{\delta}
\newcommand{\zD}{\Delta}

\newcommand{\zg}{\gamma}

\newcommand{\zS}{\Sigma}

\DeclareMathOperator{\Bang}{Bang}

\DeclareMathOperator{\Brac}{Brac}

\DeclareMathOperator{\Match}{Match}

\newcommand{\pred}{\textup{pred}}
\newcommand{\suc}{\textup{succ}}
\newcommand{\Int}{\textup{Int}}

\newcommand{\C}{\mathcal{C}}
\newcommand{\A}{\mathcal{A}}

\def\Acal{\mathcal{A}}

\DeclareMathOperator{\res}{Res \,} 
\DeclareMathOperator{\graft}{Graft \,  } 
 
\DeclareMathOperator{\match}{Match \, }

\begin{document}
\title[ Snake graph calculus II: Self-crossing snake graphs]{Snake graph calculus and cluster algebras from surfaces II: Self-crossing snake graphs}
\author{Ilke Canakci}
\address{Department of Mathematics 
University of Leicester 
University Road 
Leicester LE1 7RH 
United Kingdom}
\email{ic74@le.ac.uk}
\author{Ralf Schiffler}\thanks{The authors were supported by NSF grant DMS-10001637; the first author was also supported by EPSRC grant number EP/K026364/1, UK,  and the second author was also supported by the NSF grants  DMS-1254567, DMS-1101377 and by the University of Connecticut.}
\address{Department of Mathematics, University of Connecticut, 
Storrs, CT 06269-3009, USA}
\email{schiffler@math.uconn.edu}




%

\begin{abstract}
Snake graphs appear naturally in the theory of cluster algebras. For cluster
algebras from surfaces, each cluster variable is given by a formula which is parametrized by
the perfect matchings of a snake graph. In this paper, we continue our study of snake graphs from a combinatorial point of view. We introduce the notions  of abstract snake graphs and abstract band graphs, their crossings and self-crossings, as well as the resolutions of these crossings.
We show that there is a bijection between the set of perfect matchings of (self-) crossing snake graphs and the set of perfect matchings of the resolution of the crossing.
In the situation where the snake and band graphs are coming from arcs and loops in a surface without punctures, we obtain a new proof of skein relations in the corresponding cluster algebra.
\end{abstract}

 \maketitle




\section{Introduction}
 This paper continues the study of abstract snake graphs initiated in \cite{CS}.
  Our goals are, on the one hand, to establish a computational tool for cluster algebras of surface type, which we call  \emph{snake graph calculus}, and, on the other hand, to introduce a new algebraic structure which is not limited to a particular choice of a surface but rather inspired from the combinatorial structure of  all surface type cluster algebras.
This theory has already found the following applications. In \cite{CaSc}, the authors use snake graphs to study extensions of modules over Jacobian algebras and triangles in the cluster category associated to triangulations of unpunctured surfaces, and in \cite{CLS}, the snake graph calculus is used to show that for unpunctured surfaces with exactly one marked point, the upper cluster algebra coincides with the cluster algebra, which was one of the last cases of the mutation finite cluster algebras for which the question was still open. 

Cluster algebras were introduced in \cite{FZ1}, and further developed in \cite{FZ2,BFZ,FZ4}, motivated by combinatorial aspects of canonical bases in Lie theory \cite{L1,L2}. A cluster algebra is a subalgebra of a field of rational functions in several variables, and it is given by constructing a distinguished set of generators, the \emph{cluster variables}. These cluster variables are constructed recursively and their computation is rather complicated in general. By construction, the cluster variables are rational functions, but Fomin and Zelevinsky showed in \cite{FZ1} that they are Laurent polynomials with integer coefficients. Moreover, these coefficients are known to be non-negative \cite{LS4}.

An important class of cluster algebras is given by cluster algebras of surface type \cite{GSV,FG1,FG2,FST,FT}. From a classification point of view, this class is very important, since it has been shown in \cite{FeShTu} that almost all (skew-symmetric) mutation finite cluster algebras are of surface type. For generalizations to the skew-symmetrizable case see \cite{FeShTu2,FeShTu3}. The closely related surface skein algebras were studied in \cite{M,T}.

If $\cala$ is a cluster algebra of surface type, then there exists a surface with boundary and marked points such that the cluster variables of $\cala$ are in bijection with certain isotopy classes of curves, called \emph{arcs}, in the surface. Moreover, the relations between the cluster variables are given by the crossing patterns of the arcs in the surface. In \cite{MSW}, building on earlier work \cite{S2,ST,S3,MS}, the authors gave a combinatorial formula for the cluster variables in cluster algebras of surface type. In the sequel \cite{MSW2}, the formula was the key ingredient for the construction of two bases for the cluster algebra, in the case where the surface has no punctures and has at least 2 marked points.
As an application of the computational tools developed in \cite{CS} and in the present paper, it is proved in \cite{CLS} that the basis construction of \cite{MSW2} also applies to surfaces with non-empty boundary and with exactly one marked point.

In order to construct these bases, one associates Laurent polynomials to certain curves in the surface. If the curve is an arc, these Laurent polynomials are cluster variables and are given by the formula of \cite{MSW} in terms of perfect matchings of snake graphs. If the curve is a closed loop, one needs to replace the snake graph by a band graph, and then the Laurent polynomial is still given in terms of perfect matchings of the band graph, see \cite{MSW2}.

Perfect matchings of certain graphs have also been used in \cite{MSc} to give expansion formulas in the cluster algebra structure of the homogenous coordinate ring of a Grassmannian.

In our previous work \cite{CS}, we have studied the snake graphs from a combinatorial point of view. Instead of constructing snake graphs from arcs in a  fixed surface, we gave an abstract definition of a snake graph and studied the properties of these graphs. These \emph{abstract} snake graphs are not necessarily related to the geometric situation of arcs in a surface. In analogy with the geometric situation, we defined the notion of \emph{crossing} snake graphs and constructed the resolution of a pair of crossing snake graphs as two pairs of snake graphs. We then constructed a bijection between the set of perfect matchings of the pair of crossing snake graphs and the set of perfect matchings of its resolution. We also showed that, in the case where the snake graphs actually correspond to arcs in a surface,  the resolution of crossing snake graphs corresponds precisely to the smoothing of the crossing of the associated arcs. In particular, we obtained a new proof of the skein relations between cluster variables in the cluster algebra.

However, in order to understand the cluster algebra, arcs alone do not provide enough information. One also needs to consider self-crossing curves and closed loops. Self-crossing curves {{may}} appear already after smoothing a crossing of two arcs which have more than one crossing point, and closed loops {{may}} appear  after smoothing a self-crossing. In section~\ref{sect 2} of the present paper, we introduce the notion of self-crossing snake graphs, and we construct the resolutions of the self-crossings in section~\ref{sect 3a}. In order to describe these resolutions, the snake graphs alone are no longer sufficient; one also needs to work with band graphs. In section~\ref{sect 3}, we construct a bijection between the set of perfect matchings of a self-crossing snake graph and the set of perfect matchings of its resolution.
We then show in section~\ref{sect 5} that, in the case where the self-crossing snake graph is actually coming from a self-crossing arc in a surface, the resolution of the snake graph corresponds exactly to the smoothing of the crossing in the curve. We also show {{in section~\ref{sect 7}}} that the corresponding skein relation holds in the cluster algebra.

The step from resolving the crossing of a pair of snake graphs to resolving a self-crossing snake graph is surprisingly difficult. The naive approach of cutting one self-crossing snake graph into two crossing snake graphs and then applying the results of \cite{CS} does not work in general, because, in a self-crossing snake graph, the two positions where the crossing occurs may have an intersection and cannot be separated.
In the geometric setting, this corresponds to a curve that spirals around a boundary component several times, approaching the boundary, and then running away from the boundary, thereby crossing itself several times.

Even for snake graphs that have a geometric interpretation as curves in a surface, there is a fundamental difference between the smoothing of a crossing of curves and the resolution of a crossing of snake graphs. The definition of smoothing is very simple. It is defined as a local transformation replacing a crossing {\Large $\times$} with the pair of segments 
  $\genfrac{}{}{0pt}{5pt}{\displaystyle\smile}{\displaystyle\frown}$ (resp. {$\supset \subset$}).    
But once this local transformation is done, one needs to find representatives inside the isotopy classes of the resulting curves which realize the minimal number of crossings with the fixed triangulation. This means that one needs to deform the obtained curves isotopically, and to 'unwind' them if possible, in order to see their actual crossing pattern, which is crucial for the applications to cluster algebras. This can be quite confusing, especially in a higher genus surface.

The situation for the snake and band graphs is exactly opposite. The definition of the resolution is very complicated because one has to consider many different cases. But once all these cases are worked out, one has a complete list of rules in hand, which one can apply very efficiently in actual computations.

In \cite{CS}, we gave these rules for pairs of crossing snake graphs, in the present paper, we treat the case of self-crossing snake graphs, and in a forthcoming paper \cite{CS3}, we will complete the work by treating crossings of band graphs and self-crossing band graphs.

In the last section, we give an example of an explicit computation of the product of two cluster variables in the cluster algebra of the torus with one boundary component and one marked point.
  
 \noindent \emph{Acknowledgements.}  
{We thank Anna Felikson for suggesting improvements to the presentation of the article.}
  
\section{Abstract snake graphs and abstract band graphs}
\label{sect 2}
 Abstract snake graphs have been introduced in \cite{CS} motivated by the snake graphs appearing in the combinatorial formulas for cluster variables in cluster algebras of surface type  in \cite{Propp, MS, MSW}. Here we introduce the notion of abstract band graphs which is motivated by the band graphs used in \cite{MSW2} to construct a bases for cluster algebras of surface type.

The construction of abstract snake graphs and band graphs is completely detached from triangulated surfaces. Our goal is to study these objects in a combinatorial way. 
We shall simply say snake graphs and band graphs, since we shall always mean \emph{abstract} snake graphs and \emph{abstract} band graphs. 
 
In this section, we recall the constructions of \cite{CS} pertaining to the snake graphs and, at the same time, we develop the analogue constructions for the band graphs.  Throughout
we fix an orthonormal basis of the plane.

\subsection{Snake graphs} A {\bf tile} $G$ is a square of fixed side-length in the plane whose sides are parallel or orthogonal  to the fixed basis.
\begin{center}
  {\small
\begingroup%
  \makeatletter%
  \providecommand\color[2][]{%
    \errmessage{(Inkscape) Color is used for the text in Inkscape, but the package 'color.sty' is not loaded}%
    \renewcommand\color[2][]{}%
  }%
  \providecommand\transparent[1]{%
    \errmessage{(Inkscape) Transparency is used (non-zero) for the text in Inkscape, but the package 'transparent.sty' is not loaded}%
    \renewcommand\transparent[1]{}%
  }%
  \providecommand\rotatebox[2]{#2}%
  \ifx\svgwidth\undefined%
    \setlength{\unitlength}{68.59006958bp}%
    \ifx\svgscale\undefined%
      \relax%
    \else%
      \setlength{\unitlength}{\unitlength * \real{\svgscale}}%
    \fi%
  \else%
    \setlength{\unitlength}{\svgwidth}%
  \fi%
  \global\let\svgwidth\undefined%
  \global\let\svgscale\undefined%
  \makeatother%
  \begin{picture}(1,0.70549664)%
    \put(0,0){\includegraphics[width=\unitlength]{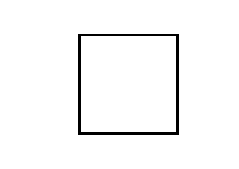}}%
    \put(0.49750979,0.31061177){\color[rgb]{0,0,0}\makebox(0,0)[lb]{\smash{$G$}}}%
    \put(0.00429408,0.31029854){\color[rgb]{0,0,0}\makebox(0,0)[lb]{\smash{West}}}%
    \put(0.79324631,0.31021311){\color[rgb]{0,0,0}\makebox(0,0)[lb]{\smash{East}}}%
    \put(0.36461283,0.61868105){\color[rgb]{0,0,0}\makebox(0,0)[lb]{\smash{North}}}%
    \put(0.36019916,0.00759722){\color[rgb]{0,0,0}\makebox(0,0)[lb]{\smash{South}}}%
  \end{picture}%
\endgroup%
}
\end{center}
We consider a tile $G$ as  a graph with four vertices and four edges in the obvious way. A {\em snake graph} $\calg$ is a connected graph consisting of a finite sequence of tiles $ \gi 12d$ with $d \geq 1,$ such that for each $i=1,\dots,d-1$

\begin{itemize}
 \item[(i)] $G_i$ and $G_{i+1}$ share exactly one edge $e_i$ and this edge is either the north edge of $G_i$ and the south edge of $G_{i+1}$ or the east edge of $G_i$ and the west edge of $G_{i+1}.$
 \item[(ii)] $G_i $ and $G_j$ have no edge in common whenever $|i-j| \geq 2.$
 \item[(ii)] $G_i $ and $G_j$ are disjoint whenever $|i-j| \geq 3.$

\end{itemize} 
An example is given in Figure \ref{signfigure}.
\begin{figure}
\begin{center}
  {\tiny \scalebox{0.9}{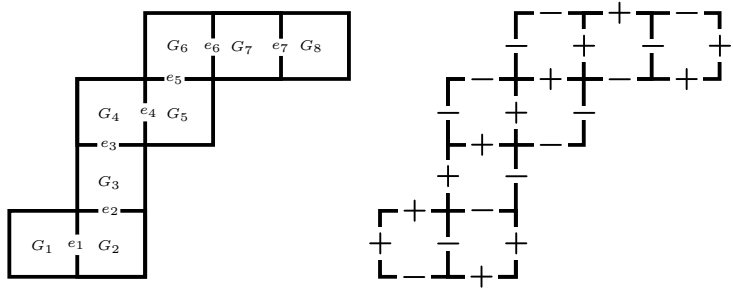}}
 \caption{A snake graph with 8 tiles and 7 interior edges (left);
 a sign function on the same snake graph (right)} 
 \label{signfigure}
\end{center}
\end{figure}
The graph consisting of two vertices and one edge joining them is also considered a snake graph.

 We sometimes use the notation $\calg =(\gi 12d)$ for the snake graph and $\calg [i, i+t] = (\gi i{i+1}{i+t})$ for the subgraph of $\calg$ consisting of the tiles $\gi i{i+1}{i+t}.$
 {One may think of this subgraph as a closed interval inside $\calg$.}

The $d-1$ edges $e_1,e_2, \dots, e_{d-1}$ which are contained in two tiles are called {\em interior edges} of $\calg$ and the other edges are called {\em boundary edges.}  
{Let $\Int(\calg)=\{e_1,e_2,\ldots,e_{d-1}\}$ be the set of interior edges of $\calg$.
We will always use the natural ordering of the set of interior edges, so that $e_i$ is the edge shared by the tiles $G_i$ and $G_{i+1}$.}

{If $i\ge 2$ and $i+t\le d-1$, the notation $\calg(i,i+t)$ means $\calg[i,i+t]\setminus\{e_{i-1},e_{i+t}\}$.  One may think of this subgraph as a open interval inside $\calg$}

{We denote by  $\calgSW$ the 2 element set containing the south and the west edge of the first tile of $\calg$ and by $\calgNE$ the 2 element set containing the north and the east edge of the last tile of $\calg$. 

If  $\calg=(\gi 12d)$ is a snake graph and $e_i$ is the interior edge shared by the tiles $G_i$ and $G_{i+1}$, we define the snake graph $\calg\setminus \textup{pred} (e_i)$ to be the graph obtained from $\calg$ by removing the vertices and edges that are predecessors of  $e_i$, more precisely,
\[ \calg\setminus \pred (e_i) =\calg[i+1 , d].\]
It will be convenient to extend this construction to the edges $e\in \calgNE $, thus 
\[ \calg\setminus \pred (e) =\{e\} \quad\textup{if } e\in \calgNE.\]  
Similarly, we define the snake graph $\calg\setminus \textup{succ} (e_i)$ to be the graph obtained from $\calg$ by removing the vertices and edges that are successors of  $e_i$, more precisely,
\[ \calg\setminus \suc (e_i) =\calg[1 , i].\]
It will be convenient to extend this construction to the edges $e\in \calgSW $, thus 
\[ \calg\setminus \suc (e) =\{e\} \quad\textup{if } e\in \calgSW.\] } 

A snake graph $\calg$ is called {\em straight} if all its tiles lie in one column or one row, and a snake graph is called {\em zigzag} if no three consecutive tiles are straight.
\subsection{Sign function} 

A {\em sign function} $f$ on a snake graph $\calg$ is a map $f$ from the set of edges of $\calg$ to $\{ +,- \}$ such that on every tile in $\calg$ the north and the west edge have the same sign, the south and the east edge have the same sign and the sign on the north edge is opposite to the sign on the south edge. See Figure \ref{signfigure} for an example.
%
%

Note that on every snake graph  with at least one tile, there are exactly two sign functions.

{Although the definition of sign function may not seem natural at first sight, it has a geometric meaning which is explained in Remark \ref{rem sign}. }

\subsection{Band graphs}\label{sect band} {Band graphs are obtained from snake graphs by identifying a boundary edge of the first tile with a boundary edge of the last tile, {where both edges have the same sign}. We use the notation $\calg^\circ$ for general band graphs, indicating their circular shape, and we also use the notation $\calg^b$ if we know that the band graph is constructed by glueing a snake graph $\calg$ along an edge $b$.} 

More precisely, to define a band graph $\calg^\circ$, we start with an abstract snake graph $\calg=(\gi 12d)$ with $d\ge 1$, and fix a sign function on $\calg$. Denote by $x$ the southwest vertex of $G_1$, let $b\in\calgSW$ the south edge (respectively the west edge) of $G_1$, and let $y$ denote the other endpoint of $b$, see Figure \ref{figband}. Let $b'$ be the unique edge in $\calgNE$ that has the same sign as  $b$, and let $y'$ be the northeast vertex of $G_d$ and $x'$ the other endpoint of $b'$.
\begin{wrapfigure}{l}{0.15\textwidth}
  \begin{center}
    \scalebox{0.15}{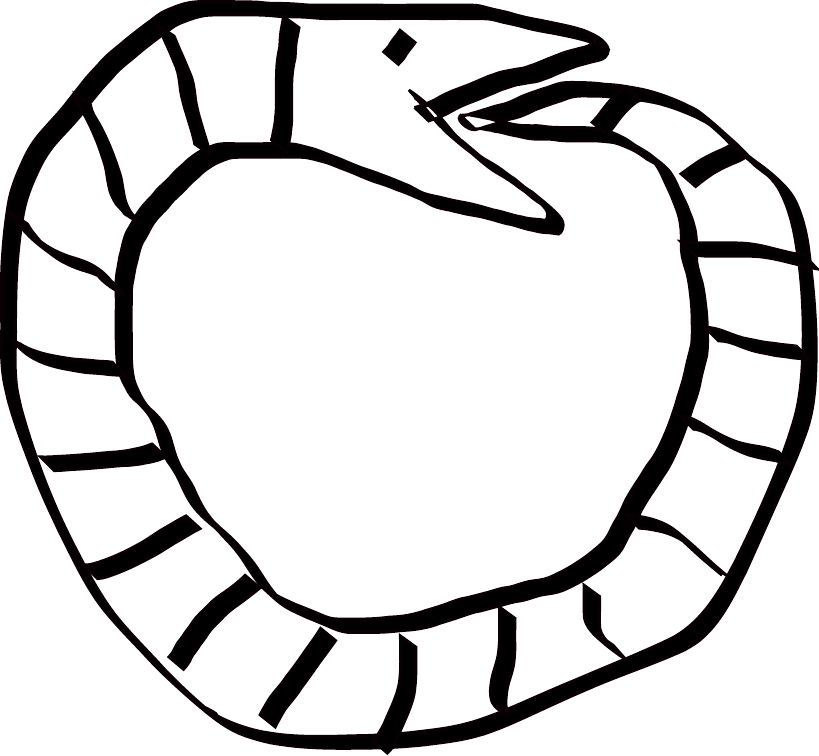}
  \end{center}
\end{wrapfigure} 
\indent Let $\calg^b$ denote the graph obtained from $\calg $ by identifying the edge $b$ with the edge $b'$ and the vertex $x$ with $x'$ and $y$ with $y'$. The graph $\calg^b$ is called a \emph{band graph} or \emph{ouroborus}\footnote{Ouroboros: a snake devouring its tail.}. Note that non-isomorphic snake graphs can give rise to isomorphic band graphs. See Figure \ref{figband} for an example.


\begin{figure}
\begin{center}
  {\Large \scalebox{.7}{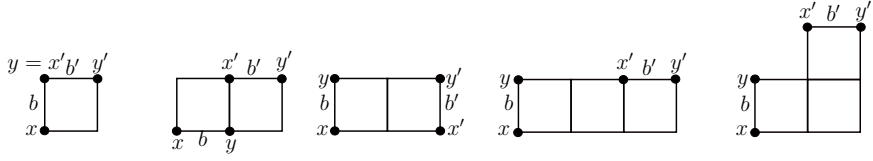}}
 \caption{Examples of small band graphs; the two band graphs with 3 tiles are isomorphic.}
 \label{figband}
\end{center}
\end{figure}

The interior edges of the band graph $\calg^b$ are by definition the interior edges of $\calg$ plus the glueing edge $b=b'$.
Given a  band graph $\calg^\circ$ with an interior edge $e$, we denote by $\calg^\circ_e $ the snake graph obtained by cutting $\calg^\circ$ along the edge $e$. 
Note that  $(\calg^\circ_e)^e = \calg^\circ$, for all band graphs $\calg^\circ$ and that $({\calg^b})_b=\calg$, for all snake graphs $\calg$.
Moreover, if $\calg^\circ$ has $d$ interior  edges $e_1,e_2,\ldots,e_d$ then the $d$ snake graphs  $\calg^\circ_{e_i}$, $i=1,\dots,d $,
are not necessarily distinct.

\begin{defn}\label{def group}
 Let  $\mathcal{R}$  denote the free abelian group generated by all isomorphism classes of {finite disjoint} unions of snake graphs and band graphs.  {If $\calg$ is a snake graph, we also denote its class in $\mathcal{R}$ by $\calg$, and we say that $\calg\in \mathcal{R}$ is a {\em positive} snake graph  and that its inverse  $-\calg\in \mathcal{R}$ is a {\em negative} snake graph.}
 \end{defn}

\subsection{Labeled snake and band graphs}
{
 A \emph{labeled} snake graph is a snake graph in which each edge and each tile carries a label or weight. 
For example, for snake graphs from cluster algebras of surface type, these labels are  cluster variables.

Formally, a labeled snake graph is a snake graph $\calg$ together with two functions \[\{\textup {tiles in }\calg\}\to \calf \qquad\textup{ and }\qquad \{\textup {edges in }\calg\}\to \calf,\] 
where $\calf$  is a  set.
Labeled band graphs are defined in the same way. 

 Let  $\mathcal{LR_\calf}$  denote the free abelian group generated by all isomorphism classes of unions of labeled snake graphs and labeled band graphs with labels in $\calf$. 
}
\subsection{Overlaps and self-overlaps} 

Let $\calg_1=(\gi 12d)$ and $\calg_2=(\gii 12{d'})$ be two snake graphs. We say that $\calg_1$ and $\calg_2$ have an {\em overlap} $\calg$ if $\calg$ is a snake graph consisting of at least one tile and there exist two embeddings $i_1 : \calg \rightarrow \calg_1,$ $i_2: \calg \rightarrow \calg_2$ which are maximal in the following sense.

\begin{itemize}
\item[(i)] If $\calg$ has at least two tiles and if there exists a snake graph $\calg'$ with two embeddings $i'_1: \calg' \to \calg_1,$ $i'_2: \calg \to \calg_2$ such that $i_1(\calg) \subseteq i'_1(\calg')$ and $i_2(\calg) \subseteq i'_2(\calg')$ then $i_1(\calg) = i'_1(\calg')$ and $i_2(\calg) = i'_2(\calg').$
\item[(ii)] If $\calg$ consists of a single tile then using the notation $G_k=i_1(\calg)$ and $G'_{k'}= i_2(\calg),$ we have
\begin{itemize}
 \item[(a)] $k \in \{ 1,d \}$ or $k' \in \{ 1, d' \}$ or
 \item[(b)] $1<k<d,$ $1<k'<d'$ and the subgraphs $(G_{k-1},G_k,G_{k+1})$ and $(G'_{k'-1},G'_{k'},G'_{k'+1})$ are either both straight or both zigzag subgraphs.
\end{itemize}
\end{itemize}
An example of type (i) is shown on the left in Figure \ref{overlap} and an example of type (ii)(b) on the right   in the same figure.
\begin{remark}
Snake graphs may have several overlaps with respect to different snake graphs $\calg$.
\end{remark}
\bigskip

\begin{figure}
\begin{center}
\scalebox{0.5} { 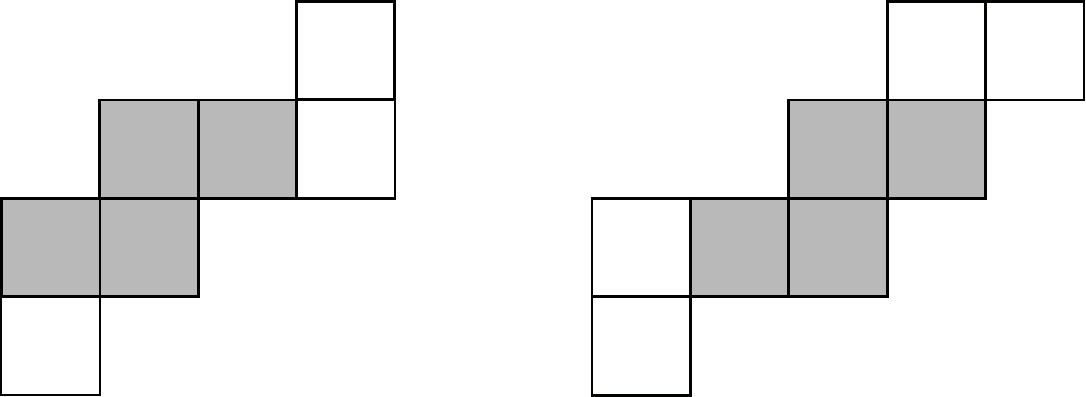 \hspace{2in} 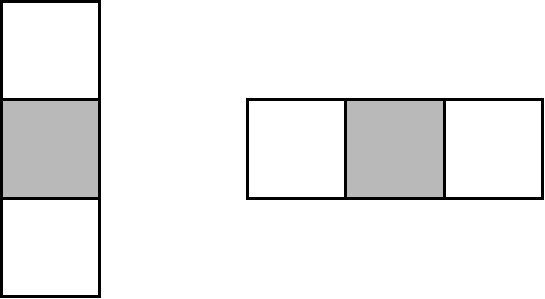}
 \caption{Two snake graphs with overlap shaded (left); two snake graphs with overlap consisting of a single tile shaded (right)}
 \label{overlap}
 \end{center}
\end{figure}

\begin{figure}
\begin{center}
\scalebox{1} { 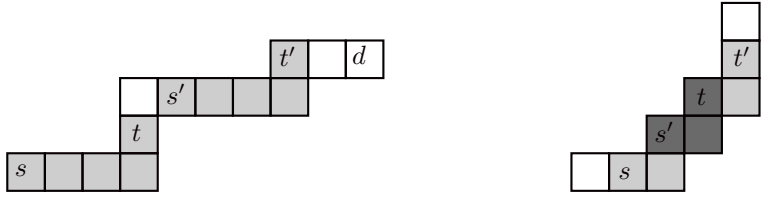}
 \caption{Two snake graphs with self-overlap $(G_s,\ldots,G_t) $ and $(G_{s'},\ldots,G_{t'})$. The self-overlap on the right has an intersection consisting of 3 tiles $G_{s'},G_{s'+1}$ and $G_t$.}
 \label{figselfoverlap}
 \end{center}
\end{figure}

\begin{figure}
\begin{center}
\scalebox{1.2} {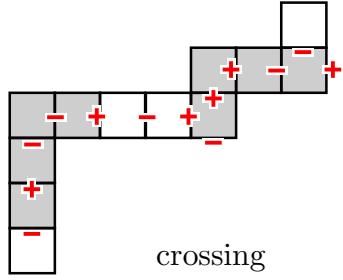}
 \caption{Example of a snake graph with crossing self-overlap (shaded) in the opposite direction. }
 \label{figselfoverlap2}
 \end{center}
\end{figure}

We say that a snake graph $\gr12d$ has a \emph{self-overlap} $\calg$ if $\calg$ is a snake graph and there exist two embeddings $i_1 : \calg \rightarrow \calg_1,$ $i_2: \calg \rightarrow \calg_1$ which satisfy the conditions (i) and (ii) above. Examples of  self-overlaps are shown in Figures \ref{figselfoverlap} and \ref{figselfoverlap2}.

Let $s,t,s',t'$ be such that $i_1(\calg)=\calg_1[s,t]$ and $i_2(\calg)=\calg_1[s',t']$ with $s\le t, s'\le t'$ and suppose without loss of generality that $s<s'.$
The self-overlap is said to have an \emph{intersection} if $i_1(\calg)\cap i_2(\calg)$ contains at least one edge, that is $s'\le t+1$. In this case, we say $\calg_1$ has an \emph{intersecting self-overlap}, see the right picture in Figure \ref{figselfoverlap}.
 
We may assume without loss of generality that the embedding $i_1:\calg\to\calg_1[s,t]$ maps the southwest vertex of  the first tile of $\calg$ to the southwest vertex of  $G_s$ in $\calg_1[s,t]$. 
We then say that the self-overlap is in the \emph{same direction} if the embedding $i_2:\calg\to\calg_1[s',t']$  maps the southwest vertex of  the first tile of $\calg$ to the southwest vertex of  $G_{s'}$ in $\calg_1[s',t']$, and we say that the self-overlap is in the \emph{opposite direction} if $i_2$ maps this vertex to the northeast vertex of $G_{t'}$ in $\calg_1[s',t']$.
The self-overlaps in Figure \ref{figselfoverlap} are in the same direction and the self-overlap in Figure~\ref{figselfoverlap2} is in the opposite direction.

\begin{remark}
\begin{enumerate}
\item The notion of direction depends on the embeddings $i_1$ and $i_2$ and not only on the subgraphs $\calg_1[s,t]$ and $ \calg_1[s',t']$. See Figure \ref{opposite2} for an example where  $\calg_1[s,t]$ and $ \calg_1[s',t']$ can be considered as overlap in either direction.
\item  $\calg_1$ may have an intersecting self-overlap such that the intersection of  $i_1(\calg)$ and $i_2(\calg)$   is a single edge. In this case, we have $e_t=e_{s'-1}$, {see Figure \ref{fig(s'=t+1)} for examples.}
\end{enumerate}
\end{remark}

{For labeled snake graphs, we define overlaps by adding the  requirement that the embeddings $i_1$ and $i_2$ are label preserving.}

\subsection{Crossing overlaps}

Let $\gr 12d,$ ${{\calg_2 = (G'_1, G'_2, \ldots, G'_{d'})}}$ be two snake graphs with $d,d'\ge1$, with overlap $\calg$ and embeddings $i_1(\calg)=\calg_1[s,t]$ and $i_2(\calg)={{\calg_2[s',t']}}$, and suppose without loss of generality that $s \leq t$ and $s' \leq t'.$ Let $e_1, \dots, e_{d-1} $ (respectively $e'_1, \dots, e'_{d'-1}$) be the interior edges of $\calg_1$ (respectively $\calg_2$.) Let $f$ be a sign function on $\calg.$ Then $f$ induces a sign function $f_1$ on $\calg_1$ and $f_2$
 on $\calg _2.$ Moreover, since the overlap $\calg$ is maximal, we have 
 
\[\begin{array}{rcll}
 f_1(e_{s-1}) &=&-f_2(e'_{s'-1}) & \mbox{ if }  s>1, s'>1\\
f_1(e_{t})&=&-f_2(e'_{t'})&\mbox{ if }t <d, t'< d'.
\end{array}
\]%
 
\begin{defn} \label{crossing} We say that $\calg_1$ and $\calg_2$ {\em cross in} $\calg$ if one of the following conditions hold.

\begin{itemize}
 \item[(i)] \label{i}
\[ \begin{array}{lrclll}
 &f_1(e_{s-1})& =& -f_1(e_t) & \mbox{ if } &s>1, t<d   \\
 \mbox{ or}\\
&  f_2 (e'_{s'-1})&=&-f_2(e'_{t'}) & \mbox{ if } &s'>1, t'<d'  
\end{array}\]
%
 \item[(ii)]  \[\begin{array}{lrcllcc}
 &f_1(e_{t}) &=& f_2(e'_{s'-1})& \mbox{ if } &s=1, t<d, s'>1, t'=d'  \\
  \mbox{ or}\\
&f_1 (e_{s-1})&=&f_2(e'_{t'}) & \mbox{ if } &s>1 , t=d, s'=1, t'<d' &
\end{array}\] 
\end{itemize}
\end{defn} 

For example, the two snake graphs on the left hand side of Figure \ref{overlap} cross in the shaded overlap because $f_1(e_{s-1}) = -f_1(e_t)$. The two snake graphs on the right hand side of the same figure cross for the same reason.
\begin{remark}
\begin{itemize}
 \item[1.] The definition of crossing does not depend on the choice of the sign function.
 \item[2.] $\calg_1$ and $\calg_2$ may still cross if $s=1$ and $t=d$ because they may satisfy condition \ref{i}(i).
\end{itemize}
\end{remark}

\subsection{Crossing self-overlaps} 
Let $\gr12d$ be a snake graph with self-overlap $i_1(\calg)=\calg_1[s,t]$ and $i_2(\calg)=\calg_1[s',t']$, where $s\le s'$. Let $f$ be a sign function on $\calg_1$.

\begin{defn}\label{def self-crossing}
 With the above  notation, we say that $\calg_1$ has a \emph{self-crossing} (or \emph{self-crosses}) in $\calg$ if  the following two conditions hold
\begin{itemize}
 \item[(i)] \label{2.4i}
\[ \begin{array}{lrcllclrclll}
 &f(e_{s-1})& =& -f(e_t) & \mbox{ if } s>1;   &
 \mbox{ or}
&  f (e_{s'-1})&=&-f(e_{t'}) & \mbox{ if }  t'<d;  \ and
\end{array}\]
 \item[(ii)] 
  \[\begin{array}{lrcllcc}
 &f(e_{t}) &=& f(e_{s'-1}).&
\end{array}\] 
\end{itemize}
\end{defn}

{
\begin{remark}
If $s>1$ and $t'<d$ then either one of the conditions in Definition~\ref{def self-crossing}(i) implies the other. Indeed, this follows from the two overlap conditions $f(e_{s-1})=-f(e_{s'-1})$ and $f(e_t)=-f(e_{t'})$. 
\end{remark}
}
\begin{example}
 Consider the self-overlaps  in Figure \ref{figselfoverlap}. In the example on the left side of the figure, the edge $e_t$ is the edge shared by the first shaded region and the first white region, and the edge $e_{s'-1}$ is the edge shared by the first white region and the second shaded region. Thus $f(e_t)= f(e_{s'-1})$. Moreover, the edge $e_{t'}$ is the edge shared by the second shaded region and the second white region. Thus $f(e_{s'-1})=-f(e_{t'})$  and the snake graph has a self-crossing in this self-overlap. 

The example on the right side of  Figure \ref{figselfoverlap}, the edge $e_t$ is the edge shared by the dark shaded region and the second light shaded region, and the edge $e_{s'-1}$ is the edge shared by the dark shaded region and the first light shaded region. Thus $f(e_t)= f(e_{s'-1})$. Moreover, $e_{t'}$ is the edge shared by the second light shaded region and the second white region. Thus $f(e_{s'-1})=-f(e_{t'})$   and  there is a self-crossing in this self-overlap. 
\end{example}
\begin{example}
 Now consider  the example in  Figure \ref{figselfoverlap2}. We have $f(e_{s-1})=- $ and $f(e_t)=+$, and $f(e_t)=+=f(e_{s'-1})$ again giving a  self-crossing. 
\end{example}

\begin{remark}
 \begin{enumerate}
\item The definition of self-crossing does not depend on the choice of the sign function $f$.
\item 
The terminology `self-cross' comes from snake graphs that are associated to generalized arcs in a surface. We shall show in  Theorem \ref{thm cross} that crossings and self-crossings of arcs in a surface  correspond precisely to crossings and self-crossings of snake graphs  in an overlap.

\end{enumerate}

\end{remark}
\section{Resolutions}\label{sect 3a}
{In this section, we define the resolutions of crossings and self-crossings. Given a pair of crossing snake graphs, or a single self-crossing snake graph, the resolution of the crossing consists of a sum of two elements in the group $\mathcal{R}$. In a forthcoming paper \cite{CS3}, we will introduce a ring structure on the group $\mathcal{R}$ and consider the ideal generated by all resolutions. Thus in the resulting quotient ring the crossing pair of snake graphs  (or the self-crossing snake graph) is equal to its resolution. This quotient ring is strongly related to cluster algebras from surfaces.}
 
\subsection {Resolution of crossing}  \label{sect res}
 Let $\calg_1,$ $\calg_2$ be two snake graphs crossing in an overlap $\calg=\calg_1[s,t]=\calg_2[s',t']$. Recall that we use the notation $\calg_k[i,j]$ for the subgraph of $\calg_k$ given by the tiles with indices $i, i+1, \dots, j.$ Let $\ocalg_k[j,i]$ be the snake graph obtained by reflecting $\calg_k[i,j]$ such that the order of the tiles is reversed.

We define four connected subgraphs $\calg_3$--$\calg_6$ below. 
In the definition of $\calg_5$ and $\calg_6$, we shall use the notation 
\[\calg_5'=\calg_1[1,s-1] \cup \ocalg_2[s'-1,1] \qquad \textup{ and } \qquad \calg_6'=\ocalg_2[d',t'+1] \cup \calg_1[t+1,d], \]
where, in the definition of $\calg_5'$, the two subgraphs are glued along the north edge of $G_{s-1}$ and the east edge  of $G'_{s'-1}$ if $G_s$ is east of $G_{s-1}$ in $\calg_1$; and along the east  edge of $G_{s-1}$ and the north edge  of $G'_{s'-1}$ if $G_s$ is north of $G_{s-1}$ in $\calg_1$, and, in the definition of $\calg_6'$, the two subgraphs are glued along the west edge  of $G_{t+1}$ and the south edge  of $G'_{t'+1}$ if $G_{t+1}$ is north of $G_{t}$ in $\calg_1$; and along the south edge  of $G_{t+1}$ and the west edge  of $G'_{t'+1}$ if $G_{t+1}$ is east of $G_{t}$ in $\calg_1$.
Let $f_5$ be a sign function on $\calg_5'$ and $f_6 $ a sign function on $\calg_6'$.

We define four connected subgraphs as follows, see Figure \ref{figres} for examples. 
\begin{align*}
 \calg_3&=\calg_1[1,t] \cup \calg_2[t'+1, d'] \mbox{ where the adjacency of the two subgraphs is induced by } \calg_2. \\
\calg_4&=\calg_2[1,t'] \cup \calg_1[t+1,d] \mbox{ where the adjacency of the two subgraphs is induced by } \calg_1.\\
 \calg_5 &=
\begin{cases} \calg_5'&  \parbox[t]{.55\textwidth}{ if $s>1$, $s'>1$ }\\
 \calg_5'\setminus \suc (e) &     \parbox[t]{.65\textwidth}{ if $s'=1$ where $e$ is the last edge in $\Int(\calg_5') \cup {}_{SW}\calg_5'$ such  that $f_5(e)=f_5(e_{s-1})$; }  \\
 \calg_5'\setminus \pred (e) &     \parbox[t]{.65\textwidth}{ if $s=1$ where $e$ is the first edge in $\Int(\calg_5') \cup \calg_5'^{N\!E}$ such  that $f_5(e)=f_5(e'_{s-1})$; }  \\ 
\end{cases}\\ 
\calg_6&=
\begin{cases}
\calg_6'&  \parbox[t]{.55\textwidth}{ if $t<d$, $t' <d'$;}\\
\calg_6'\setminus \suc (e) &     \parbox[t]{.65\textwidth}{ if $t=d$, where $e $  is the last edge in $\Int(\calg_6') \cup {}_{SW}\calg_6'$ such  that $f_6(e)=f_6(e'_{t'})$; }   \\
\calg_6'\setminus \pred (e) &     \parbox[t]{.65\textwidth}{ if $t'=d'$, where $e $  is the first edge in $\Int(\calg_6') \cup \calg_6'^{N\!E}$ such  that $f_6(e)=f_6(e_{t})$. } 
\end{cases}\\
\end{align*}

\begin{defn} \label{resolution}
 In the above situation, we say that the element $(\calg_3 \sqcup \calg_4 ) + (\calg_5 \sqcup \calg_6 )\in \mathcal{R}$
   is the {\em resolution of the crossing} of $\calg_1$ and $\calg_2$ at the overlap $\calg$ and we denote it by $\res_{\calg} (\calg_1,\calg_2).$ 
\end{defn}

If $\calg_1, \calg_2$ have no crossing in $\calg$ we let $\res_{\calg}  (\calg_1,\calg_2)=  \calg_1 \sqcup \calg_2. $

\begin{remark}
The pair $(\calg_3,\calg_4)$ still has an overlap in $\calg$ but without crossing. The pair  $(\calg_5,\calg_6)$ can be thought of as a reduced symmetric difference of $\calg_1$ and $\calg_2$ with respect to the overlap $\calg$.
\end{remark}

\begin{figure}
\begin{center}
\scalebox{0.8}{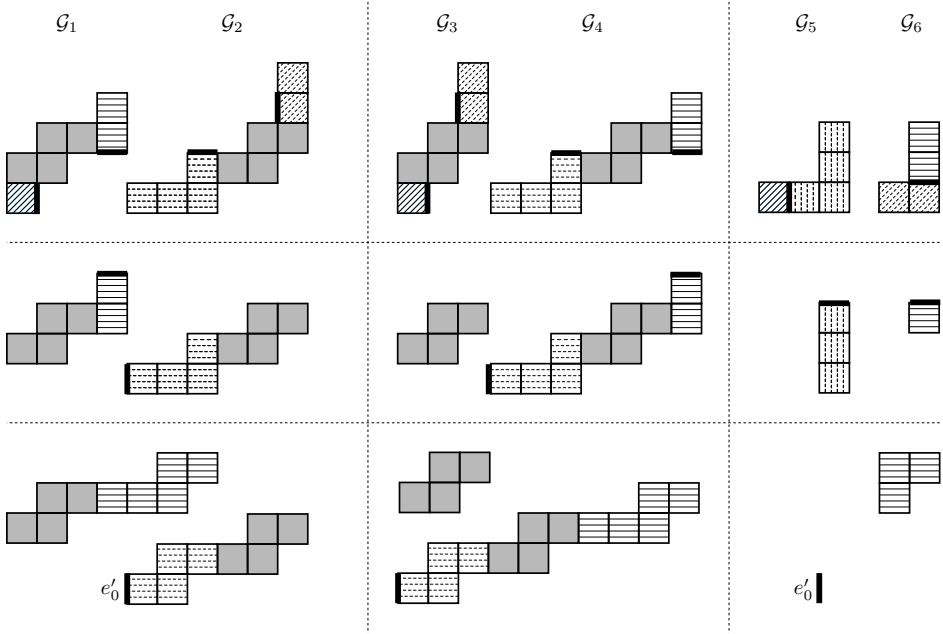}
\caption{Examples of resolutions: $s>1,s'>1,t<d,t'<d'$ in the first row;  $s=1, t'=d'$ in the second and third row;}
\label{figres}
\end{center}
\end{figure}

\subsection{Resolution of self-crossing} \label{sect resolution}
To define the resolution of a self-crossing, we construct two snake graphs and a band graph from the self-crossing snake graph. 
{We consider these snake and band graphs as elements of the group $\mathcal{R}$ of Definition~\ref{def group}. In particular, we allow them to be negative.}
In section \ref{sect 3}, we show that there is a bijection between the set of perfect matchings of a self-crossing snake graph and the set of perfect matchings of the resolution. In sections \ref{sect 5} {and \ref{sect 7}}, we show that this construction is related to multiplication formulas given by skein relations in cluster algebras.

Let $\calg_1$ be a self-crossing snake graph with self-overlap
$i_1(\calg)=\calg_1[s,t]\cong \calg_1[s',t']=i_2(\calg).$ We consider two cases.

\emph{Case 1. Overlap in the same direction.} We define the following  connected graphs. 
\[\begin{array}
{rcll}
\calg_3 &=& \calg_1[1,t]\cup\calg_1[t'+1,d], \parbox[t]{.60\textwidth}{where the adjacency of the subgraphs is induced by $\calg$;}\\ \\
\calg_4^\circ&= &\calg_1^b[s,s'-1], \parbox[t]{.70\textwidth}{ where $b=e_{s'-1}$ is the interior edge shared by $G_{s'-1}$ and $G_{s'}$;} \\ 
&=& { \calg_1^{b'}[t+1,t'],\parbox[t]{.70\textwidth}{ where $b'=e_{t}$ is the interior edge shared by $G_{t}$ and $G_{t+1}$;} } \\
\end{array}\]
and $\calg_{56}$ depends on several cases {and is defined below. We illustrate many of these cases in the figures~\ref{fig(s'<t)} -- \ref{opposite2}. These figures also show geometric realizations of the snake graphs in triangulated surfaces; this geometric construction of snake graphs is explained in section \ref{sect 4}.} Note however that not every self-crossing abstract snake graph has a geometric realization in an unpunctured surface.

\begin{enumerate}
\item If $s'\le t$ (see Figure \ref{fig(s'<t)}) then \[
  \calg_{56} =-\calg_1[1,s-1]\cup\calg_1[s',t]\cup \calg_1[t'+1,d], 
 \]
{where the adjacencies of the subgraphs are induced by $\calg$;}

\item If $s'>t+1$ (see Figures \ref{fig(s'>t)} and \ref{fig(s'>t)2}) then $\calg_{56}$ is defined to be a subgraph of the following graph  
\[\calg_{56}' =  \calg_1[1,s-1]\cup\overline\calg_1[s'-1,t+1]\cup\calg_1[t'+1,d] ,\] {where the first two graphs are glued along the unique boundary edge of $G_{s-1}$ which is north or east and the unique boundary edge of $G_{s'-1}$ which is north or east, whereas the second two graphs are glued along the unique boundary edge of $G_{t+1}$ which is south or west and the unique boundary edge of $G_{t'+1}$ which is south or west.  }
 Let $f_{56} $ be a sign function on $\calg_{56}'$.
 
\begin{enumerate}
\item if $s\ne 1 $ and $t'\ne d$ 
\[ \calg_{56} = \calg_{56}' , \] 

\item  if $s= 1 $ and $t'\ne d$ (see Figure \ref{fig(s'>t)}) 
\[ \calg_{56} = \calg_{56}'\setminus\pred(e),\] \textup{where $e$ is the first edge in $\Int(\calg_{56}')\cup\calg_{56}'^{N\!E}$ such that $f_{56}(e)=f_{56}(e_{s'-1})$.  } 

\item  if $s\ne 1 $ and $t'= d$  
\[ \calg_{56} =  \calg_{56}'\setminus \suc(e') ,\] 
where $e'$ is the last edge in $\Int(\calg_{56}')\cup{}_{SW}\calg_{56}' $ such that $f_{56}(e')=f_{56}(e_{t})$

\item  if $s=1 $ and $t'= d$ (see Figure \ref{fig(s'>t)2})
\[\calg_{56}= 
\calg_{56}'\setminus \pred(e)\setminus \suc(e')
,\]
where $e$ is the first edge in $\Int(\calg_{56}')\cup\calg_{56}'^{N\!E} $ such that $f_{56}(e)=f_{56}(e_{s'-1})$
and $e'$ is the last edge in $\Int(\calg_{56}')\cup{}_{SW}\calg_{56}' $ such that $f_{56}(e')=f_{56}(e_{t})$, if this set is non-empty. Otherwise, let $\calg_{56}=0$.
\end{enumerate}

\item If $s'=t+1$ (see Figure \ref{fig(s'=t+1)}) then 
  \begin{enumerate}
  \item if $s=1$ and $d\ne t'$ then 
%
  \[\calg_{56}= -\calg_1\setminus\pred(e),\, \, \, \parbox[t]{.55\textwidth}{ where $e$ is the first edge in $\Int(\calg_1[t'+1,d])\cup\calgNE_1$ such that $f(e)=-f(e_{t'})$; }
  \] 
  
  \item  if $s=1$ and $d= t'$ then 
    \[\calg_{56} = -\{e_t\};\]
    
  \item if $s\ne1$ and $d\ne t'$ (see Figure \ref{fig(s'=t+1)}) then we need to consider the local overlap in opposite direction $\calg_1[k+1,s-1]$  and $\calg_1[t'+1,t'+s-k-1] $ consisting of tiles preceding  $G_{s}$ and tiles succeeding $G_{t'}$, where  $k$ is given by the maximality condition for overlaps. Thus
$k$  is the least integer such that $k\ge0$, $d\ge t'+s-k-1$,  and $\ocalg_1[{ s-1,k+1}]\cong\calg_1[t'+1,t'+s-k-1].$ Thus there is a snake graph $\calh$ and embeddings $j_1(\calh)= \ocalg_1[{ s-1,k+1}]$ and 
$j_2(\calh)= \calg_1[t'+1,t'+s-k-1]$.
  In the examples in Figure~\ref{fig(s'=t+1)}, we have
$k=1$, $\calh$ consists of a single tile, and $j_1(\calh)$, $j_2(\calh)$ are the two unlabeled tiles in $\calg_1$.
The four cases in the following definition reflect whether $j_1(\calh)$ contains the first tile and $j_2(\calh)$ contains the last tile of $\cala_1$.

Then let
 \[\calg_{56} = \pm \left\{\begin{array}{ll} 
 \calg_1[1,k] \cup \calg_1[t'+s-k,d]
&\parbox[t]{.55\textwidth}{if $k>0$ and $k>s+t'-d-1$, where the two graphs are glued along the unique boundary edge of $G_k$ which is north or east and the unique boundary edge of $G_{t'+s-k}$ which is south or west;}\\ \\
\calg_1\setminus\pred(e') &\parbox[t]{.55\textwidth}{ if $k=0$ and $k>s+t'-d-1$, where $e'$ is the first edge in $ \Int(\calg_1[s+t',d])\cup\calg_1^{NE}$ such that $f(e')=f(e_{s+t'-1})$;
}
\\ \\
 \calg_1\setminus\suc(e)
&\parbox[t]{.55\textwidth}{if $k>0$ and $k=s+t'-d-1$, where $e$ is the last edge in $\Int(\calg_1[1,s+t'-d-1])\cup{}_{SW}\calg_1$ such that $f(e)=f(e_{s+t'-d-1})$;}\\ \\
0&\parbox[t]{\textwidth}{if $k=0=s+t'-d-1$.}
\end{array}\right.\]
where the sign is negative if and only if the local overlap $j_1(\calh)$ and $j_2(\calh)$ is crossing. In the first example in Figure \ref{fig(s'=t+1)} the local overlap $j_1(\calh) $ and $j_2(\calh)$ is crossing, and in the second example it is non-crossing.
  \item if $s\ne1$ and $d= t'$ then 
 \[\calg_{56} ={{-\calg_1\setminus\pred(e) \, \, \, \parbox[t]{.80\textwidth}{ where $e$ is the last edge in $\Int(\calg_1[1,s-1])\cup {}_{SW}\calg_1$ such that $f(e)=-f(e_{s-1}).$}}}  
%
\]

  \end{enumerate}

\end{enumerate}

\begin{defn} \label{resolution self-crossing}
 In the above situation, we say that the element $(\calg_3 \sqcup \calg_4^\circ ) + \calg_{56}\in \mathcal{R}$
   is the {\em resolution of the self-crossing} of $\calg_1$  at the overlap $\calg$ and we denote it by $\res_{\calg} (\calg_1).$ 
\end{defn}

\begin{figure}
\begin{center}
\scalebox{0.8}{ 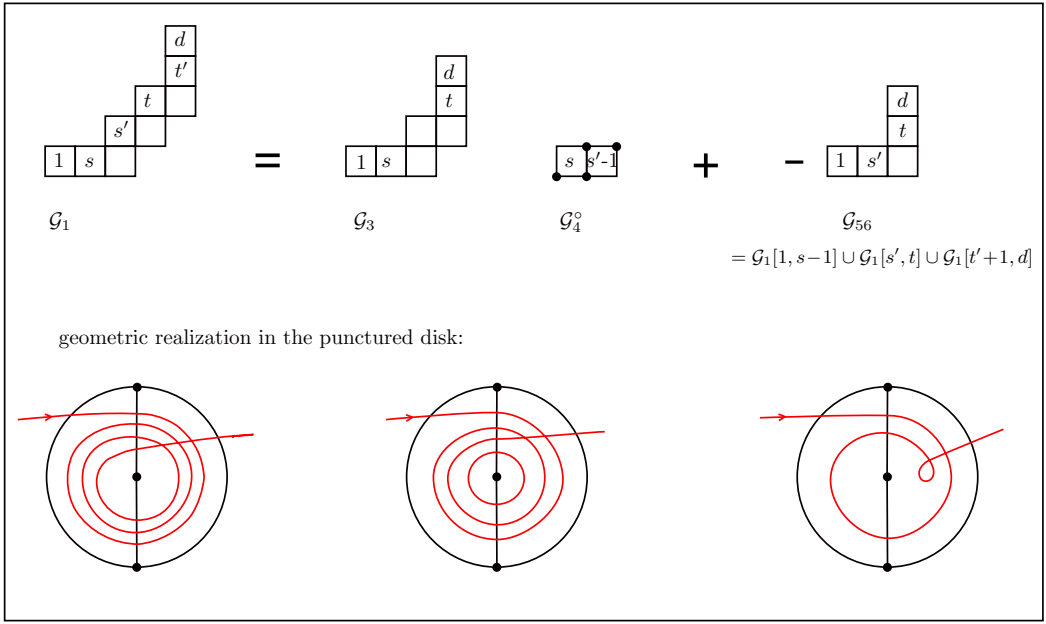} 
\caption{Example of resolution of selfcrossing when $s'\le t$ together with geometric realization on the punctured disk.}
\label{fig(s'<t)}
\end{center}
\end{figure}

\begin{figure}
\begin{center}
\scalebox{0.8}{ 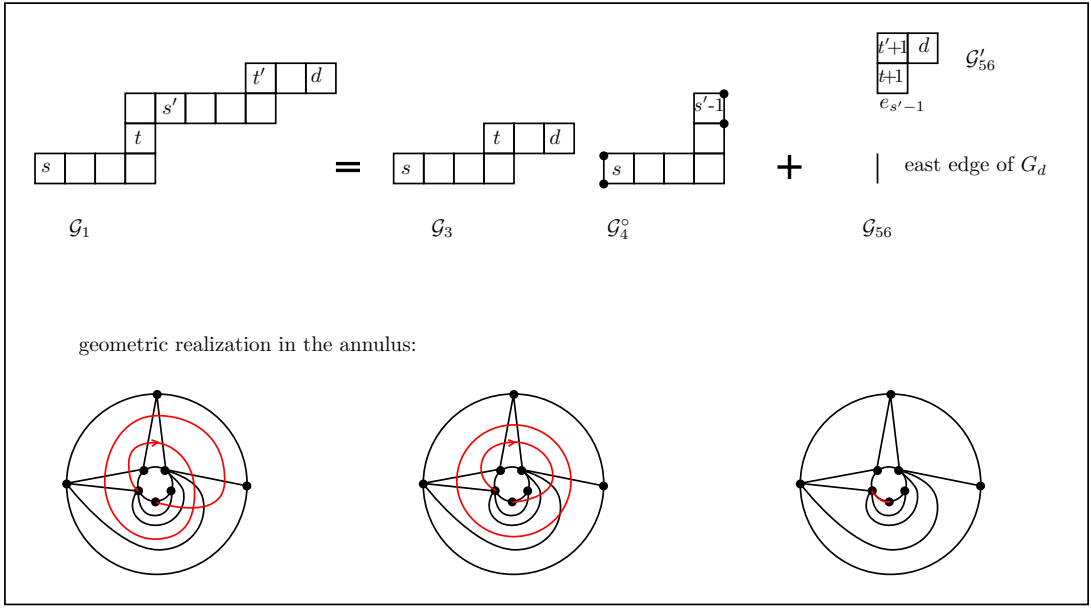}
\caption{Example of resolution of self-crossing when $s'>t+1$ and $s=1$ together with geometric realization on the annulus. Here the snake graph $\calg_{56}$ is   a single edge and the corresponding arc in the surface is a boundary segment.}
\label{fig(s'>t)}
\end{center}
\end{figure}

\begin{figure}
\begin{center}
\scalebox{0.8}{ 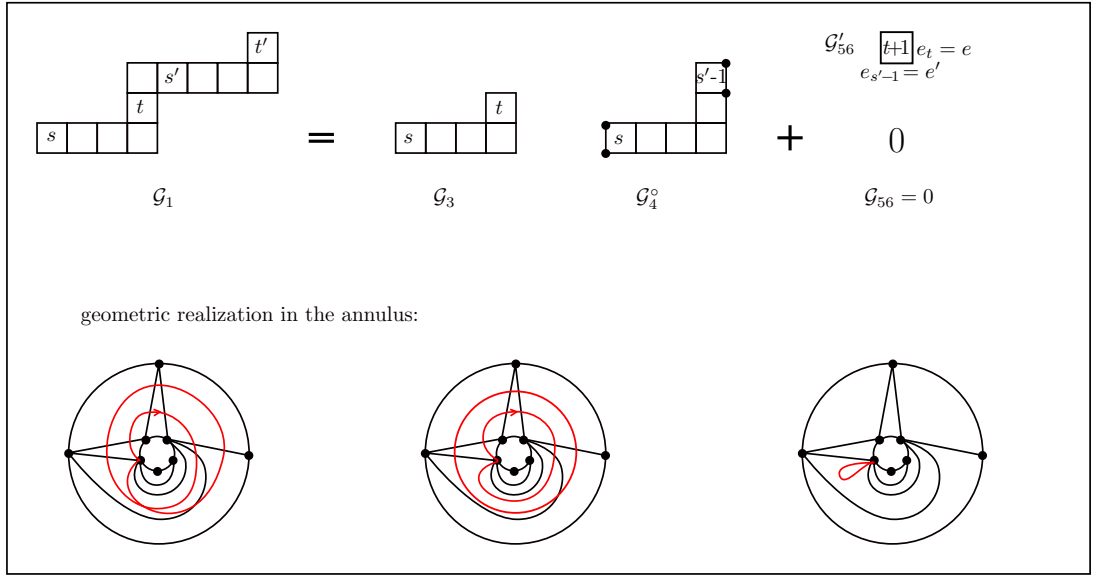}
\caption{Example of resolution of selfcrossing when $s'>t+1$, $s=1$ and $t'=d$, together with geometric realization on the annulus. Here  $\calg_{56}$ is the empty set and the corresponding arc in the surface is a contractible loop at a boundary point.}
\label{fig(s'>t)2}
\end{center}
\end{figure}
%

\begin{figure}
\begin{center}
\scalebox{0.8}{ 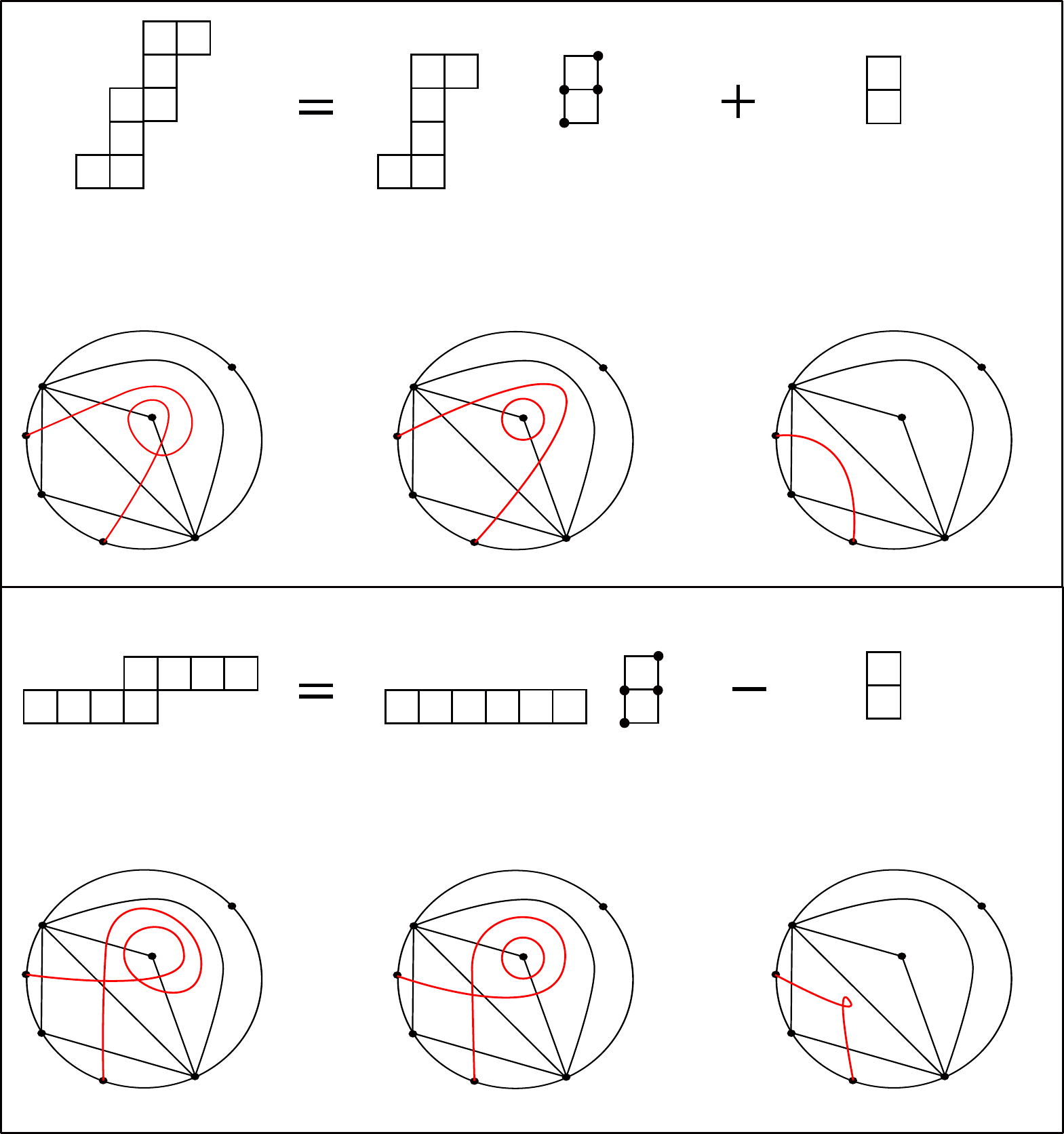} 
\caption{Examples of selfcrossing when $s'=t+1$. }
\label{fig(s'=t+1)}
\end{center}
\end{figure}

%
%
%
\bigskip 

\pagebreak

\emph{Case 2. Overlap in the opposite direction.} 
As before, 
let $\calg_1$ be a self-crossing snake graph with self-overlap
$i_1(\calg)=\calg_1[s,t]\cong \calg_1[s',t']=i_2(\calg).$ We suppose now that the overlap is in the opposite direction, thus the first tile of $\calg$ is mapped to the first tile of $\calg_1[s,t]$ {{under $i_1$}} and to the last tile of $\calg_1[s',t']$ {{under $i_2$}}.

We define the following  connected graphs, see Figure~\ref{opposite} for an example. In the definition of $\calg_5$ we shall use the notation
\[\calg_5'=\calg_1[1,s-1]\cup \calg_1[t'+1,d],\]
{where the two subgraphs are glued along the unique boundary edges of $G_{s-1}^{NE}$ and $_{SW}G_{t'+1}$.}

Let $f_5$ be a sign function on $\calg_{5}'$.
{With this notation, we} define

\[\begin{array}{rcl} \calg_{34}&=&\calg_1[1,t]\cup\ocalg_1[{ s'-1,t+1}]\cup\calg_1[s',d],\\
&&\parbox[t]{.85\textwidth}{{ where the first two subgraphs are glued along the unique boundary edge of $G_t^{NE}$ and the interior edge $e_{s'-1}$, and the last two subgraphs are glued along the interior edge $e_t$ and the unique boundary edge in $_{SW}G_{s'}$; }}
\\ \\
\calg_5 &=& \left\{
	\begin{array}{ll}
	\calg_5' & 
	\parbox[t]{.70\textwidth}{if  $s\ne 1, t'\ne d$; }\\
  \\
 \calg_5'\setminus\pred(e)& 
	\parbox[t]{.70\textwidth}{if  $s= 1, t'\ne d$, where $e\in \Int(\calg_5')\cup \calg_5'^{N\!E}$ is the first edge such that $f_{5}(e)=f_5(e_{t'})$;
	}\\ \\
\calg_5'\setminus\suc(e) & 
	\parbox[t]{.70\textwidth}{if  $s\ne 1, t'= d$, where $e \in\Int(\calg_5')\cup {}_{SE}\calg_5'$ is the last edge such that $f_{5}(e)=f_5(e_{s-1})$;
	}\\
\\
0& 
	\parbox[t]{.70\textwidth}{if  $s= 1, t'= d$.}\\

	\end{array}
\right.
\\ \\
\calg_6^\circ&=& \calg_1^b[t+1,s'-1] \ 
\parbox[t]{.70 \textwidth}{where $b$ is the unique boundary edge in {$_{SW}G_{t+1}$} and $b'$ is the unique boundary edge in {$G_{s'-1}^{NE}$}.}
\end{array}\]

\begin{defn} \label{resolution self-crossing in opposite direction}
 In the above situation, we say that the element $\calg_{34}  + (\calg_{5}\sqcup \calg_6^\circ )\in \mathcal{R}$
   is the {\em resolution of the self-crossing} of $\calg_1$  at the overlap $\calg$ and we denote it by $\res_{\calg} (\calg_1).$ 
\end{defn}

\begin{remark}
 \label{indices}{
 For the two possible directions of  overlap, our choice of notation for indices is consistent in the sense that indices $3,4$ always refer to the part of the resolution that contains the overlaps, and the indices 5,6 always refer to the part that does not.}
\end{remark}

\begin{figure}
\begin{center}
\scalebox{0.8}{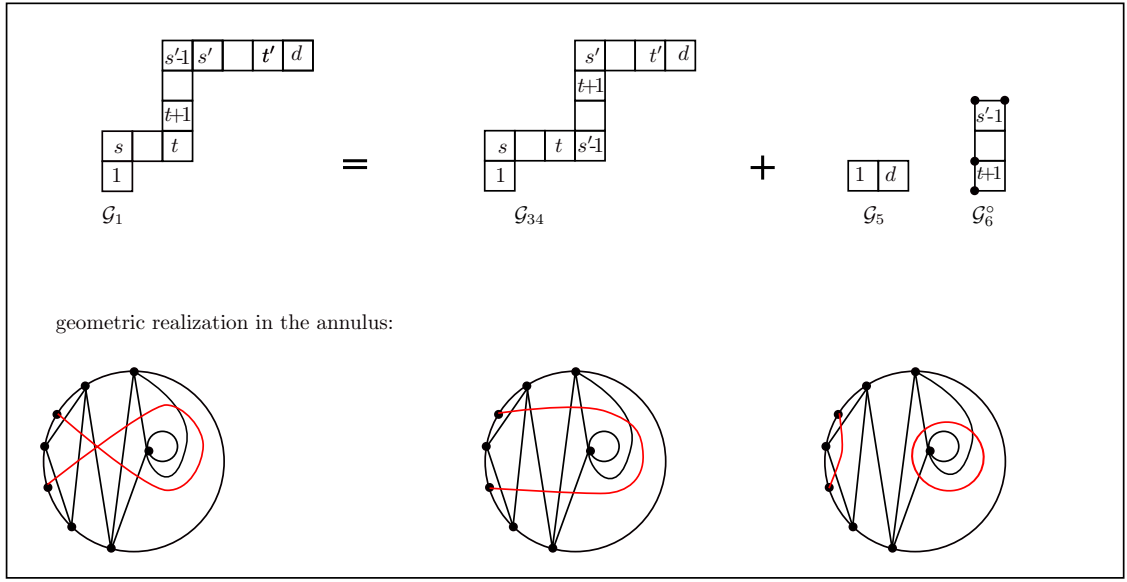}
\caption{Resolution of a self-crossing in opposite direction.}
\label{opposite}
\end{center}
\end{figure}

\begin{figure}
\begin{center}
\scalebox{0.8}{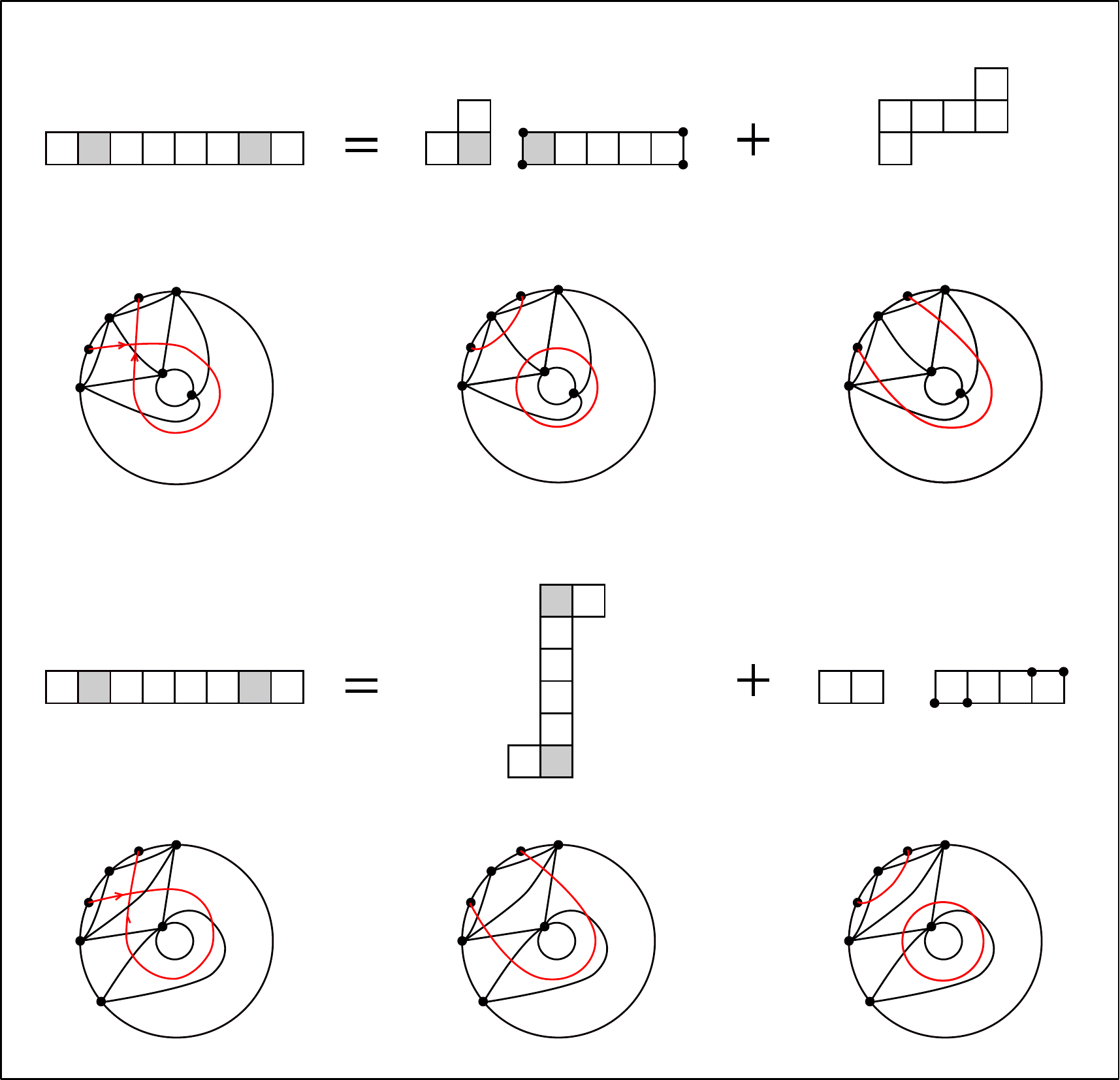}
\caption{Example of a snake graph with  self-crossing in a single tile which can be considered as in either direction.}
\label{opposite2}
\end{center}
\end{figure}


%
\subsection{Grafting} 
In this subsection, we define another operation which to two snake graphs associates two pairs of snake graphs. This operation does not involve the notion of overlaps.

Let $\calg_1=(\gi 12d),$ $\calg_2=(\gii 12{d'})$ be two snake graphs and let $f_1$ be a sign function on $\calg_1.$

\emph{Case 1.}
Let $s$ be such that $1\le s<d.$

If $G_{s+1}$ is north of $G_s$ in $\calg_1$ then let $\e_3$ denote the east edge of $G_s,$  $\e_5$ the west edge of $G_{s+1}$ and $\e'_3$ the west edge of $G'_1,$ $\e'_5$ the south edge of $G'_1.$

If $G_{s+1}$ is east of $G_s$ in $\calg_1,$ then let 
$\e_3 $  denote the north edge of $G_s $,
$\e_5$   the south edge of $G_{s+1}$ and
$\e_{3}'$   the south edge of  $ G'_1 $, 
$\e_{5}' $   the west edge of $ G'_1  $. 
Thus we have one of the following two situations.
\begin{center}
\scalebox{0.8}{
\begingroup%
  \makeatletter%
  \providecommand\color[2][]{%
    \errmessage{(Inkscape) Color is used for the text in Inkscape, but the package 'color.sty' is not loaded}%
    \renewcommand\color[2][]{}%
  }%
  \providecommand\transparent[1]{%
    \errmessage{(Inkscape) Transparency is used (non-zero) for the text in Inkscape, but the package 'transparent.sty' is not loaded}%
    \renewcommand\transparent[1]{}%
  }%
  \providecommand\rotatebox[2]{#2}%
  \ifx\svgwidth\undefined%
    \setlength{\unitlength}{312.66984863bp}%
    \ifx\svgscale\undefined%
      \relax%
    \else%
      \setlength{\unitlength}{\unitlength * \real{\svgscale}}%
    \fi%
  \else%
    \setlength{\unitlength}{\svgwidth}%
  \fi%
  \global\let\svgwidth\undefined%
  \global\let\svgscale\undefined%
  \makeatother%
  \begin{picture}(1,0.22787615)%
    \put(0,0){\includegraphics[width=\unitlength]{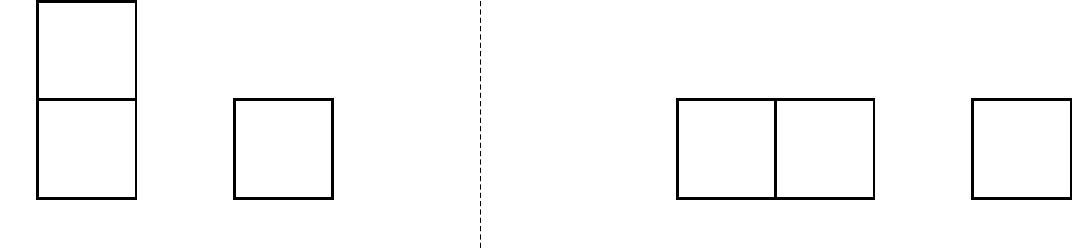}}%
    \put(0.92214271,0.0802294){\color[rgb]{0,0,0}\makebox(0,0)[lb]{\smash{$G_1'$}}}%
    \put(0.92579783,0.01434845){\color[rgb]{0,0,0}\makebox(0,0)[lb]{\smash{$\e_3'$}}}%
    \put(0.86247545,0.0802294){\color[rgb]{0,0,0}\makebox(0,0)[lb]{\smash{$\e_5'$}}}%
    \put(0.7268886,0.0802294){\color[rgb]{0,0,0}\makebox(0,0)[lb]{\smash{$G_{s+1}$}}}%
    \put(0.74740947,0.01458317){\color[rgb]{0,0,0}\makebox(0,0)[lb]{\smash{$\e_5$}}}%
    \put(0.64924712,0.0802294){\color[rgb]{0,0,0}\makebox(0,0)[lb]{\smash{$G_s$}}}%
    \put(0.65436434,0.14611037){\color[rgb]{0,0,0}\makebox(0,0)[lb]{\smash{$\e_3$}}}%
    \put(0.06168075,0.08105505){\color[rgb]{0,0,0}\makebox(0,0)[lb]{\smash{$G_s$}}}%
    \put(0.04885712,0.17160831){\color[rgb]{0,0,0}\makebox(0,0)[lb]{\smash{$G_{s+1}$}}}%
    \put(0.13167701,0.08105505){\color[rgb]{0,0,0}\makebox(0,0)[lb]{\smash{$\e_3$}}}%
    \put(-0.00006996,0.17084109){\color[rgb]{0,0,0}\makebox(0,0)[lb]{\smash{$\e_5$}}}%
    \put(0.2484092,0.01443661){\color[rgb]{0,0,0}\makebox(0,0)[lb]{\smash{$\e_5'$}}}%
    \put(0.18197352,0.08105505){\color[rgb]{0,0,0}\makebox(0,0)[lb]{\smash{$\e_3'$}}}%
    \put(0.24228363,0.08105505){\color[rgb]{0,0,0}\makebox(0,0)[lb]{\smash{$G_1'$}}}%
  \end{picture}%
\endgroup%
}
\end{center}
Define four snake graphs as follows; see Figure \ref{figrafting} for examples.
\begin{align*}
 \calg_3=&\,\calg_1[1,s] \cup \calg_2 , \mbox{ where the  two subgraphs are glued along the edges $\e_3$ and $\e'_3$}; \\
 \calg_4 =&\,\calg_1 \setminus \pred(e),\   \parbox[t]{.65\textwidth}{where $e\in\Int(\calg_1[s\!+\!1,d])\cup\calgNE_1$  is the first edge such that $f_1(e)=f_1(\delta_3)$;}
\\
 \calg_5 =&\,\calg_1 \setminus \suc(e) ,\  \parbox[t]{.75\textwidth}{where $e\in\calgSW_1\cup\,\Int(\calg_1[1,s])$  is the last edge such that $f_1(e)=f_1(\delta_3)$;}
\\
 \calg_6=&\,\ocalg_2[d',1] \cup \calg_1 [s+1,d] ,\ \mbox{where the two subgraphs are glued along the edges $\e_5$ and $\e'_5$.} &
 \\
\end{align*}

\emph{Case 2.} Now let $s=d.$ Choose a pair of edges $(\e_3,\e'_3)$ such that either $\e_3$ is the north edge in $G_s$ and $\e'_3$ is the south edge in $G'_1$ or $\e_3$ is the east edge in $G_s$ and $\e'_3$ is the west edge in $G'_1.$ Let $f_2$ be  a sign function on $ \calg_2$ such that $f_2(\e'_3)=f_1(\e_3).$ 
 Then define four snake graphs as follows.

\begin{align*}
 \calg_3=&\calg_1[1,s] \cup \calg_2 \mbox{ where the  two subgraphs are glued along the edges $\e_3$ and $\e'_3$.} &\\
\calg_4=& \{ \e_3 \}
\\
 \calg_5 =&\,\calg_1 \setminus \suc(e) ,\  \parbox[t]{.75\textwidth}{where $e\in\calgSW_1\cup\,\Int(\calg_1[1,s])$  is the last edge such that $f_1(e)=f_1(\delta_3)$;}
\\
 \calg_6=&\,\calg_2\setminus \pred(e),\ \mbox{where $e\in \Int(\calg_2)\cup \calgNE_2$ is the first edge such that $f_2(e)=f_2(\delta_3')$}&
\end{align*}

\begin{defn} \label {grafting}
 In the above situation, we say that the element $(\calg_3 \sqcup \calg_4) + (\calg_5 \sqcup \calg_6) \in\mathcal{R}$ is the {\em resolution of the grafting of $\calg_2$ on $\calg_1$ in $G_s$ }and we denote it by $\graft_{s,\e_3} (\calg_1,\calg_2).$ 
\end{defn}

\begin{figure}
\begin{center}
\scalebox{0.8}{ 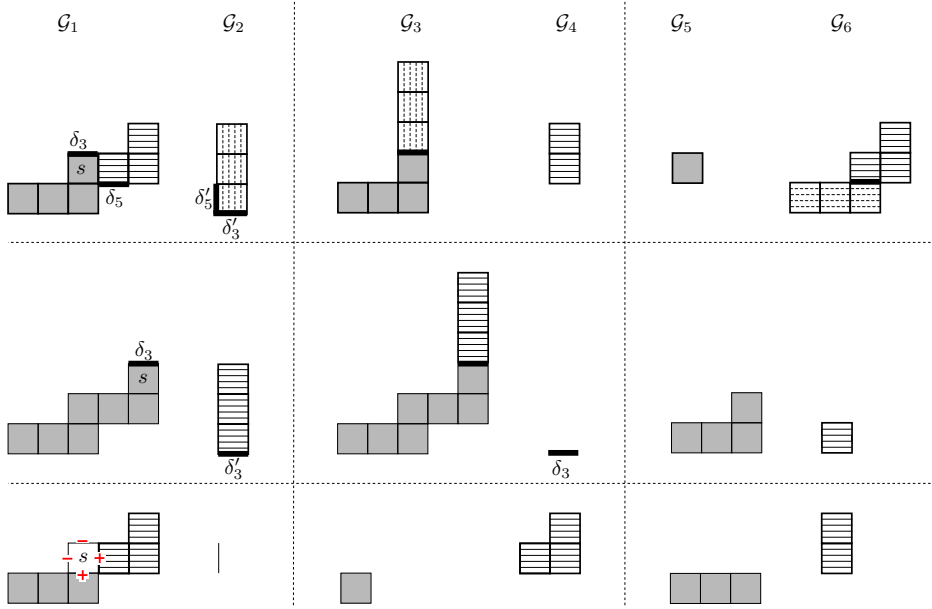}
\caption{Examples of resolutions of graftings: $1<s<d$ in the first row;  $s=d'$ in the second row {and grafting with a single edge in the third row.}}
\label{figrafting}
\end{center}
\end{figure}

\emph{Case 3.} Grafting with a single edge. In this case let $\calg_1=(G_1,\ldots,G_d)$ and let the snake graph $\calg_2$ consist of a single edge. We are grafting this edge at a position $s$, where $1\le s\le d$, as follows.   We define 

\begin{align*}
 \calg_3=&\,\calg_1\setminus\suc(e), \textup{where $e\in\Int(\calg_1[1,s])\cup\calgSW_1$ is the last edge such that $f(e)=-$}; \\
 \calg_4 =&\,\calg_1 \setminus \pred(e),\   \textup{where $e\in\Int(\calg_1[s,d])\cup\calgNE_1$  is the first edge such that $f(e)=+$;}
\\
 \calg_5=&\,\calg_1\setminus\suc(e), \textup{where $e\in\Int(\calg_1[1,s])\cup\calgSW_1$ is the last edge such that $f(e)=+$}; \\
 \calg_6 =&\,\calg_1 \setminus \pred(e),\   \textup{where $e\in\Int(\calg_1[s,d])\cup\calgNE_1$  is the first edge such that $f(e)=-$.}
 \\
\end{align*}
In particular, if $\calg_1$ consists of a single tile, then $s=d=1$ and one of $\calg_3 \cup \calg_4 $, $\calg_5 \cup \calg_6 $ is the east and the west edge of $\calg_1$ and the other is the north and the south edge of $\calg_1$.

\subsection{Self-grafting}\label{self-grafting} In this subsection, we define the grafting operation on a single snake graph. 
Let $\calg_1=(\gi 12d),$ be a snake graph and let $f$ be a sign function on $\calg_1.$ We will define the resolution of self-grafting on $\calg_1$ in $G_s$.

\emph{Case 1.}
Let $s$ be such that $1\le s<d$. 
If $G_{s+1}$ is north of $G_s$ in $\calg_1$ then let $\e_3$ denote the east edge of $G_s,$  $\e_5$ the west edge of $G_{s+1}$ and if $G_{s+1}$ is east of $G_s$ in $\calg_1,$ then let 
$\e_3 $  denote the north edge of $G_s $,
$\e_5$   the south edge of $G_{s+1}$.
In either case we have $f(\e_3)=f(\e_5)$.
 Let $\e'_3,\e'_5 \in\calgSW_1$ be such that  $f(\e'_3)=f(\e_3)=f(\e_5)$ and $f(\e'_5)=-f(\e_3)=-f(\e_5)$.

Thus we have one of the following two situations.
\begin{center}
\scalebox{0.8}{
\begingroup%
  \makeatletter%
  \providecommand\color[2][]{%
    \errmessage{(Inkscape) Color is used for the text in Inkscape, but the package 'color.sty' is not loaded}%
    \renewcommand\color[2][]{}%
  }%
  \providecommand\transparent[1]{%
    \errmessage{(Inkscape) Transparency is used (non-zero) for the text in Inkscape, but the package 'transparent.sty' is not loaded}%
    \renewcommand\transparent[1]{}%
  }%
  \providecommand\rotatebox[2]{#2}%
  \ifx\svgwidth\undefined%
    \setlength{\unitlength}{310.87563324bp}%
    \ifx\svgscale\undefined%
      \relax%
    \else%
      \setlength{\unitlength}{\unitlength * \real{\svgscale}}%
    \fi%
  \else%
    \setlength{\unitlength}{\svgwidth}%
  \fi%
  \global\let\svgwidth\undefined%
  \global\let\svgscale\undefined%
  \makeatother%
  \begin{picture}(1,0.36562104)%
    \put(0,0){\includegraphics[width=\unitlength]{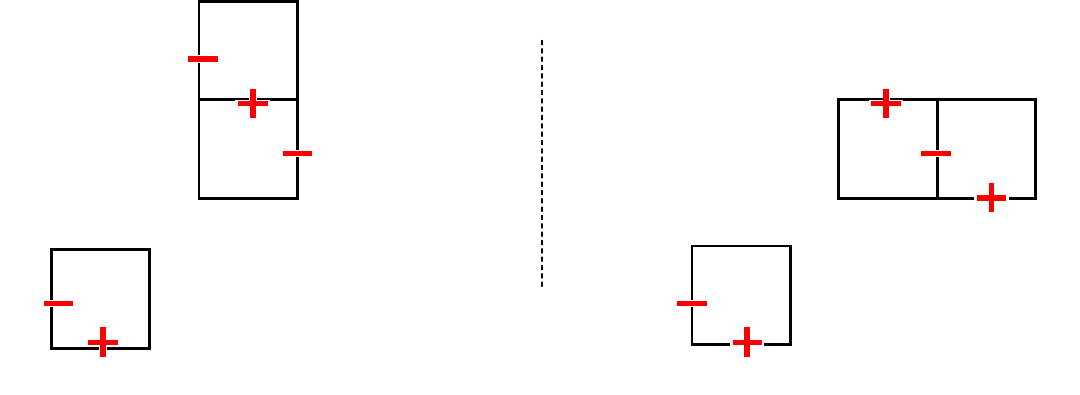}}%
    \put(0.66758406,0.08170372){\color[rgb]{0,0,0}\makebox(0,0)[lb]{\smash{$G_1$}}}%
    \put(0.67640703,0.01029578){\color[rgb]{0,0,0}\makebox(0,0)[lb]{\smash{$\e_3'$}}}%
    \put(0.58698547,0.08170372){\color[rgb]{0,0,0}\makebox(0,0)[lb]{\smash{$\e_5'$}}}%
    \put(0.88068732,0.21712215){\color[rgb]{0,0,0}\makebox(0,0)[lb]{\smash{$G_{s+1}$}}}%
    \put(0.90132663,0.14595029){\color[rgb]{0,0,0}\makebox(0,0)[lb]{\smash{$\e_5$}}}%
    \put(0.80259774,0.21712215){\color[rgb]{0,0,0}\makebox(0,0)[lb]{\smash{$G_s$}}}%
    \put(0.80774449,0.2885301){\color[rgb]{0,0,0}\makebox(0,0)[lb]{\smash{$\e_3$}}}%
    \put(0.21164023,0.21795256){\color[rgb]{0,0,0}\makebox(0,0)[lb]{\smash{$G_s$}}}%
    \put(0.19874259,0.3141752){\color[rgb]{0,0,0}\makebox(0,0)[lb]{\smash{$G_{s+1}$}}}%
    \put(0.29233397,0.21795256){\color[rgb]{0,0,0}\makebox(0,0)[lb]{\smash{$\e_3$}}}%
    \put(0.13923962,0.3082568){\color[rgb]{0,0,0}\makebox(0,0)[lb]{\smash{$\e_5$}}}%
    \put(0.08034772,0.00684056){\color[rgb]{0,0,0}\makebox(0,0)[lb]{\smash{$\e_5'$}}}%
    \put(-0.00007037,0.07648542){\color[rgb]{0,0,0}\makebox(0,0)[lb]{\smash{$\e_3'$}}}%
    \put(0.07418681,0.07899024){\color[rgb]{0,0,0}\makebox(0,0)[lb]{\smash{$G_1$}}}%
  \end{picture}%
\endgroup%
}
\end{center}

Define two snake graphs and a band graph as follows, {see Figure~\ref{selfgrafting} for an example}. 
\begin{align*}
 \calg_3=&\,\calg_1\setminus \pred(e),   \parbox[t]{\textwidth} { where $e\in\Int(\calg_1[s\!+\!1,d])\cup\calgNE_1$  is the first edge such that $f(e)=f(\delta_3)$;} \\
 \calg_4 =&\,(\calg_1[1,s])^{\e_3},  \parbox[t]{.65\textwidth}{ the band graph obtained by identifying $\e_3$ and $\e'_3$;}
\\
\end{align*}
Let $\calg_{56}'=\ocalg_1[{ s},1]\cup\calg_1[s+1,d]$ glued along $\e_5$ and $\e'_5$, and let $f_{56}$ be a sign function on $\calg_{56}'$. Then define
\begin{align*}
 \calg_{56} =&\,\calg_{56}' \setminus \pred(e) ,\  \parbox[t]{.75\textwidth}{where $e\in\,\Int(\calg_{56}')\cup{\calg'_{56}}^{N\!E}$  is the first edge such that $f_{56}(e)=f_{56}(\delta_3)$.}
\\
\end{align*}

\begin{figure}
\begin{center}
\scalebox{0.7}{\Large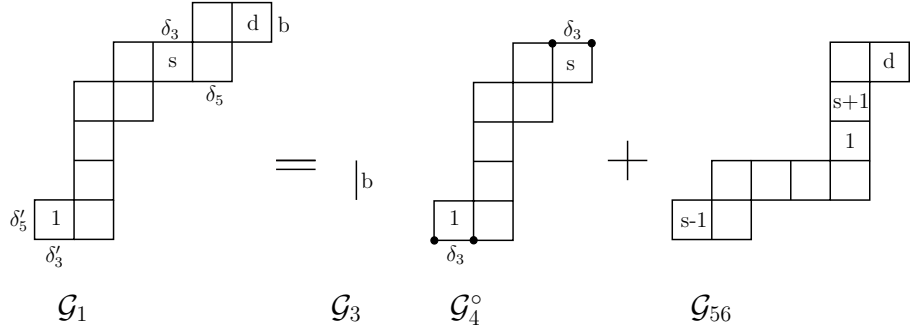}
\caption{ An example of self-grafting with $s<d$.}
\label{selfgrafting}
\end{center}
\end{figure}

 \emph{Case 2.}  Now let $s=d$ (see Figure \ref{kroneckerBand}). Choose an edge $\e_3\in \calgNE_1$, and let $\e'_3\in\calgSW_1$ be the unique edge such that $f(\e_3)=f(\e_3')$.
 Then define 
\[ \begin{array}{rcllllll}
 \calg_3&=&\,\{\e_3\} ; \\
\calg_4^\circ&=&\,\calg_1^{\e_3} \quad \parbox[t]{.65\textwidth}{ the band graph obtained by identifying $\e_3$ and $\e'_3$;}
\\
\\
\calg_{56}&=&\left\{\begin{array}{ll}\calg_1\setminus\suc(e)\setminus\pred(e') & \parbox[t]{.60\textwidth}{where $e$ is the last edge in $\Int(\calg_1)$ such that $f(e)=f(\e_3)$ and $e'$ is the first edge  in $\Int(\calg_1\setminus\suc(e))$  such that $f(e')=f(\e_3)$, if such an $e'$ exists;} \\
  0 & \mbox {otherwise}.
  \end{array}\right.
\end{array}
\]

  {\begin{remark}
 This last case is particularly important in actual computations, since it computes the difference between a  band graph and any of its underlying snake graphs.
  \end{remark}}


\begin{defn} \label{resolution self-grafting}
 In the above situation, we say that the element $(\calg_3 \sqcup \calg_4^\circ ) + \calg_{56}\in \mathcal{R}$
   is the {\em resolution of the self-grafting} of $\calg_1$  in $G_s$  and we denote it by $\graft_{s,\e_3} (\calg_1).$ 
\end{defn}

\begin{figure}
\begin{center}
\scalebox{1}{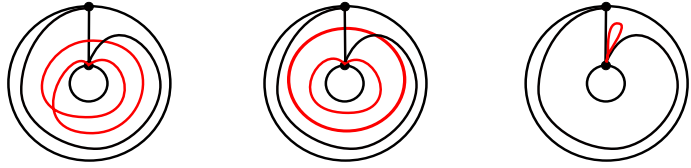}
\caption{An example of self-grafting with $s=d$ and $\calg_{56}=0$.}
\label{kroneckerBand}
\end{center}
\end{figure}

\section{Perfect Matchings}\label{sect pm}\label{sect 3}
{In this section, we show that there is a bijection between the set of perfect matchings of (self-)crossing snake graphs and the set of perfect matchings of the resolution. In sections~\ref{sect 5} and \ref{sect 7}, we will show that for labeled snake graphs coming from an unpunctured surface, this bijection is weight preserving and induces an identity in the corresponding cluster algebra.
}

Recall that a {\em perfect matching} $ P$ of a graph $G$ is a subset of the set of edges of $G$ such that each vertex of $G$ is incident to exactly one edge in $ P.$  

For band graphs, we need the notion of \emph{good} perfect matchings, which was introduced in \cite[Definition 3.8]{MSW2}. 
Let $\calg$ be a snake graph and let  $\calg^b$ be the band graph obtained from $\calg $ by glueing along the edge $b$ as defined in section~\ref{sect band}. If $P$ is a perfect matching of $\calg$ containing the edge $b$,  then $P\setminus\{b\}$ is a perfect matching of $\calg^b$.
In the following definition of good perfect matchings for arbitrary band graphs, we start with the band graph and cut it at an interior edge.
\begin{defn}
 Let  $\calg^\circ$ be a  band graph. A perfect matching $P$ of $\calg^\circ$ is called a \emph{good perfect matching} if there exists an interior edge $e$ in $\calg^\circ$ such that $P\sqcup\{e\}$ is a perfect matching of the snake graph  
$(\calg^\circ)_e$ obtained by cutting $\calg^\circ$ along $e$.

\end{defn}

\begin{defn}\ 
\begin{enumerate}
\item If $\calg$ is a snake graph,  let $\match \calg$ denote the set of all perfect matchings of  $\calg$.
\item If $\calg^\circ$ is  a band graph,  let $\match \calg^\circ$ denote the set of all good perfect matchings of  $\calg^\circ$. 
\item If $R=(\calg_1\sqcup\calg_2 + \calg_3\sqcup\calg_4)\in \mathcal{R}$, we let \[\match R=\match \calg_1\times\match\calg_2 \cup \match\calg_3\times\match \calg_4.\]
\end{enumerate}
\end{defn}

{The following Lemma will be useful later on.
\begin{lem}
 \label{lemNE}
 Let $\calg$ be a snake graph with sign function $f$ and let $P$ be a matching of $ \calg$ which consists of boundary edges only. Let NE be the set of all north and all east edges of the boundary of $\calg$ and let SW be the set of all south and west edges of the boundary. Then
 \begin{enumerate}
\item $f(a)=f(b)$ if $a$ and $b$ are both in $P \,\cap\,$NE or both in $P\,\cap\,$SW.

\item $f(a)=-f(b)$ if one of $a,b$ is in  $P\,\cap\,$NE and the other  in $P\,\cap\,$SW.
\item If $a\in P\,\cap\,$NE, or if  $a\in $SW but $a\notin P$, then 
\[P= \{b\in \textup{NE}\mid f(b)=f(a)\}\cup \{b\in \textup{SW}\mid f(b)=-f(a)\}.\]

\item If $a\in P\,\cap\,$SW, or if  $a\in $NE but $a\notin P$, then 
\[P= \{b\in \textup{NE}\mid f(b)=-f(a)\}\cup \{b\in \textup{SW}\mid f(b)=f(a)\}.\]
\end{enumerate}
\end{lem}
\begin{proof} Statements 
 (1) and (2) are Lemma 7.2 in \cite{CS}. 
The statements (3) and (4) clearly hold for straight snake graphs. 
For general snake graphs it follows from the following observation. Suppose that  three consecutive tiles $G_{i-1},G_i,G_{i+1}$ form a zigzag, and suppose without loss of generality that $G_i$ is east of $G_{i-1}$ and $G_{i+1}$ is north of $G_i$. Then the north edge $n$ of $G_{i-1}$, the south edge $s$ and the east edge $e$ of  $G_i$, and the west edge $w$ of $G_{i+1}$ are boundary edges. Moreover all four edges have the same sign $f(n)=f(s)=f(e)=f(w)$. 
Since $P$ consists of boundary edges only there are two possibilities, either  $n,e\in P$  and $s,w\notin P$ or $n,e\notin P$  and $s,w\in P$. 
\end{proof}
}

\smallskip

In \cite{CS}, we constructed bijections between the set of perfect matchings of two crossing snake graphs and the set of perfect matchings of the resolution of the crossing. One of the main results of \cite{CS} is the following.

\begin{thm}\cite[Theorem 3.1]{CS} \label {bijections} Let $\calg_1, \calg_2$ be two snake graphs. Then there are bijections
\begin{itemize}
 \item[(1)] $\match (\calg_1 \sqcup \calg_2) \longrightarrow  \match( \res_{\calg} (\calg_1 \sqcup \calg_2))$;
 \item[(2)] $\match (\calg_1 \sqcup \calg_2) \longrightarrow    \match( \graft_{s, e_3} (\calg_1 \sqcup \calg_2))$.
\end{itemize}
\end{thm}

We now extend this construction to give bijections between sets of perfect matchings of a self-crossing snake graph and of the resolution of the  self-crossing. 

\subsection{Switching operation}\label{switching}
The following construction was introduced in \cite{CS}. Let $\calg_1,\calg_2$ be two snake graphs  with a crossing local overlap $\calg$ and embeddings $i_1:\calg\to\calg_1$ and $i_2:\calg\to \calg_2$. Let $\calg_3,\calg_4$ be the pair of snake graphs in the resolution of the crossing which contain the overlap. Thus $\calg_3$ contains the initial part of $\calg_1$, the overlap, and the terminal part of $\calg_2$, whereas $\calg_4$
contains the initial part of $\calg_2$, the overlap, and the terminal part of $\calg_1$.

Given two perfect matchings $P_1\in\match\calg_1$ and $P_2\in \match\calg_2$, there may or may not exist a \emph{switching position}, that is, a position in the overlap such that using the edges of $P_1$ up to that position and the edges of $P_2$ after that position yield a perfect matching on $\calg_3$.  In this case, using the edges of $P_2$ up to the switching position and those of $P_1$ after the switching position {also} yields a matching of $\calg_4$. {The possible local configurations of the perfect matchings at the switching position are explicitly listed in Figures 6-11 in \cite[section 3]{CS}.}

If a switching position exists, we define the \emph{switching operation} to be the map $(P_1,P_2)\mapsto (P_3,P_4)$, 
where $P_3$, respectively $P_4$, is the matching of $\calg_3$, respectively $\calg_4$, obtained by applying the method above at the first switching position. In the same way, we define the switching operation sending matchings of $ \calg_3\sqcup \calg_4$ to matchings of $\calg_1\sqcup\calg_2$. {
If no switching position exists, then the restrictions $(P_1\cup P_2)|_{\calg_5}$  and $(P_1\cup P_2)|_{\calg_6}$ are perfect matchings of $\calg_5$ and $\calg_6$, respectively.}

The switching operation generalizes in a straightforward
way when we consider a single snake graph $\calg_1$ with a crossing self-overlap instead of the pair $(\calg_1,\calg_2)$, as long as the self-overlap does not have an intersection.

By Definition \ref{resolution self-crossing}, $\calg_3$ and $ \calg_4^\circ$ are still two distinct graphs if the overlaps have the same orientation, and if the orientations are opposite, then, by Definition~\ref{resolution self-crossing in opposite direction}, we obtain a single snake graph $\calg_{34}$.

Throughout the rest of this section let $\calg_1=(G_1,\ldots,G_d)$ be a snake graph with self-crossing in the overlap $i_1(\calg)=\calg_1[s,t]\cong\calg_1[s',t']=i_2(\calg)$, with $s<s'$, and let $f$ be a sign function on $\calg_1$.
\subsection{Self-crossing case 1: Overlap in the same direction and {$\bf{s'>t+1}$.}}\label{sect 3.3}
Let $P_1\in\match\calg_1$. If $P_1$ has a switching position in $i_1(\calg)$ and $i_2(\calg)$, we define $P_3\in\match\calg_3$ by 
\[\begin{array}{lll}P_3=P_1& \textup{on $\calg_1[1,s-1]\cup \calg_1[t'+1,d]$,}\\  
P_3=\left\{\begin{array}{ll}
P_1|_{i_1(\calg)} &\textup{before the first switching position;}\\
P_1|_{i_2(\calg)} &\textup{after the first switching position;}\\\end{array}\right\}& \textup{on $\calg$.}\end{array}\]
And we define $P_4\in\match\calg_4^\circ$ by
\[\begin{array}{lll}P_4=P_1& \textup{on $\calg_1[t+1,s'-1]$,}\\  
P_4=\left\{\begin{array}{ll}
P_1|_{i_2(\calg)} &\textup{before the first switching position;}\\
P_1|_{i_1(\calg)} &\textup{after the first switching position;}\\\end{array}\right\}& \textup{on $\calg$.}\end{array}\]
{With this notation,} we define 
\[\varphi(P_1)=\left\{\begin{array}{ll} (P_3,P_4) &\textup{if $P_1$ has a switching position in $i_1(\calg)$ and $i_2(\calg)$;} \\
P_{56}=P_1|_{\calg_{56}};
&\textup{if $P_1$ has no switching position in $i_1(\calg)$ and $i_2(\calg)$.}\end{array}\right.
\]

To show that $\varphi$ is a bijection, we define its inverse function $\psi$. Let $i_3:\calg\to \calg_3$ and $i_4:\calg\to \calg_4^\circ$ be the embeddings of the overlap, {and} let
$(P_3,P_4)\in \match\calg_3\times\calg_4^\circ$. It  follows from \cite[Lemma~7.3]{CS}
that $(P_3,P_4)$ has a switching position in $i_3(\calg)$ and $i_4(\calg)$.
We define $\psi(P_3,P_4)=P_1$ to be 
\[\begin{array}{lll}
P_1=P_3 &\textup{on $\calg_1[1,s-1]\cup\calg_1[t'+1,d]$} \\
P_1=P_4&\textup{on $\calg_1[t+1,s'-1]$}\\
P_1=\left\{\begin{array}{ll}
P_3|_{i_3(\calg)} &\textup{before the first switching position;}\\
P_4|_{i_4(\calg)} &\textup{after the first switching position;}\\\end{array}\right\}& \textup{on $i_1(\calg)$;}\\
P_1=\left\{\begin{array}{ll}
P_4|_{i_4(\calg)} &\textup{before the first switching position;}\\
P_3|_{i_3(\calg)} &\textup{after the first switching position;}\\\end{array}\right\}& \textup{on $i_2(\calg)$.}\end{array}\]

Now let $P_{56}\in\match\calg_{56}$. Define $\psi(P_{56})\in \match\calg_1$ {as follows. If $\calg_{56}\ne\calg_{56}'$ then first complete $P_{56}$ to a matching $P_{56}'$ of $\calg_{56}'$ using only boundary edges of $\calg_{56}'$ and then complete $P_{56}'$ to a matching of $\calg_1$ using only boundary edges of $\calg_1$.} 
 This unique completion does not depend on $P_{56}$ and it is determined by the fact that it does not contain the edges $e$ and $e'$, where $e$ is the unique edge in $_{SW}G_s$ such that $f(e)=f(e_{s'-1})$ and 
$e'$ is the unique edge in $G_{t'}^{N\!E}$ such that $f(e')=f(e_{t})$. 
{
Indeed, if $s\ne 1$, then $_{SW} G_s=\{e,e_{s-1}\}$ and the endpoints of $e_{s-1}$ are matched in $\psi(P_{56})$ by edges of $G_{s-1}$, so $e\notin \psi(P_{56})$. If $s=1$, then the endpoints of $e_t$ are matched in $\psi(P_{56})$ by edges of $G_{t+1}$, because $G_{t+1}\subset\calg_{56}'$. Since $e_t$ is in $G_t^{N\!E}$, it follows from Lemma~\ref{lemNE} that all north and east edges of $\calg_1[s,t]\cap \psi(P_{56})$ have the same sign as $e_t$ and all south and west edges have the opposite sign. Since $e\in {}_{SW}G_s$ and $f(e)=f(e_{s'-1})=f(e_t)$, it follows that $e\notin\psi(P_{56})$. The proof for $e'$ is similar.
}

It is clear that if $P_1$ has a switching position in $i_1(\calg)$ and $i_2(\calg)$ then the resulting pair $\varphi(P_1)=(P_3,P_4)$ has the same switching position in $i_3(\calg)$ and $i_4(\calg)$, and the resulting matching $\psi(P_3,P_4)=P_1$. Similarly, if $(P_3,P_4)$ has a switching position in $i_3(\calg)$ and $i_4(\calg)$
then the resulting matching $\psi(P_3,P_4)=P_1$ has the same switching position in $i_1(\calg)$ and $i_2(\calg)$ and the resulting pair is $\varphi(P_1)=(P_3,P_4)$.
Since we are always using the first switching position, it follows that $\varphi$ and $\psi$ are mutually inverse bijections between the sets 
\[\{P_1\in\match \calg_1\mid \textup{$P_1$ has a switching position in $i_1(\calg)$ and $i_2(\calg)$}\}\]
and $\match\calg_3\times\match\calg_4^\circ$.
On the other hand, it follows  from {\cite[section 3]{CS} that $\psi(P_{56})$ does not have a switching position and therefore} $\varphi$ and $\psi$ are mutually inverse bijections between the sets 
\[\{P_1\in\match \calg_1\mid \textup{$P_1$ has no switching position in $i_1(\calg)$ and $i_2(\calg)$}\}\]
and $\match\calg_{56}$.
We have proved the following theorem.

\begin{thm} \label{thm 3.7}In the case $s'>t+1$, the map \[\varphi:\match\calg_1\to{ \match(\res_\calg (\calg_1))}
\] is a bijection with inverse $\psi$.
\end{thm}

{Let us also point out the following useful fact.}
{
\begin{prop}
 In the case $s'>t+1$, let $P_{56}\in\match\calg_{56}$ and let $P_1=\psi(P_{56})$. Then $P_1|_{i_1(\calg)}$ and $P_1|_{i_2(\calg)}$ consist of boundary edges of $\calg$ only and are complementary as matchings of $\calg$.
\end{prop}
\begin{proof}
 By definition the restriction of $P_1$ to both subgraphs consists of boundary edges only. {So we need to show complementarity.}
 
 $P_1$ does not contain the boundary edge $e$ in $_{SW}G_s$ whose sign is $f(e)=f(e_{s'-1})$. Then by Lemma~\ref{lemNE}, the matching $P_1|_{i_1(\calg)}$ consists precisely of all south and west boundary edges of $i_1(\calg)$ whose sign is $-f(e)$ and all north and east boundary edges of $i_1(\calg)$ whose sign is $f(e)$.
 Similarly  $P_1$ does not contain the boundary edge $e'$ in $G_{t'}^{NE}$ whose sign is $f(e')=f(e_t)$. Then again by Lemma~\ref{lemNE}, the matching $P_1|_{i_2(\calg)}$ consists precisely of all south and west boundary edges of $i_1(\calg)$ whose sign is $f(e')$ and all north and east boundary edges of $i_2(\calg)$ whose sign is $-f(e')$.

 Since $i_1(\calg)$ and $i_2(\calg)$ form a  crossing overlap, we have $f(e_t)=
 f(e_{s'-1})$ and thus $f(e)=f(e')$. 
This shows that  $P_1|_{i_1(\calg)}$ and $P_1|_{i_2(\calg)}$ consist of boundary edges of $\calg$ only and are complementary as matchings of $\calg$.\end{proof}

}
\subsection{Self-crossing case 2: Overlap in the opposite direction.}
Recall that 
\[\calg_{34}=\calg_1[1,t]\cup\ocalg_1[{ s'-1,t+1}]\cup\calg_1[s',d],\]
where $i_1(\calg)=\calg_1[s,t]$ and $i_2 (\calg)=\calg_1[s',t']$ are the self-overlaps.

Let $P_1\in\match\calg_1$. If $P_1$ has a switching position in $i_1(\calg)$ and $i_2(\calg)$, choose the first switching position and define a matching $P_{34}$ by
\[P_{34}=P_1\qquad\textup{on ${\calg_1[1,s-1]\cup\ocalg_1[{ s'-1,t+1}]\cup\calg_1[t'+1,d]}$} \]
\[\begin{array}{lll}P_{34}=\left\{\begin{array}{ll}
P_1|_{i_1(\calg)} &\textup{before the first switching position;}\\
P_1|_{i_2(\calg)} &\textup{after the first switching position;}\\\end{array}\right\}& \textup{on $i_1(\calg)$.}\end{array}\]
\[\begin{array}{lll}P_{34}=\left\{\begin{array}{ll}
P_1|_{i_2(\calg)} &\textup{before the first switching position;}\\
P_1|_{i_1(\calg)} &\textup{after the first switching position;}\\\end{array}\right\}& 
\textup{on $i_2(\calg)$.}
\end{array}\]

If $P_1$ has no switching position in $i_1(\calg)$ and $i_2(\calg)$, we define $P_5\in\match \calg_5$ by $P_5=P_1|_{\calg_5}$ and $P_6\in\match \calg_6^b$ by 
\[P_6=\left\{\begin{array}{ll}
P_1|_{\calg_6\setminus\{b\}} &\textup{if both $b,b'\in P_1$};\\
P_1|_{\calg_6} &\textup{otherwise}.\end{array}\right.\]

{With this notation,} we define 
\[\varphi(P_1)=\left\{\begin{array}{ll} P_{34} &\textup{if $P_1$ has a switching position in $i_1(\calg)$ and $i_2(\calg)$}; \\
(P_5,P_6)
&\textup{otherwise}.\end{array}\right.
\]

To show that $\varphi$ is a bijection, we define its inverse function $\psi$. Let $P_{34}\in\match\calg_{34}$ and $(P_5,P_6)\in\match\calg_5\times\calg_6^\circ$. Define $\psi(P_{34})$ to be the matching of $\calg_1$ obtained from $P_{34}$ by switching at the first possible switching position. Thus 
\[
\psi(P_{34})=P_{34} \qquad \textup{on $\calg_1[1,s-1]\cup\ocalg_1[{ s'-1,t+1}]\cup\calg_1[t'+1,d]$}\] 
\[\begin{array}{lll}
\psi(P_{34})=\left\{\begin{array}{ll}
P_{34}|_{j_1(\calg)} &\textup{before the first switching position;}\\
P_{34}|_{j_2(\calg)} &\textup{after the first switching position;}\\\end{array}\right\}& \textup{on $i_1(\calg)$;}\\
\psi(P_{34})=\left\{\begin{array}{ll}
P_{34}|_{j_2(\calg)} &\textup{before the first switching position;}\\
P_{34}|_{j_1(\calg)} &\textup{after the first switching position;}\\\end{array}\right\}& \textup{on $i_2(\calg)$;}\end{array}\]
where $j_1:\calg\to\calg_{34}$, $j_2:\calg\to\calg_{34}$
are the embeddings of the overlap such that $j_1(\calg)\subset \calg_{34}$ corresponds to $i_1(\calg)\subset\calg_1$.

On the other hand, $\psi(P_5,P_6)$ is the unique matching of $\calg_1$ whose restriction to $\calg_5\sqcup\calg_6^\circ$ is the pair $(P_5,P_6)$ and whose restriction to the complement $\calg_1\setminus(\calg_5\cup\calg_6^\circ)
$ consists of boundary edges only.

\begin{lem}\label{lem opp}
 $P_6\in\match\calg_6^\circ$.
\end{lem}
\begin{proof}
 Suppose first that $\calg_6^\circ$ consists of at least two tiles. Let $\hat\calg_6=(\calg_6^\circ)_{e_{t+1}}$ be the snake graph obtained from $\calg_6^\circ$ by cutting along the interior edge $e_{t+1}$ between the tiles $G_{t+1}$ and $G_{t+2}$. We cut $\calg_1$ along the same edge into the two snake graphs
 \[\calg_1'=\calg_1[1,t+1], \quad \calg_2'=\calg_1[t+2,d],
\]
with $t+2<s'$ and obtain two perfect matchings $P_1'$ and $P_2'$ of $\calg_1'$ and $\calg_2'$ respectively, one by restricting $P_1$ and the other by restricting $P_1$ and adding the cut edge. Say $P_1'=P_1|_{\calg_1'}$ and $P_2'=P_2|_{\calg_2'}\cup\{\textup{cut edge}\}$. 
Observe that $\calg_1'$ and $\calg_2'$  have a crossing overlap $i_1(\calg)$, $i_2(\calg)$ and $(P_1',P_2')$ is a pair of matchings without switching position in the overlap. Theorem 3.1 of \cite{CS} implies that $(P_1'\cup P_2')|_{\hat\calg_6}$ is a perfect matching of the band graph $\calg_6^\circ$. Observe that, since $P_1$ originally was a matching of the snake graph $\calg_1$, its restriction to $\calg_6$ is a good matching of the band graph. This completes the proof in this case.

Suppose now that $\calg_6^\circ$ consists of a single tile. Since there is no switching position at the tiles $G_t,G_{s'}$,  we see from \cite[Figure 7]{CS} that the local configuration is one of the two shown in Figure~\ref{fig lem opp}. { Note that the tiles $(G_t,G_{t+1},G_{s'})$ must form a zigzag, since the crossing overlap condition implies that $f(e_t)=f(e_{s'-1})$.}

\begin{figure}
\begin{center}
\tiny\scalebox{1}{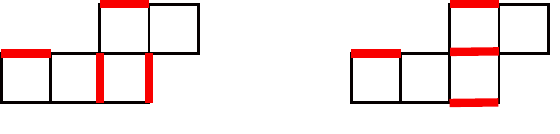}
\caption{Proof of Lemma \ref{lem opp}}
\label{fig lem opp}
\end{center}
\end{figure}

In both cases, $\calg_6^\circ$ is glued along the boundary edges $b,b'$ of the tile $G_{t+1}$, and thus $P_6\in\match\calg_6$.
\end{proof}

The following lemma shows that $\psi$ is well-defined.
\begin{lem}
 \begin{itemize}
\item [(a)] $P_{34}$ always has a switching position in $j_1(\calg)$ and $j_2(\calg)$. 
\item [(b)] $P_5\cup P_6$ can be completed to a matching of $\calg_1$ using only boundary edges.
\item[(c)] For every pair of matchings $(P_5,P_6)$, the completion in (b) is unique and complementary on the overlaps $i_1(\calg)$ and $i_2(\calg)$.
\end{itemize}
\end{lem}
\begin{proof}
 Part (a) follows from \cite[Lemma 7.3]{CS} and parts (b) and (c) from \cite[Lemma 7.4]{CS}.
\end{proof}

\begin{thm} If the overlap is in the opposite direction then the map \[\varphi:\match\calg_1\to { \match(\res_\calg(\calg_1))}
\] is a bijection with inverse $\psi$.
\end{thm}

\begin{proof} The proof is analogous to the proof of Theorem \ref{thm 3.7}.
\end{proof}

\subsection{Self-crossing case 3:  Overlap in the same direction and {$\bf{s'<t}$.}} \label{sect 3.4} Let $\calg_1$ be a self-crossing snake graph with self-overlap $i_1(\calg)=\calg_1[s,t]\cong\calg_1[s',t']=i_2(\calg)$ with $s'<t$, and let $\calg_3,\calg_4^\circ,\calg_{56}$ be the resolution of the self-crossing as defined in section~\ref{sect resolution}.
We will define  a map 
\[
\begin{array}
 {rcl}
{\varphi } : \match \calg_1 \cup \match \calg_{56}&\longrightarrow 
& \match \calg_3\times \match \calg_4^\circ 
\\
P&\longmapsto &( \varphi_3({ P}),\varphi_4({ P})).
\end{array}\]
{Recall that in this case $\calg_{56}$ is a negative snake graph in $\mathcal{R}$. This is why $\match \calg_{56}$ is now part of the domain of $\varphi$.}
Let $P_1\in \match \calg_1$. See the top row of Figure~\ref{fig bijection1} for an example.  If the matching $P_1$ contains an edge of $\calg_1[s,2s'-s-1]$ which is an interior edge in $\calg_4^\circ$, then let  $e$ denote the first such edge. Then the snake graph $(\calg_4^\circ)_e$ obtained from $\calg_4^\circ$ by cutting along the edge $e$  is isomorphic to the subsnake graph $\calg_4'\subset \calg_1$ consisting of the last $(s'-s)$ tiles preceding the edge $e$ in $\calg_1$ or the first $(s'-s)$ tiles following $e$ in $\calg_1$, depending whether the edge $e$ comes after or before the tile $G_{s'}$ in $\calg_1$. On the other hand, the band graph $\calg_4^\circ$ can be recovered from $\calg_4'$ by glueing the edge $e$ to an edge $e'$ at the opposite end of $\calg_4'$. Thus $\calg_4^\circ=(\calg_4')^e$. The restriction of $P_1$ to this subgraph $\calg_4'$  induces a perfect matching $\varphi_4(P_1)$ on $\calg_4^\circ=(\calg_4')^e$, since $P_1$ contains the glueing edge $e$. 
Indeed, if the vertices incident to $e'$ are matched in $P_1$ by the edges in $\calg_4'$, then the induced matching $\varphi_4(P_1)$ on $\calg_4^\circ$ is the restriction of $P_1$ to $\calg_4'\setminus\{e\}$ (as in Figure~\ref{fig bijection1}), and if the vertices incident to $e'$ are matched in $P_1$ by edges in $\calg_1\setminus\calg_4'$, then $\varphi_4(P_1)$ is the restriction of $P_1$ to $\calg_4$. 

\begin{figure}
\begin{center}
\scalebox{1}{\small 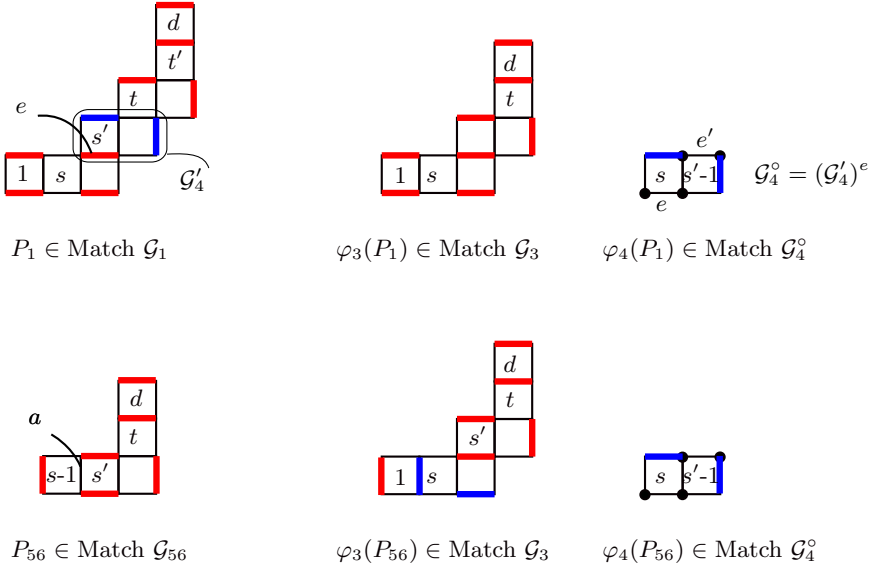}
\caption{The bijection $\varphi$}
\label{fig bijection1}
\end{center}
\end{figure}

On the other hand, if $P_1 $ does not contain an edge from $\calg_1[s,2s'-s-1]$ which is an interior edge in $\calg_4^\circ$, then the first $(s'-s)$ edges in $P_1$ on the set
$\calg_1[s,s'-1]\setminus \{e\}$, where $e$ is the boundary edge of $G_s$ that becomes interior in $\calg_4$, induce a matching $\varphi_4(P_1)$ on $\calg_4^\circ$ consisting of boundary edges only.

In both cases, we have constructed a perfect matching $\varphi_4(P_1)$ on $\calg_4^\circ$, and moreover $\varphi_4(P_1)$ is a subset of $P_1$.
Let $\varphi_3(P_1)= P_1\setminus \varphi_4(P_1)$ be the complement. Then it follows from the construction that $\varphi_3(P_1)$ is a perfect matching of $\calg_3$.

Now let $P_{56}\in \match \calg_{56}$. See the bottom row of Figure~\ref{fig bijection1} for an example. We define a pair $(\varphi_3(P_{56}),\varphi_4(P_{56}))\in \match\calg_3\times\match\calg_4^\circ$.
Suppose without loss of generality that in $\calg_{56}$
the tile $G_{s-1}$ is west of the tile $G_{s'}$, and denote by $a$ the interior edge between these tiles. {Let $x$ be the northern endpoint of $a$, and let $y$ be the southern endpoint}. Let $e(x)\in P_{56}$ (respectively $e(y)\in  P_{56}$) be the unique edge incident to $x$ (respectively $y$). Then there are three cases.

\begin{itemize}
\item [case 1: ] $e(x)$ is the north edge of $G_{s-1}$. This implies that $e(y)$ is an edge of $G_{s-1}$ or $G_{s-2}$. In this case, define $\varphi_3(P_{56})$ to be the matching $P_{56}$ on $\calg_{56}\subset\calg_3$ and completed by boundary edges of $\calg_3$. 
\item [case 2: ] $e(x)=e(y)=a$ is the interior edge shared by $G_{s-1}$ and $G_{s'}$. In this case, define $\varphi_3(P_{56})$ to be the matching $P_{56}$ on $\calg_{56}\subset\calg_3$, where we agree that the edge $a$ is in the tile $G_{s-1}$, and  complete $\varphi_3(P_{56})$  with boundary edges of $\calg_3$.
\item [case 3: ] $e(x)$ is not an edge of $G_{s-1}$. This implies that $e(y)$ is the south edge of $G_{s'}$. In this case, define $\varphi_3(P_{56})$ to be the matching $P_{56}$ on $\calg_{56}\subset\calg_3$ together with the edge $a$ on the tile $G_{s-1}$ and  then completed by boundary edges of $\calg_3$. This is the case shown in Figure~\ref{fig bijection1}.
\end{itemize}
In each case, define $\varphi_4(P_{56})$ to be the unique matching on $\calg_4^\circ$ consisting of those boundary edges that were not used in the completion of $\varphi_3(P_{56})$. In other words, the matchings $\varphi_3(P_{56})\setminus P_{56}$ and $\varphi_4(P_{56})$ are complementary and consist of boundary edges only.

\begin{remark}\label{rem 3.4}
 In each case, the vertex $x$ of the tile $G_{s-1}$ in $\calg_3$ is matched in $\varphi_3(P_{56})$ by an edge in $G_{s-1}$ (namely $e(x)$ in case 1 and the edge $a$ in the other cases). This implies that 
 \begin{enumerate}
\item 
the completed part $\varphi_3(P_{56})\cap\calg_3(s,s'-1)$ is the same in each case, 
\item $\varphi_4(P_{56})$ does not depend on $P_{56}$,
\item the matching $\varphi_4(P_{56})$ does not contain the edge $a$.
\end{enumerate} \end{remark}

%

In order to show that the map $\varphi $ is a bijection, we construct its inverse \[
\begin{array}
 {rcl}
{\psi } : \match \calg_3\times \match \calg_4^\circ &\longrightarrow 
&  \match \calg_1 \cup \match \calg_{56}.
\end{array}\]
Let $(P_3,P_4)
\in \match \calg_3\times \match \calg_4^\circ$. We define $\psi(P_3,P_4)$ by
\[\psi(P_3,P_4)|_{\calg_1\setminus\calg_4'' }= P_3 \quad \textup{and} \quad \psi(P_3,P_4)|_{\calg_4''}=P_4\] 
where  $\calg_4''\subset\calg_1[s,2s'-s-1]$ is the first subgraph such that this definition yields a matching on $\calg_1$ and $\calg_4''=(\calg_4^\circ)_e$ for some interior edge $e$, 
if such a subgraph $\calg_4''$ exists.
If such a subgraph $\calg_4''$ does not exist, we define \[\psi(P_3,P_4) = P_{56}=P_3|_{\calg_{56}}.\]

\begin{lem}\label{lem 3.5}
 The subgraph $\calg_4''$ in the definition above exists if and only if the pair $(P_3,P_4) $ is not of the form $(\varphi_3(P_{56}),\varphi_4(P_{56}))$ with $P_{56}\in \match\calg_{56}$. 
 \end{lem} 
 
\begin{proof}
 ($\Rightarrow$) Suppose   that $(P_3,P_4) = (\varphi_{3}(P_{56}),\varphi_{4}(P_{56}))$ for some $P_{56}\in \match \calg_{56}$. By definition of $\varphi$, the matchings $P_4$ and $P_3|_{\calg_1[s,s'-1]}$ are complementary and consist of boundary edges only. It follows that $\calg_4''$  could only be  $\calg_1[s,s'-1]$ or $\calg_1[s',2s'-s-1]$. On the other hand, Remark~\ref{rem 3.4} implies that the northeast vertex $x$ of $G_{s-1}$ is matched in $P_3$ by the edges of $G_{s-1}$ and  the northwest vertex of $G_s$, which is also the point $x$, is matched in $P_4$ by a boundary edge. This implies that $  \calg_4''$  cannot be  $\calg_1[s,s'-1]$. A similar argument shows that $\calg_4''$ cannot be $\calg_1[s',2s'-s-1]$ either. It follows that $\calg_4''$ does not exist.
 
 ($\Leftarrow$) Suppose   that the pair $(P_3,P_4) $ {is not of the form $(\varphi_3(P_{56}),\varphi_4(P_{56}))$ with $P_{56}\in \match\calg_{56}$}. It has been shown in \cite[Section 3]{CS} that if one of $P_4$ or $P_3|_{\calg_3\setminus\calg_{56}}$ contains an interior edge of $\calg_4^\circ$, or if both $P_4$ and $P_3|_{\calg_3\setminus\calg_{56}}$ have an edge in common, then $\calg_4''$ always exists.
 
 Thus we only need to consider the pairs $(P_3,P_4) $ where  $P_4$ and $P_3|_{\calg_3\setminus\calg_{56}}$ are complementary and consist of boundary edges only.
 Moreover $(P_3,P_4)$ {is not of the form $(\varphi_3(P_{56}),\varphi_4(P_{56}))$ with $P_{56}\in \match\calg_{56}$}, and from the construction of $\varphi$ and Remark \ref{rem 3.4} (3), it follows that the edge $a$ in the definition of $\varphi$ belongs to $P_4$. Therefore $\calg_4''=\calg_1[s,s'-1]$ satisfies the required properties.
\end{proof}

\begin{thm} {In the case $s'<t$, the map \[\varphi: \match \calg_1 \sqcup \match \calg_{56}\longrightarrow 
 \match \calg_3\times \match \calg_4^\circ 
\]
is a bijection with inverse $\psi$.}\end{thm}
\begin{proof} Let $P_1\in \match \calg_1$. If the interior edge $e$ in the definition of $\varphi$ exists, we have
$\psi\varphi(P_1) =\psi(\varphi_3(P_1),\varphi_4(P_1)) =P_1$, because in this case the subgraph $\calg_4''$ in the definition of $\psi$ is equal to the subgraph $\calg_4'$ in the definition of $\varphi$.
On the other hand, if no such $e$ exists then $\calg_4''$ is the subgraph 
$\calg_1[s,s'-1]\setminus \{e\}$, where $e$ is the boundary edge of $G_s$ that becomes interior in $\calg_4$, and again $\psi\varphi(P_1) =\psi(\varphi_3(P_1),\varphi_4(P_1)) =P_1$.

Next, let $P_{56}\in \match\calg_{56}$. Then 
$\psi\varphi(P_{56}) =\psi(\varphi_3(P_{56}),\varphi_4(P_{56})) = \varphi_3(P_{56})|_{\calg_{56}} = P_{56} $
where the second identity holds because of Lemma \ref{lem 3.5} and the last identity by definition of $\varphi_3$.

This shows that $\psi\varphi$ is the identity.

Now let $(P_3,P_4)\in\match\calg_3\times\calg_4$. If $\calg_4''$ in the definition of $\psi$ exists then 
$\varphi \psi (P_3,P_4) =\varphi (P_3\cup P_4) = (P_3,P_4)$, by construction. On the other hand, if $\calg_4''$ does not exist, then Lemma~\ref{lem 3.5} implies that there exists $P_{56}\in \match \calg_{56}$ such that 
$(P_3,P_4)= \varphi(P_{56})$. In this case $\psi (P_3,P_4)= \varphi_3(P_{56})|_{\calg_{56}}=P_{56}$, and thus $\varphi \psi (P_3,P_4) =\varphi(P_{56})= (P_3,P_4)$.
 This shows that $\varphi\psi$ is the identity.
\end{proof}

\subsection{Self-crossing case 4:  Overlap in the same direction and {$\bf{s'=t+1}$}} \label{sect s'=t+1} The bijections $\varphi, \psi$ are almost exactly as in the case $s'>t+1$ except for one particular type of matchings which we describe now.
Let $k$ be as in section \ref{sect resolution} case (3){(c)}. Recall that in this case there is a second overlap $j_1(\calh)={\ocalg_1[s-1,k+1]}$  and $j_2(\calh)=\calg_1[t'+1,t'+s-k-1] $ just before and after the overlap $i_1(\calg)=\calg_1[s,t]$ and $i_2(\calg)=\calg_1[s',t']$,
which determines the sign of $\calg_{56}$ in the resolution of the self-crossing.

  We consider $\calg_1$ as a union of 6 subgraphs
\[\calg_1=\calg_1[1,k]\cup\calg_1[k+1,s-1]\cup\calg_1[s,t]\cup\calg_1[s',t']\cup\calg_1[t'+1,t'+s-k-1]\cup\calg_1[t'+s-k, d].\]

Recall that $\calg_{56}=\calg_1[1,k]\cup\calg_1[t'+s-k,d]$ glued along the edge $e$ which is the boundary edge in $G_k^{N\!E}$ and the edge $e'$
which is the boundary edge in $_{SW}G_{t'+s-k}$. 
Suppose $P_1\in\match\calg_1 $ is such that $e_t\in P_1$, $P_1|_{\calg_1({k+1},t)}$ consists of boundary edges only, $P_1|_{\calg_1(s',{t'+s-k-1})}$ consists of boundary edges only, and is complementary on both pieces.

Without loss of generality, suppose $e_t$ is {north} of $G_t$ and $f(e_t)=-$. Then $f(e_{s-1})=+=f(e_{t'})$, {because we have a crossing overlap. Moreover, Lemma~\ref{lemNE} implies that} the east and the north edges in $P_1|_{\calg_1(k+1,t)}$ have sign $-$
(thus  the east and the north edges in $P_1|_{\calg_1(s',t'+s-k-1)}$ have sign $+$)
and the south and the west edges in $P_1|_{\calg_1(k+1,t)}$ have sign $+$ (thus  the south and the {west} edges in $P_1|_{\calg_1(s',t'+s-k-1)}$ have sign $-$).
Moreover, {since the overlaps $j_1(\calh), j_2(\calh)$ have opposite direction, we have} $f(e_{s-\ell-1})=-f(e_{t'+\ell})$ for $\ell=1,2,\ldots,s-k-2$,
$f(e_k)=f(e_{t'+s-k-1})$ and $f(e)=f(e')$.

There are two cases: the overlaps $j_1({\calh})$ and $j_2({\calh})$ cross or not. Let us suppose first that they do not cross.
Then $f(e_k)=f(e_{s-1})=f(e_{t'+s-k+1})=f(e_{t'})=+$, because it is a non-crossing overlap in the opposite direction. 
Then $f(e)={-f(e_k)=  -}$.

If $G_{k+1}$ is east of $G_k$ then $e$ is the north edge of $G_k$ and this shows that  the south edge  of $G_{k+1}$ is not in $P_1$ because it has sign $-$; and 
if $G_{k+1}$ is north of $G_k$ then {$e$ is the east edge of $G_k$ and this shows that} the west edge  of $G_{k+1}$ is not in $P_1$ because it has sign  $-$.
Therefore the two vertices of {the interior edge $e_k$ shared by the tiles $G_k$ and $G_{k+1}$} are matched by edges from $G_k$ or $G_{k-1}$. Thus either $e_k\in P_1 $ or $e\in P_1$ and similarly, either  $e_{t'+s-k-1}\in P_1 $ or $e'\in P_1$.  
 We define 
 \[ \varphi(P_1)=P_1|_{\calg_{56}}\setminus \{e \}\textup{ if $e\in P_1$ or $e,e'\in P_1$;} \]
 and 
 \[\varphi(P_1)=P_1|_{\calg_{56}}\setminus \{e' \}\textup{ if $e'\in P_1$ and $e\notin P_1$.}\]
 
The only remaining case is $e_k\in P_1$ and $e_{t'+s-k-1}\in P_1$, see Figure \ref{funnymatching}. This is the special  case, and we denote {the restriction $P_1|_{\calg_1[k+1,t'+s-k-1]}$} by $P_{\tiny\smiley}$. {See Figure~\ref{funnymatching} for an example where $k=1$ and $j_1(\calh)$ is the second tile.}

\begin{figure}
\begin{center}
\scalebox{1}{\small 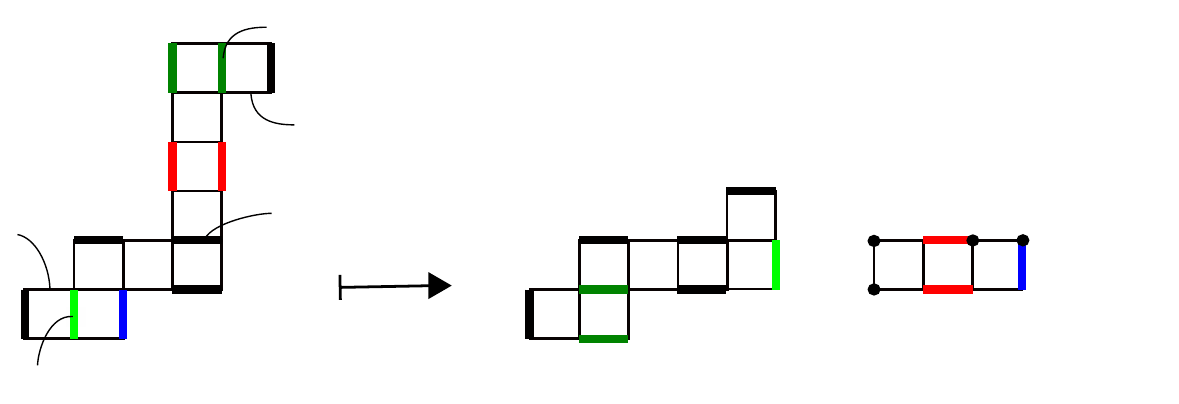}
\caption{The matching $P_{\tiny\smiley}$ with $k=1$, $j_1(\calh)=G_2$, $j_2(\calh)=G_{t'+1}$}
\label{funnymatching}
\end{center}
\end{figure}

 In this case, we define
 $\varphi(P_{\tiny\smiley}) $ as follows
 \[
\begin{array}{lll}
 \textup{on $ \calg_4^\circ$ use $P_{\tiny\smiley}|_{\calg_1(s',t')}\cup \{$boundary edge of $G_{s-1}^{N\!E}\}$} \\
 \textup{on $ \calg_3[1,k)$ use $P_{\tiny\smiley}|_{\calg_1[1,k)}$}\\ 
 \textup{on $ \calg_3[k+1,s-1]$ use $P_{\tiny\smiley}|_{j_2(\calh)}$}\\
 \textup{on $ \calg_3(s,t]$ use $P_{\tiny\smiley}|_{\calg_1(s,t]}$}\\  
 \textup{on the part of $\calg_3$ coming from $  \calg_1(t'+1,t'+s-k-1]$ use $P_{\tiny\smiley}|_{\calg_1[k,s-1)}$}\\  \textup{on the part of $\calg_3$ coming from $ \calg_1(t'+s-k,d]$ use $P_{\tiny\smiley}|_{\calg_1(t'+s-k,d]}$.}
 \end{array}
 \]
 
Now we define the inverse map for this special case.
Let $(P_3,P_4)\in \match\calg_3\times\match\calg_4^\circ$ be the image of the special matching $P_{\tiny\smiley}$ of $\calg_1$ as described above. In this case $P_4$ consists of boundary edges 
and $P_3$ consist of boundary edges on the part of $\calg_3$ given by $
i_3(\calg)\cup j_4(\calh)$ 
(this is the part coming from $
\calg_1[s,t]\cup \calg_1[t'+1,t'+s-k-1]$),
and
$P_3$ and $P_4$ are complementary on $i_3(\calg)$ and $\calg_4^\circ$.
Let $\psi(P_3,P_4)$ be the matching $P_{\tiny\smiley}$ on $\calg_1[k+1,t'+s-k-1]$ and on $\calg_1[1,k)$ and $\calg_1(t'+s-k,d]$ use the restriction of $\calg_3$ to the corresponding subgraphs of $\calg_3$.

Finally suppose that $j_1(\calh)$ and $j_2(\calh)$ are crossing overlaps. Then we define $\varphi(P_1)$ in exactly the same way as above. In this case, $f(e_k)=-f(e_{s-1})=-f(e_{t'})=f(e_{t'+s-k+1})$. 
If $P_1\in\match\calg_1 $ satisfies the condition above, that is $P_1$ is such that $e_t\in P_1$, $P_1|_{\calg_1(k,t)}$ consists of boundary edges only, $P_1|_{\calg_1(s',t'+s-k)}$ consists of boundary edges only, and is complementary on both pieces, 
then the glueing edge $e$ is  not in $P_1$. Thus $\varphi(P_1)$ is always a matching of $\calg_3\sqcup\calg_4^\circ$.
On the other hand, if $P_{56}\in\match\calg_{56}$ let $\varphi(P_{56})$ be the unique completion to a matching of $\calg_3$ using only boundary edges, together with the unique complementary matching on $\calg_4^\circ$ using only boundary edges. {This completes the case $s'=t+1.$ }
 
\subsection{Self-grafting}
Let $\calg_1=(G_1,\ldots,G_d)$ be a snake graph and let  $ \graft_{s,\e_3} (\calg_1)=(\calg_3 \sqcup \calg_4^\circ ) + \calg_{56}\in \mathcal{R}$
   be the {resolution of the self-grafting} of $\calg_1$  in $G_s$ as defined in section \ref{self-grafting}.

\begin{thm} \label{thm 3.11} There is a bijection 
\[\varphi:\match\calg_1\to{ \match(\graft_{s,\delta_3}(\calg_1))}.
\] \end{thm}

\begin{proof}
 The bijection $\varphi$ is the same as in the grafting case in Theorem \ref{bijections} (2). It is  given explicitly by the operation described in \cite[Figure 11]{CS}.
\end{proof}

\section{Labeled snake and band graphs arising from  cluster algebras of unpunctured surfaces}\label{sect 4}

In this section we recall how snake graphs {and band graphs} arise naturally in the theory of cluster algebras. We follow the exposition in \cite{MSW2}.
%

\subsection{Cluster algebras  from unpunctured
    surfaces}\label{sect surfaces} 
%

Let $S$ be a connected oriented 2-dimen\-sional Riemann surface with
nonempty
boundary, and let $M$ be a nonempty finite subset of the boundary of $S$, such that each boundary component of $S$ contains at least one point of $M$. The elements of $M$ are called {\it marked points}. The
pair $(S,M)$ is called a \emph{bordered surface with marked points}.
 
For technical reasons, we require that $(S,M)$ is not
a disk with 1,2 or 3 marked points.

\begin{defn}   \label{gen-arc}
A \emph{generalized arc}  in $(S,M)$ is a curve $\gamma$ in $S$, considered up
to isotopy,  such that:
\begin{itemize}
\item[(a)] the endpoints of $\gamma$ are in $M$; 
\item[(b)] except for the endpoints,
$\gamma$ is disjoint  from the boundary of $S$; and
\item[(c)]
$\gamma$ does not cut out a monogon or a bigon.
\end{itemize}

A  generalized arc $\zg$ is called an \emph{arc} if in  addition $\zg$ does not cross itself, except that its endpoints may coincide;

\end{defn}
%

Thus a generalized arc is allowed to cross itself a finite 
number of times.

Curves that connect two
marked points and lie entirely on the boundary of $S$ without passing
through a third marked point are \emph{boundary segments}.
Note that boundary segments are not   arcs.

For any two arcs $\zg,\zg'$ in $S$, let $e(\zg,\zg')$ be the minimal
number of crossings of 
arcs $\za$ and $\za'$, where $\za$ 
and $\za'$ range over all arcs isotopic to 
$\zg$ and $\zg'$, respectively.
We say that arcs $\zg$ and $\zg'$ are  \emph{compatible} if $e(\zg,\zg')=0$. 

A \emph{triangulation} is a maximal collection of
pairwise compatible arcs (together with all boundary segments). 

Triangulations are connected to each other by sequences of 
{\it flips}.  Each flip replaces a single arc $\gamma$ 
in a triangulation $T$ by a (unique) arc $\gamma' \neq \gamma$
that, together with the remaining arcs in $T$, forms a new 
triangulation.

\begin{defn}
Choose any   triangulation
$T$ of $(S,M)$, and let $\tau_1,\tau_2,\ldots,\tau_n$ be the $n$ arcs of
$T$.
For any triangle $\Delta$ in $T$, we define a matrix 
$B^\Delta=(b^\Delta_{ij})_{1\le i\le n, 1\le j\le n}$  as follows.
\begin{itemize}
\item $b_{ij}^\Delta=1$ and $b_{ji}^{\Delta}=-1$ if $\tau_i$ and $\tau_j$ are sides of 
  $\Delta$ with  $\tau_j$ following $\tau_i$  in the 
  clockwise order.
\item $b_{ij}^\Delta=0$ otherwise.
\end{itemize}
 
Then define the matrix 
$ B_{T}=(b_{ij})_{1\le i\le n, 1\le j\le n}$  by
$b_{ij}=\sum_\Delta b_{ij}^\Delta$, where the sum is taken over all
triangles in $T$.
\end{defn}

Note that  $B_{T}$ is skew-symmetric and each entry  $b_{ij}$ is either
$0,\pm 1$, or $\pm 2$, since every arc $\tau$ is in at most two triangles.

\begin{thm} \cite[Theorem 7.11]{FST} and \cite[Theorem 5.1]{FT}
\label{clust-surface}
Fix a bordered surface $(S,M)$ and let $\Acal$ be the cluster algebra associated to
the signed adjacency matrix of a   triangulation. Then the (unlabeled) seeds $\Sigma_{T}$ of $\Acal$ are in bijection
with  the triangulations $T$ of $(S,M)$, and
the cluster variables are  in bijection
with the arcs of $(S,M)$ (so we can denote each by
$x_{\gamma}$, where $\gamma$ is an arc). Moreover, each seed in $\Acal$ is uniquely determined by its cluster.  Furthermore,
if a   triangulation $T'$ is obtained from another
  triangulation $T$ by flipping an arc $\gamma\in T$
and obtaining $\gamma'$,
then $\Sigma_{T'}$ is obtained from $\Sigma_{T}$ by the seed mutation
replacing $x_{\gamma}$ by $x_{\gamma'}$.
\end{thm}

{From now on suppose that $\cala$ has principal coefficients in the initial seed $\zS_T=(\mathbf{x}_T,\mathbf{y}_T,B_T)$.}


\begin{defn}
A \emph{closed loop} in $(S,M)$ is a closed curve
$\gamma$ in $S$ which is disjoint from the 
boundary of $S$.  We allow a closed loop to have a finite
number of self-crossings.
As in Definition~\ref{gen-arc}, we consider closed
loops up to isotopy.
%
A closed loop in $(S,M)$ is called \emph{essential} if
  it is not contractible
and it does not have self-crossings.
\end{defn}

\begin{defn} 
A \emph{multicurve} is  a finite multiset of generalized  
arcs and closed loops such that there are only a finite number of pairwise crossings among the collection.
We say that a multicurve is \emph{simple}
if there are no pairwise crossings among the collection and 
no self-crossings.
\end{defn}

If a multicurve is not simple, 
then there are two ways to \emph{resolve} a crossing to obtain a multicurve that no longer contains this crossing and has no additional crossings.  This process is known as \emph{smoothing}.

\begin{defn}\label{def:smoothing} (Smoothing) Let $\gamma, \gamma_1$ and $\gamma_2$ be generalized  
arcs or closed loops such that we have
one of the following two cases:

\begin{enumerate}
 \item $\gamma_1$ crosses $\gamma_2$ at a point $x$,
  \item $\gamma$ has a self-crossing at a point $x$.
\end{enumerate}

\noindent Then we let $C$ be the multicurve $\{\gamma_1,\gamma_2\}$ or $\{\gamma\}$ depending on which of the two cases we are in.  We define the \emph{smoothing of $C$ at the point $x$} to be the pair of multicurves $C_+ = \{\alpha_1,\alpha_2\}$ (resp. $\{\alpha\}$) and $C_- = \{\beta_1,\beta_2\}$ (resp. $\{\beta\}$).

Here, the multicurve $C_+$ (resp. $C_-$) is the same as $C$ except for the local change that replaces the crossing {\Large $\times$} with the pair of segments 
  $\genfrac{}{}{0pt}{5pt}{\displaystyle\smile}{\displaystyle\frown}$ (resp. {$\supset \subset$}).    
\end{defn}   

Since a multicurve  may contain only a finite number of crossings, by repeatedly applying smoothings, we can associate to any multicurve  a collection of simple multicurves.  
  We call this resulting multiset of multicurves the \emph{smooth resolution} of the multicurve $C$.

{\begin{remark}\label{rem smoothing}
 This smoothing operation can be rather complicated, since the multicurves are considered up to isotopy. Thus after performing the local operation of smoothing described above, one needs to find representatives of the isotopy classes of $C_+$ and $C_-$ which have a minimal number of crossings with the triangulation. In practice, this can be quite difficult especially if one needs to smooth several crossings. This difficulty was one of the original motivations to develop the snake graph calculus. The isotopy is already contained in the definition of the resolutions of the (self-)crossing snake graphs. 
\end{remark}}

\begin{thm}\label{th:skein1}(Skein relations) \cite[Propositions 6.4,6.5,6.6]{MW}
Let $C,$ $ C_{+}$, and $C_{-}$ be as in Definition \ref{def:smoothing}. 
Then we have the following identity in $\cala$,
\begin{equation*}
x_C = \pm Y_1 x_{C_+} \pm Y_2 x_{C_-},
\end{equation*}
where $Y_1$ and $Y_2$ are monomials in the variables $y_{\tau_i}$.    
The monomials $Y_1$ and $Y_2$ can be expressed using the intersection numbers of the elementary laminations (associated to triangulation $T$) with  
the curves in $C,C_+$ and $C_-$.
\end{thm}

\subsection{Labeled snake graphs from surfaces}\label{sect tiles}\label{sect graph}
%
%

 Let 
$\zg$ be an   arc in $(S,M)$ which is not in $T$. 
Choose an orientation on $\zg$, let $s\in M$ be its starting point, and let $t\in M$ be its endpoint.
We denote by
$s=p_0, p_1, p_2, \ldots, p_{d+1}=t$
the points of intersection of $\zg$ and $T$ in order.  
Let $\tau_{i_j}$ be the arc of $T$ containing $p_j$, and let 
$\zD_{j-1}$ and 
$\zD_{j}$ be the two   triangles in $T$ 
on either side of 
$\tau_{i_j}$. Note that each of these triangles has three distinct sides, but not necessarily three distinct vertices, see Figure \ref{figr1}.
\begin{figure}
\includegraphics{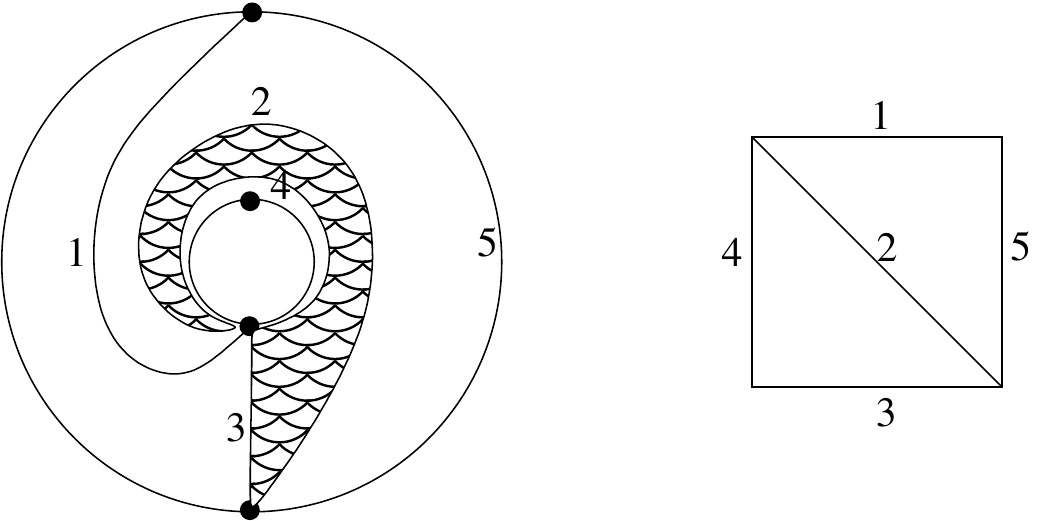}
\caption{On the left, a triangle with two vertices; on the right the tile $G_{j}$ where $i_j=2$. }\label{figr1}
\end{figure}
Let $G_j$ be the graph with 4 vertices and 5 edges, having the shape of a square with a diagonal, such that there is a bijection between the edges of $G_j$ and the 5 arcs in the two triangles $\zD_{j-1}$ and $\zD_j$, which preserves the signed adjacency of the arcs up to sign and such that the diagonal in $G_j$ corresponds to the arc $\tau_{i_j}$ containing the crossing point $p_j$. Thus $G_j$ is given by the quadrilateral in the triangulation $T$ whose diagonal is $\tau_{i_j}$. 

Given a planar embedding $\tilde G_j$ 
of $G_j$, we define the \emph{relative orientation} 
$\mathrm{rel}(\tilde G_j, T)$ 
of $\tilde G_j$ with respect to $T$ 
to be $\pm 1$, based on whether its triangles agree or disagree in orientation with those of $T$.  
For example, in Figure \ref{figr1},  $\tilde G_j$ has relative orientation $+1$.

Using the notation above, 
the arcs $\tau_{i_j}$ and $\tau_{i_{j+1}}$ form two edges of a triangle $\zD_j$ in $T$.  Define $\tau_{e_j}$ to be the third arc in this triangle.

We now recursively glue together the tiles $G_1,\dots,G_d$
in order from $1$ to $d$, so that for two adjacent   tiles, 
 we glue  $G_{j+1}$ to $\tilde G_j$ along the edge 
labeled $\tau_{e_j}$, choosing a planar embedding $\tilde G_{j+1}$ for $G_{j+1}$
so that $\mathrm{rel}(\tilde G_{j+1},T) \not= \mathrm{rel}(\tilde G_j,T).$  See Figure \ref{figglue}.

\begin{figure}
\begin{center}
 \scalebox{1}{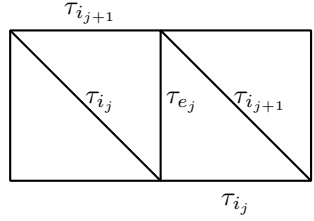}
\caption{Glueing tiles $\tilde G_j$ and $\tilde G_{j+1}$ along the edge labeled  $\tau_{e_j}$}
\label{figglue}
\end{center}
\end{figure}

After gluing together the $d$ tiles, we obtain a graph (embedded in 
the plane),
which we denote  by
${\calg}^\triangle_{\zg}$. 
\begin{defn}
The \emph{(labeled) snake graph} $\calg_{\gamma}$ associated to $\gamma$ is obtained 
from ${\calg}^\triangle_{\zg}$ by removing the diagonal in each tile.
\end{defn}

In Figure \ref{figsnake}, we give an example of an 
 arc $\gamma$ and the corresponding snake graph 
${\calg}_{\zg}$. Since $\gamma$ intersects $T$
five times,  
${\calg}_{\zg}$ has five tiles. 

\begin{remark}
 \label{rem sign}{
 Let $f$ be a sign function on $\calg_\zg$ as in section \ref{sect 2}. The interior edges $e_1,\ldots,e_{d-1}$  are corresponding to the sides of the triangles $\zD_1,\ldots,\zD_{d-1}$ that are not crossed by $\zg$. Two interior edges $e_j,e_k$ have the same sign $f(e_j)=f(e_k)$ if and only if the sides $\tau_{e_j},\tau_{e_k}$ lie on the same side of the segments of $\zg$ in $\zD_j$ and $\zD_k$, respectively.
 }
\end{remark}

 \begin{defn}If $\tau \in T$ then we define its (labeled) snake graph $\calg_{\tau}$ to be the graph consisting of one single edge with weight $x_{\tau}$ and two distinct endpoints (regardless whether the endpoints of $\tau$ are distinct).
 \end{defn}

\begin{figure}
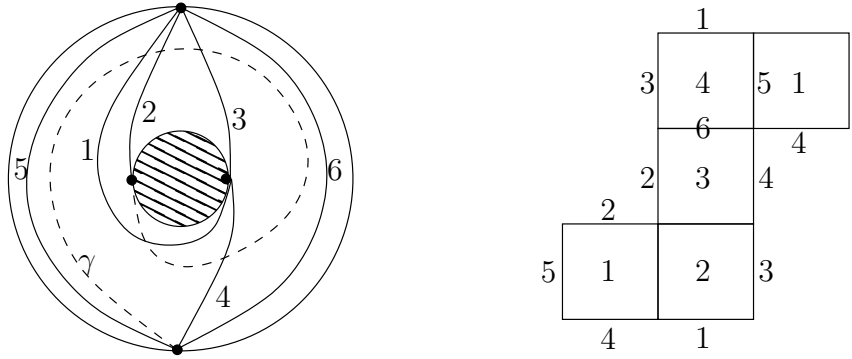
\caption{An arc $\gamma$ in a triangulated annulus on the left and the corresponding {labeled} snake graph $\calg_{\gamma}$ on the right. The tiles labeled 1,3,1 have positive relative orientation and the tiles 2,4 have negative relative orientation.}\label{figsnake}
\end{figure}

Now we associate a similar graph to closed loops.
Let $\zeta$ be a closed loop in $(S,M)$, which may or may not have self-intersections, 
but which is not contractible and has no contractible kinks.  Choose
an orientation for $\zeta$, and a 
triangle $\Delta$ which is crossed by $\gamma$.  Let $p$ be a point in the interior of $\Delta$ which lies on $\gamma$, and let $b$ and $c$ be the two sides of the triangle crossed by $\gamma$
immediately before and following its travel through point $p$.
Let $a$ be the third side of $\Delta$.   
We let $\tilde{\gamma}$ denote the arc from $p$ back to itself that exactly
follows closed loop $\gamma$. 

\begin{figure}
\scalebox{1.3}
{\small
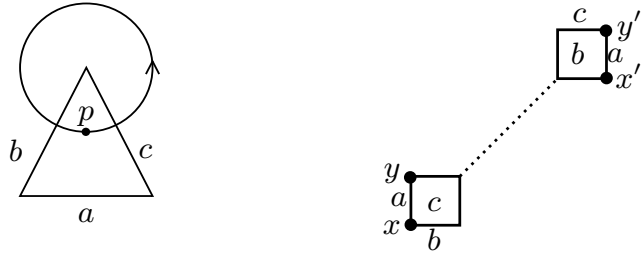
}
\caption{A triangle containing $p$ along a closed loop $\zeta$ (on the left) and the  corresponding band graph with $x\sim x'$, $y\sim y'$ (on the right) }
\label{fig band}
\end{figure}

We start by building the snake graph $\calg_{\tilde{\gamma}}$ 
as defined above.  In the first tile of $\calg_{\tilde{\gamma}}$,  
let $x$ denote the vertex
at the corner of the edge labeled $a$ and the edge labeled $b$, and let $y$ denote the vertex at the other end of the edge labeled $a$.  Similarly, in the last tile of
$G_{\tilde{\gamma}}$,  
let $x'$ denote the vertex at the corner of the edge labeled
$a$ and the edge labeled $b$, and let $y'$ denote the vertex at the other
end of the edge labeled $a$.  See the right of Figure \ref{fig band}.
Our convention for $x'$ and $y'$ are exactly opposite to those in \cite{MSW2}.

\begin{defn} \label{def band} The (labeled) \emph{band graph}  $\calg^\circ_{\zeta}$ associated to the loop $\zeta$
is the graph obtained from $\calg_{\tilde{\zeta}}$ by identifying the edges labeled $a$ in the first and last tiles so that the vertices $x$ and $x'$ and the vertices $y$ and $y'$ are glued together.  \end{defn}

\subsection{Snake graph formula for cluster variables}\label{secdefloop}
Recall that if $\tau$ is a boundary segment then $x_{\tau} = 1$,

If $\calg$ is a (labeled) snake graph and the edges of  a perfect matching $P$ of  
$\calg $ are labeled $\tau_{j_1},\dots,\tau_{j_r}$, then 
the {\it weight} $x(P)$ of $P$ is 
$x_{\tau_{j_1}} \dots x_{\tau_{j_r}}$.

Let $\zg$ be a generalized  arc and  $\tau_{i_1}, \tau_{i_2},\dots, \tau_{i_d}$
be the sequence of arcs in $T$ which $\zg$ crosses. The \emph{crossing monomial} 
of $\gamma$ with respect to $T$ is defined as
$$\mathrm{cross}(T, \gamma) = \prod_{j=1}^d x_{\tau_{i_j}}.$$
%
  

By induction on the number of tiles it is easy to see that the snake graph
$\calg_{\zg}$  
has  precisely two perfect matchings which we call
the {\it minimal matching} $P_-=P_-(\calg_{\zg})$ and 
the {\it maximal matching} $P_+
=P_+(\calg_{\zg})$, 
which contain only boundary edges.
To distinguish them, 
if  $\mathrm{rel}(\tilde G_1,T)=1$ (respectively, $-1$),
we define 
$e_1$ and $e_2$ to be the two edges of 
${\calg}^\triangle_{\zg}$ which lie in the counterclockwise 
(respectively, clockwise) direction from 
the diagonal of $\tilde G_1$.  Then  $P_-$ is defined as
the unique matching which contains only boundary 
edges and does not contain edges $e_1$ or $e_2$.  $P_+$
is the other matching with only boundary edges.
In the example of Figure \ref{figsnake}, the minimal matching $P_-$ contains the bottom edge of the first tile labeled 4.

\begin{lem}\cite[Theorem 5.1]{MS}
\label{thm y}
The symmetric difference $P_-\ominus P$ is the set of boundary edges of a 
(possibly disconnected) subgraph $\calg_P$ of $\calg_\zg$,
which is a union of cycles.  These cycles enclose a set of tiles 
$\cup_{j\in J} G_{j}$,  where $J$ is a finite index set.
\end{lem}

\begin{defn} \label{height} With the notation of Lemma \ref{thm y},
we  define the \emph{height monomial} $y(P)$ of a perfect matching $P$ of a snake graph $\calg_\gamma$ by
\begin{equation*}
y(P) = \prod_{j\in J} y_{\tau_{i_j}}.
\end{equation*}\end{defn}

Following \cite{MSW2}, for each generalized arc $\gamma$, we now define a Laurent polynomial $x_\zg$, as well as a polynomial 
$F_\zg^T$ obtained from $x_{\gamma}$ by specialization.
\begin{defn} \label{def:matching}
Let $\zg$ 
be a generalized arc and let $\calg_\zg$,
be its snake graph.  
\begin{enumerate}
\item If $\zg$ 
has a contractible kink, let $\overline{\zg}$ denote the 
corresponding generalized arc with this kink removed, and define 
$x_{\zg} = (-1) x_{\overline{\zg}}$.  
\item Otherwise, define
\[ x_{\gamma}= \frac{1}{\mathrm{cross}(T,\zg)} \sum_P 
x(P) y(P),\]
 where the sum is over all perfect matchings $P$ of $G_{\zg}$.
\end{enumerate}
Define $F_{\gamma}^T$ to be the polynomial obtained from 
$x_{\gamma}$ by specializing all the $x_{\tau_i}$ to $1$.

If $\gamma$ is a curve that 
cuts out a contractible monogon, then we define $\gamma =0$.     
\end{defn}

\begin{thm}\cite[Thm 4.9]{MSW}
\label{thm MSW}
If $\gamma$ is an arc, then 
$x_{\gamma}$ 
is a the cluster variable in $ \A$,
written as a Laurent expansion with respect to the seed $\Sigma_T$,
and $F_{\gamma}^T$ is its \emph{F-polynomial}.
\end{thm}

 Again following \cite{MSW2}, we define for every closed loop $\zeta$, a Laurent polynomial $x_\zeta$, as well 
as a polynomial 
$F_\zeta^T$ obtained from $x_{\zeta}$ by specialization.
\begin{defn} \label{def closed loop} 
Let $\zeta$  
be a closed loop.  
\begin{enumerate}
\item If $\zeta$ is a contractible loop,
 then let $x_\zeta = -2$.
\item If $\zeta$ 
has a contractible kink, let $\overline{\zeta}$ denote the 
corresponding closed loop with this kink removed, and define 
$x_{\zeta} = (-1) x_{\overline{\zeta}}$.
\item Otherwise, let 
$$x_{\zeta} = \frac{1}{\mathrm{cross}(T,\zg)} \sum_{P} 
x(P) y(P),$$ where the sum is over all good matchings $P$ of the  band graph $\calg^\circ_{\zeta}$. 
\end{enumerate}
Define $F_\zeta^T$ to be the Laurent polynomial obtained from 
$x_\zeta$ by specializing all the $x_{\tau_i}$ to $1$.
\end{defn}

\subsection{Bases of the cluster algebra}\label{sec:bangbrac}
{We recall the construction of the two bases given in \cite{MSW2} in terms of bangles and bracelets.}

\begin{defn}
Let $\zeta$ be an essential loop in $(S,M)$.  
The 
\emph{bangle} $\Bang_k \zeta$ is the union of $k$ loops isotopic to $\zeta$.
(Note that $\Bang_k \zeta$ has no self-crossings.)  And the 
\emph{bracelet} $\Brac_k \zeta$ is the closed loop obtained by 
concatenating $\zeta$ exactly $k$ times, see Figure \ref{figbangbrac}.  (Note that it will have $k-1$ self-crossings.)
\end{defn}

\begin{figure}
\scalebox{0.8}{\includegraphics{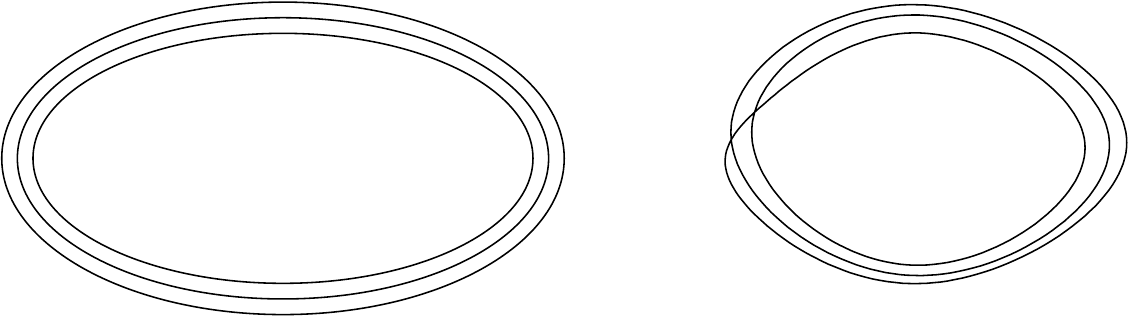}}
\caption{A bangle $\Bang_3 \zeta$, on the left, and a bracelet $\Brac_3 \zeta$, on the right.}\label{figbangbrac}
\end{figure}
Note that $\Bang_1 \zeta = \Brac_1 \zeta = \zeta$.

\begin{defn}
\label{def C0-compatible}
A collection $C$ of arcs and essential loops is called 
\emph{$\C^{\circ}$-compatible} if no two elements of $C$ cross each other.
Let $\C^{\circ}(S,M)$ be the set of all 
$\C^{\circ}$-compatible collections in $(S,M)$.
\end{defn}

\begin{defn}
A collection $C$ of arcs and bracelets is called 
\emph{$\C$-compatible} if:
\begin{itemize}
\item no two elements of $C$ cross each other except for the self-crossings of a bracelet; and
\item given an essential loop $\zeta$ in $(S,M)$, 
there is at most one $k\ge 1$ such
that the $k$-th bracelet $\Brac_k\zeta$ lies in $C$, and, moreover, there is at
most one copy of this bracelet $\Brac_k\zeta$ in $C$.
\end{itemize}
Let $\C(S,M)$  be the set of all $\C$-compatible
collections in $(S,M)$.
\end{defn}

Note that a $\C^{\circ}$-compatible collection may contain 
bangles $\Bang_k \zeta$ for $k \geq 1$, but it will not contain
bracelets $\Brac_k \zeta$ except when $k=1$.
And  a $\C$-compatible collection may contain bracelets, but will never
contain a bangle $\Bang_k \zeta$ except when $k=1$.

\begin{defn}
Given an arc or closed loop $c$, let 
$x_c$ denote the corresponding Laurent polynomial
defined in Section \ref{secdefloop}.
Let $\B^\circ$ 
be the set of all cluster algebra 
elements corresponding to the set $C^{\circ}(S,M)$, 
\[\B^{\circ} = \left\{\prod_{c\in C} x_c \ \vert \ C \in \C^{\circ}(S,M) \right\}.\]
Similarly, let
\[\B = \left\{\prod_{c\in C} x_c \ \vert \ C \in \C(S,M) \right\}.\]
\end{defn}

\begin{remark}
Both $\B^{\circ} $ and $\B $ contain
the cluster monomials of $\A $.
\end{remark}

We are now ready to state the main result of \cite{MSW2}.

\begin{thm}\cite[Theorem 4.1]{MSW2} If the surface has no punctures and at least two marked points then
the sets $\B^\circ$ and $\B$ are bases of the cluster algebra $\A$.                             
\end{thm}

\begin{remark}
 This result has been extended to surfaces with only one marked point in \cite{CLS}.
\end{remark}

\section{Relation to cluster algebras}\label{sect 5}
In this section we show how our results on abstract snake graphs are related to computations in cluster algebras from unpunctured surfaces. 

\subsection{Homotopy after removing a puncture and its effect on the snake graph}\label{sect homotopy}
{We start by studying the technique of introducing 
and removing a puncture and its effect on snake graphs.}
Let $T$ be a triangulation of an {unpunctured} surface $(S,M)$ and let $\Delta$ be a triangle in $T$ and label the sides of $\Delta$ by $1,2,3.$

Let $(\dot S, \dot M)$ be the surface obtained from $(S,M)$ by introducing a puncture in the interior of the triangle $\Delta,$ and let $T'$ be the triangulation obtained from $T$ by adding three arcs $a, b, c$ connecting  the puncture to the vertices of $\Delta$.

Let $\dot \gamma$ be an arc in $(\dot S, \dot M)$ and let $\gamma$ be the corresponding arc in $(S,M).$ Consider the  three local configurations illustrated in Figure \ref{puncture+homotopy}.

\begin{figure}\scalebox{0.8}{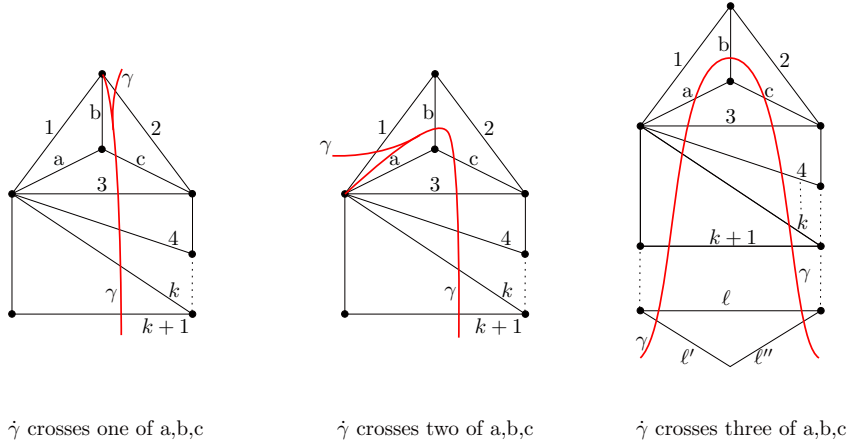}
\caption{Homotopy after removing a puncture } 
\label{puncture+homotopy}
\end{figure}

In the first case, the arc $\dot\gamma$ crosses one of the arcs $a,b,c$ and then the arc $\gamma$ crosses the same arcs as $\dot \gamma$ (locally) except for the arc $a, b$ or $c$ crossed by $\dot \gamma.$

In the second case, the arc $\dot\gamma$ crosses two of the arcs $a,b,c$ and   we have two subcases. If $\gamma$ crosses a side of $\Delta$ immediately before and after crossing two of the arcs $a, b, c$ then $\gamma$ crosses the same arcs as $\dot \gamma$ except for the two arcs of $a, b, c.$ On the other hand, if $\gamma$ starts at a vertex $x$ of $\Delta$ and then crosses two of the arcs $a,b,c,$ then the difference in the crossings of $\dot \gamma$ and $\gamma$ is the two arcs among $a, b, c$ plus a whole fan of arcs with common vertex $x.$ In the example in Figure~\ref{puncture+homotopy} these arcs are $b, c, 3, 4, \dots, k. $ 

In the third case, the {arc} $\dot\gamma$ crosses all three arcs $a,b,c$ and   the difference in the crossings of $\dot \gamma$ and $\gamma$ is the $3$ arcs $a, b, c$ plus the local overlap of arcs before and after the arcs $a, b, c.$ In the example in Figure \ref{puncture+homotopy}, these arcs are $\ell, \dots, k+1, k, \dots, 4, 3, a, b, c, 3,4, \dots, k, k+1, \dots, \ell. $ 

The corresponding snake graphs $\calg_{\dot \gamma}, \calg_{\gamma}$ are related by deleting the corresponding tiles and glueing the remaining pieces. 

Thus in the first case:
\[
\begin{tikzpicture}[scale=.4]
\node at (1.5,5){$\calg_{\dot \gamma}$};
\node at (-.5,-.5){$\Ddots$};
 \draw (0,0)--(1,0)--(1,2)--(0,2)--(0,0) (0,1)--(3,1)--(3,2)--(1,2) (2,1)--(2,3)--(3,3)--(3,2);
\node at (3.5,3.5){$\Ddots$};
\node[scale=.7] at (.5,.5){$2$};
\node[scale=.7] at (.5,1.5){$c$};
\node[scale=.7] at (1.5,1.5){$3$};
\node[scale=.7] at (2.5,1.5){$4$};
\node[scale=.7] at (2.5,2.5){$5$};

\node at (9.5,5){$\calg_{\gamma}$};
\draw (9,0)--(10,0)--(10,3)--(9,3)--(9,0) (9,1)--(10,1) (9,2)--(11,2)--(11,3)--(10,3);
\node at (8.5,-.5){$\Ddots$};
\node at (11.5,3.5){$\Ddots$};
\node[scale=.7] at (9.5,.5){$2$};
\node[scale=.7] at (9.5,1.5){$3$};
\node[scale=.7] at (9.5,2.5){$4$};
\node[scale=.7] at (10.5,2.5){$5$};

\end{tikzpicture}
\]

In the second case:

\[
\begin{tikzpicture}[scale=.4]
\node at (4,7){$\calg_{\dot \gamma}$};
\draw (0,0)--(4,0)--(4,1)--(0,1)--(0,0) (1,0)--(1,1) (2,0)--(2,1)  (3,0)--(3,2)--(5,2)--(5,1)--(4,1)--(4,2);
\node[scale=.7] at (.5,.5){$b$};
\node[scale=.7] at (1.5,.5){$c$};
\node[scale=.7] at (2.5,.5){$3$};
\node[scale=.7] at (3.5,.5){$4$};
\node[scale=.7] at (3.5,1.5){$5$};
\node[scale=.7] at (4.5,1.5){$6$};
\node at (5.5,2.5){$\Ddots$};
\draw (6,3)--(7,3)--(7,5)--(6,5)--(6,3) (6,5)--(6,6)--(7,6)--(7,5) (6,4)--(7,4);
\node[scale=.7] at (6.5,3.5){$k\!\!-\!\!1$};
\node[scale=.7] at (6.5,4.5){$k$};
\node[scale=.7] at (6.5,5.5){$k\! \!+\! \! 1$};
\node at (7.5,6.5){$\Ddots$};

\node at (14.5,7){$\calg_{\gamma}$};
\draw (14,4)--(15,4)--(15,5)--(14,5)--(14,4);
\node[scale=.7] at (14.5,4.5){$k\! \!+\! \! 1$};
\node at (15.5,5.5){$\Ddots$};
\node[scale=.7] at (15.5,2){Delete tiles $b,c$ and everything};
\node[scale=.7] at (15.5,1){up to the next sign change};
\end{tikzpicture}
\]

In the third case:
\[
\begin{tikzpicture}[scale=.4]
\node at (4,9){$\calg_{\dot \gamma}$};

\draw (0,0)--(2,0)--(2,1)--(0,1)--(0,0) (1,0)--(1,1);
\node[scale=.7] at (.5,.5){$\ell'$};
\node[scale=.7] at (1.5,.5){$\ell$};
\node at (2.5,1.5){$\Ddots$};
\draw (3,2)--(4,2)--(4,4)--(3,4)--(3,2) (3,3)--(6,3)--(6,4)--(3,4)--(3,3) (5,3)--(5,7)--(6,7)--(6,4) (5,5)--(6,5) (5,6)--(6,6);
\node at (6.5,7.5){$\Ddots$};
\draw (7,8)--(9,8)--(9,9)--(7,9)--(7,8) (8,8)--(8,9);

\node[scale=.7] at (3.5,2.5){$4$};
\node[scale=.7] at (3.5,3.5){$3$};
\node[scale=.7] at (4.5,3.5){$a$};
\node[scale=.7] at (5.5,3.5){$b$};
\node[scale=.7] at (5.5,4.5){$c$};
\node[scale=.7] at (5.5,5.5){$3$};
\node[scale=.7] at (5.5,6.5){$4$};
\node[scale=.7] at (7.5,8.5){$\ell$};
\node[scale=.7] at (8.5,8.5){$\ell''$};

\node at (14.5,7){$\calg_{\gamma}$};
\draw (14,4)--(15,4)--(15,6)--(14,6)--(14,4) (14,5)--(15,5);
\node[scale=.7] at (14.5,4.5){$\ell'$};
\node[scale=.7] at (14.5,5.5){$\ell''$};

\node[scale=.7] at (15.5,2){Delete tiles $a,b,c$ and then cancel};
\node[scale=.7] at (15.5,1){tiles with the same label};
\end{tikzpicture}
\]

\subsection{Crossing arcs and crossing snake graphs}
{In this subsection, we show that the notions of crossings for arcs and snake graphs coincide}.
\begin{thm}\label{thm cross} Let $\gamma_1, \gamma_2$ be generalized arcs and $\calg_1, \calg_2$ their corresponding {labeled} snake graphs (might have self-crossings.)
\begin{itemize}
\item[a)] $\gamma_1, \gamma_2$ cross with a nonempty local overlap $(\tau_{i_s}, \cdots, \tau_{i_t})=(\tau_{i'_{s'}}, \cdots, \tau_{i'_{t'}})$ if and only if $\calg_1, \calg_2$ cross in $\calg_1[s,t] \cong \calg_2[s',t'].$
\item[b)] $\gamma$ has a self-crossing with a nonempty local overlap $(\tau_{i_s}, \cdots, \tau_{i_t})=(\tau_{i_{s'}}, \cdots, \tau_{i_{t'}})$ (with or without self-intersection) if and only if $\calg_{\gamma}$ has a crossing self-overlap on $\calg_{\gamma}[s,t] \cong \calg_{\gamma}[s',t'].$
\end{itemize}
\end{thm}

\begin{proof} 
a) This follows directly from Theorem 5.3 of \cite{CS}.\\
b) Let $\gamma$ be a self-crossing arc. {Choose a parametrization $\zg=\zg(t)$}, and say the self-crossing occurs at the times $t_1$ and $t_2.$ Take now two copies $\gamma_1, \gamma_2$ of $\gamma$ and consider their crossing at $\gamma_1(t_1)=\gamma_2(t_2).$ The local overlap of $\gamma_1$ and $\gamma_2$ at this crossing is the same as the local overlap of the self-crossing, and the local overlap in the corresponding snake graphs $\calg_1, \calg_2$ of $\gamma_1, \gamma_2$ is the same as the local self-overlap of the snake graph $\calg_{\gamma}.$ Now the result follows from part a).
\end{proof}

\subsection{Smoothing crossings and resolving snake graphs}
{In this subsection, we show that the smoothing operation for arcs corresponds to the resolution of crossings for snake graphs.}
The following result has been shown in \cite{CS}.
\begin{thm}\cite[Theorem 5.4]{CS} \label{smoothing1} Let $\gamma_1$ and $\gamma_2$ be two (generalized) arcs which cross with a non-empty local overlap, and let $\calg_1$ and $\calg_2$ be the corresponding {labeled} snake graphs with overlap $\calg$. Then the {labeled} snake graphs of the four arcs obtained by smoothing the crossing of $\gamma_1$ and $\gamma_2$ in the overlap are given by the resolution $\re 12$ of the crossing of the labeled snake graphs $\calg_1$ and $\calg_2$ at the overlap $\calg.$
\end{thm}

We now show the analogous statement for self-crossings.
\begin{thm}\label{thm430} Let $\gamma_1$ be a self-crossing arc with nonempty local overlap and let $\calg_1$ be the corresponding {labeled} snake graph with crossing overlap $i_1(\calg)=\calg_1[s,t]$ and $i_2(\calg)=\calg_1[s',t'].$ Then the {labeled} snake graphs of the two arcs and the {labeled} band graph of the loop obtained by smoothing the crossing of $\gamma_1$ on the overlap are given by the resolution $\res_{\calg}(\calg_1)$ of the self-crossing of the labeled snake graph $\calg_1$ at the overlap $\calg.$
\end{thm}

\begin{proof}
{
We start with the case where the overlap is in the same direction and $s'\ge t+1$. 
Let $\Delta$ be the triangle 
in the surface which contains the segment of $\gamma_1$ between the $t$-th and the $(t+1)$-st crossing point. We introduce a puncture $p$ on this segment of $\zg_1$ and in the interior of $\zD$, see Figure~\ref{fig puncture}, {where the triangulation arcs are black and the arc $\zg$ is red.} We complete $T$  to a triangulation $\dot T$ by adding three arcs $a,b,c$ from the puncture $p$ to the vertices of $\zD$. 
Let $\gamma_{11}$ be the segment of $\gamma_1$ up to the puncture $p$, and let $ \gamma_{12}$ the segment of $\gamma_1$ after   $p$. 
\begin{figure}\begin{center}
\scalebox{0.8}{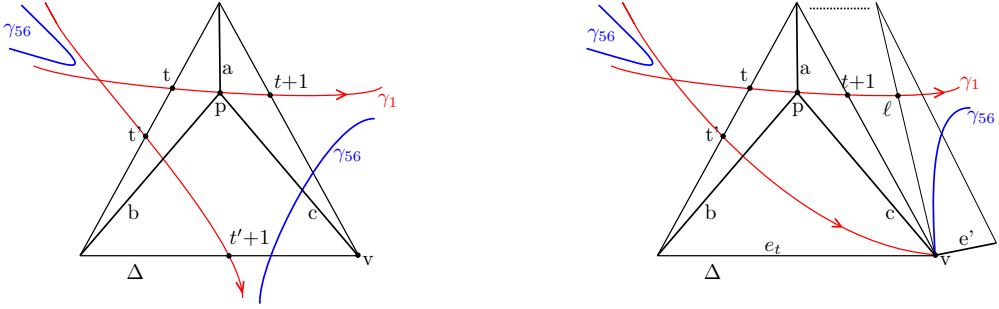}
\caption{Introducing a puncture}\label{fig puncture}\end{center}
\end{figure}

On the other hand, consider the snake graphs $\calg_{11}=\calg_1[1,t]$ and $\calg_{12}=\calg_1[t+1,d].$
Observe that these snake graphs do not necessarily correspond to snake graphs of $\gamma_{11}$ and $\gamma_{12}$, since the arc $\gamma$ might run through the triangle $\Delta$ several times, and introducing a puncture in the surface might create crossings with the new arcs in $\dot T \backslash T$. Then the snake graphs $\widetilde{\calg}_{11}, \widetilde{\calg}_{12}$ corresponding to $\gamma_{11}$ and $\gamma_{12}$ will be obtained from $\calg_{11}$ and $\calg_{12},$ respectively, by inserting single tiles which correspond to these new crossings.

Suppose first that $t'<d$. 
Then the triangle $\zD$ has sides  $\tau_{i_t}=\tau_{i_{t'}},\tau_{i_{t+1}},\tau_{i_{t'+1}}$ and $i_{t+1}\ne i_{t'+1}$, since we have an overlap.
Smoothing the self-crossing of $\zg_1$ is the same as smoothing the corresponding crossing of the two curves $\zg_{11}$ and $\zg_{12}$ and then removing the puncture again. This will produce a pair $(\zg_3,\zg_4)$, where $\zg_3 $ is an arc and $\zg_4$ is a loop, and an arc $\zg_{56}$ which crosses $\tau_{i_{t+1}}$ and $\tau_{i_{t'+1}}$. {In Figure~\ref{fig puncture}, the arc $\zg_{56}$ is the blue one. }In terms of snake graphs, $\res_{\calg}(\calg_1)$ is obtained from $\res_{\calg}(\widetilde{\calg}_{11}, \widetilde{\calg}_{12} )$ by removing the tiles corresponding to the crossings with the arcs at the puncture and glueing. The graphs $\calg_3$ and $\calg_4^\circ$ are glued along the edge labeled $\tau_{i_{t'+1}}$ and correspond to $\zg_3$ and $\zg_4$. The graph $\calg_{56}$ is glued along the edge labeled $\tau_{i_t}$ and corresponds  to $\zg_{56}$.

Now suppose that $t'=d$. 
Then the triangle $\zD$ has sides  $\tau_{i_t}=\tau_{i_{t'}},\tau_{i_{t+1}},\tau_{e_t}$, {see the right hand side of Figure~\ref{fig puncture}}. For  $\zg_3$ and $\zg_4 $ the proof is exactly as above. For $\zg_{56}$, there is a slight difference. Now smoothing the crossing of the two arcs $\zg_{11}$ and $\zg_{12}$ will produce an arc $\zg_5$ starting at $p$ and then crossing $\tau_{i_{t+1}}$, as well as an arc $\zg_6=c$ of the triangulation $\dot T\setminus T$ from the puncture $p$ to the common endpoint $v$ of $\tau_{i_{t+1}}$ and $\tau_{e_t}$. Removing the puncture will then produce an arc $\zg_{56}$  by glueing the two arcs $\zg_5$ and $\zg_6$. Thus $\zg_{56}$ starts at $v$, follows $c$ and then follows $\zg_{12}$, but because of isotopy, this arc will not cross the arcs in the fan of arcs $\tau_{i_{t+1}},\tau_{i_{t+2}},\ldots,\tau_{i_{\ell}}$ that are crossed by $\zg_{12}$ and incident to $v$. This situation is exactly reflected in the definition of the snake graph $\calg_{56}=\calg_{56}'\setminus\suc(e')$, since $e'$ is the last interior edge such that $f_{56}(e')=f_{56}(e_t)$.

When the overlap is in the opposite direction, the proof is similar.}

It remains the case where the overlap is in the same direction and $s'\le t$. Thus the self-overlap has an intersection $\calg_1[s',t]$.

Let $\tau_{i_1},\tau_{i_{2}},\ldots,\tau_{i_s}$ be the sequence of arcs of the triangulation crossed by $\gamma_1$ in order. 
By the definition of overlap, we have that the sequences $\tau_{i_s},\tau_{i_{s+1}},\ldots,\tau_{i_t}$ and
$\tau_{i_{s'}},\tau_{i_{s'+1}},\ldots,\tau_{i_{t'}}$
are  equal or opposite to each other. Since $s'<t$ it follows that they have to be equal, since otherwise the segment of $\gamma$ crossing $\tau_{i_{s'+1}},\tau_{i_{s'+2}},\ldots,\tau_{i_{t}}$ would be isotopic to a curve not crossing these arcs at all, see Figure \ref{homotopy2}.

\begin{figure}\begin{center}
\scalebox{1}{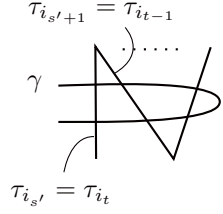}
\caption{Proof of Theorem \ref{thm430}}\label{homotopy2}\end{center}
\end{figure}

Let $\gamma[s,t]$ and $\gamma[s',t']$ be the segments of $\gamma$ corresponding to the overlaps and let $\gamma[s',t]$ be their common subsegment corresponding to the intersection of the overlaps.

Since the 
sequences $\tau_{i_s},\tau_{i_{s+1}},\ldots,\tau_{i_t}$ and
$\tau_{i_{s'}},\tau_{i_{s'+1}},\ldots,\tau_{i_{t'}}$
are  equal, it follows that the curves $\gamma[s,t]$ and $\gamma[s',t']$ run parallel before and after their crossing at $p$. Moreover, since $s'<t$, the following sequences  are equal as well:
\[\begin{array}
 {ll}
 \tau_{i_s},\tau_{i_{s+1}},\ldots,\tau_{i_{s'-1}}\\
 \tau_{i_{s'}},\tau_{i_{s'+1}},\ldots,\tau_{i_{2s'-s-1}}\\
 \tau_{i_{2s'-s}},\tau_{i_{2s'-s+1}},\ldots,\tau_{i_{3s'-2s-1}}\\
 \ldots \tau_{i_{t}}\\
\end{array}
\]
Thus the curve $\gamma[s,s']$ after identifying its endpoints is a closed and non-contractible curve. This implies that the curve $\gamma[s,t]$ is of the form as in Figure \ref{fig shape}, where points with equal labels $a,b,c,d,e$ are identified. The crossing point $p$ can be any of the points labeled 1,2,3,4,5. For example, if $p$ is the point labeled 5,4,3,2,1 respectively, then the crossing point $s'$ at the beginning of the second overlap must be the point on $\tau_{i_s}$ crossed by $\gamma$ after the point $e,d,c,b,a$ respectively, and the crossing point $t$  at the end of the first overlap must be  the point on $\tau_{i_{t'}}$ first crossed by $\gamma$ after passing through the point $a,b,c,d,e$ respectively.

\begin{figure}\begin{center}
\scalebox{1}{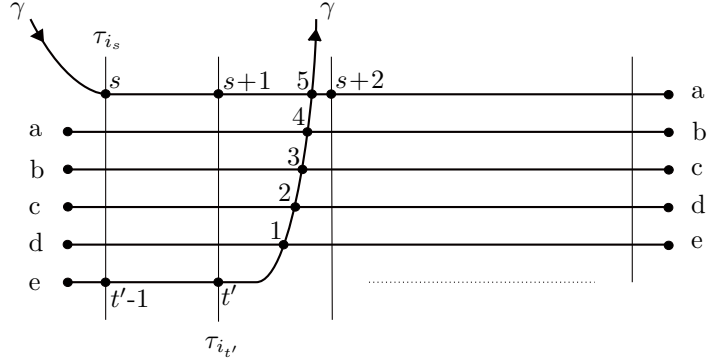}
\caption{Proof of Theorem\ref{thm430}.
This situation can arise in any non-simply connected surface. The curve $\gamma$ between the points 5 and 5 is the concatenation of an essential loop $\zeta$ with itself 5 times, thus $\textup{Brac}_5\zeta$.}\label{fig shape}\end{center}
\end{figure}

Therefore the condition $s'\le t$ implies that in the example in Figure \ref{fig shape} the point $p$ must be  the point 2 or 1.
We now study the smoothing of these self-crossings.

If $p=1$  then the smoothing at $p$ will produce the two multicurves $$\{\zeta, \textup{generalized arc with 4 self-crossings at 2,3,4,5}\},$$
 and
  $$\{ \textup{generalized arc with  self-crossings at 3,4,5 and a kink at 2} \}.$$

If $p=2$ then the smoothing at $p$ will produce the two multicurves 
$$\{\textup{Brac}_1\zeta, \textup{generalized arc with 3 self-crossings at 3,4,5}\},$$
 and
  $$\{ \textup{generalized arc with  self-crossings at 4,5 and a kink at 3} \}.$$
  
Again we conclude that the snake graphs of the arcs (and bands) obtained by smoothing the self-crossing of $\gamma$ are given by $\res_{\calg}(\calg_1).$
\end{proof}

So far, we have considered two arcs which cross with a non-empty local overlap. Now we study two arcs which cross with an empty local overlap. The following result has been shown in \cite{CS}.

\begin{thm}\cite[Theorem 5.7]{CS} \label{smoothing2} Let $\gamma_1$ and $\gamma_2$ be two arcs which cross in a triangle $\Delta$ with an empty local overlap, and let $\calg_1$ and $\calg_2$ be the corresponding snake graphs. Assume $\Delta=\Delta'_0$ is the first triangle $\gamma_2$ meets. Then the snake graphs of the four arcs obtained by smoothing the crossing of $\gamma_1$ and $\gamma_2$ in $\Delta$ are given by the resolution $\gt s{\delta_3}12$ of the grafting of $\calg_2$ on $\calg_1$ in $G_s,$ where $0 \leq s \leq d$ is such that $\Delta=\Delta_s$ and if $s=0$ or $s=d$ then $\delta_3$ is the unique side of $\Delta$ that is not crossed by neither $\gamma_1$ nor $\gamma_2.$
\end{thm}

For self-crossing arcs with empty local overlap, we have the following result, see Figure~\ref{fig selfgrafting}.

\begin{thm}\label{smoothingselfgrafting} Let $\gamma_1$  be a generalized arc which has a self-crossing in a triangle $\Delta$ with an empty local overlap, and let $\calg_1$ be the corresponding snake graph. Thus $\Delta=\Delta_0$ is the first triangle $\gamma_1$ meets and  $\Delta=\Delta_s$ is met again after $s$ crossings. Then the snake graphs of the {two arcs and the band graph of the loop} obtained by smoothing the self-crossing of $\gamma_1$ in $\Delta$ are given by the resolution $\graft_{s,\e_3} (\calg_1)$ of the self-grafting of $\calg_1$  in $G_s,$ and if  $s=d$ then $\delta_3$ is the unique side of $\Delta$ that is not crossed by  {$\gamma_1$.}
\end{thm}
\begin{proof}
As in the proof of Theorem \ref{thm430}, we can introduce a puncture on the segment  of $\gamma_1$
between the two crossing points and complete to a triangulation. Then the two segments of $\gamma_1$ before and after the puncture still have the same crossing. We can use Theorem \ref{smoothing2} to resolve that crossing and then remove the puncture to get the desired resolution.
\end{proof}

\begin{figure}
\begin{center}
\scalebox{0.8}{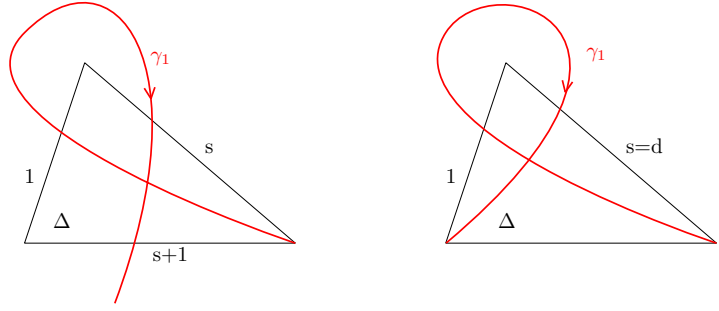}
\caption{Generalized arcs crossing with an empty overlap. On the left $s<d$, and  on the right  $s=d$.}
\label{fig selfgrafting}
\end{center}
\end{figure}


\section{Snake graph calculus for cluster algebras}\label{sect 7}
{In this section, we show that we can use snake graph calculus to make explicit computations in the cluster algebras from unpunctured surfaces. In particular, we give a new proof of the skein relations.}
\subsection{Non-empty overlaps}

If $\calg$ is a snake graph associated to an arc $\gamma$ in a triangulated surface $(S,M,T)$ then each tile of $\calg$ corresponds to a quadrilateral in the triangulation $T,$ and we denote by $\tau_{i(G)}\in T$ the diagonal of that quadrilateral. With this notation we define
\begin{align*}
 x(\calg) = \prod_{ G \mbox{ tile in } \calg} x_{i(G)}\\
 y(\calg) = \prod_{ G \mbox{ tile in } \calg} y_{i(G)}
\end{align*}
If $\calg=\{\tau\}$ consists of a single edge, we let  $x(\calg)=1$ and $y(\calg)=1$.

Let $\gamma_1$ and $\gamma_2$ be two arcs which cross with a non-empty  overlap. Let $x_{\gamma_1}$ and $x_{\gamma_2}$ be the corresponding cluster variables and $\calg_1$ and $\calg_2$ the snake graphs with corresponding overlap $\calg.$ 
Recall that $\re 12$ consists of two pairs $(\calg_3,\calg_4)$ and $(\calg_5,\calg_6)$ of snake graphs. 
The number of tiles in $(\calg_3,\calg_4)$ is equal to the number of tiles in $(\calg_1,\calg_2)$, whereas the number of tiles in $(\calg_5,\calg_6)$ is strictly smaller, since this pair does not contain the overlaps.

Define 
$\tcalg$  to be the union of all tiles in $\calg_1 \sqcup \calg_2$ which are not in $\calg_5 \sqcup \calg_6.$

Similarly, if $\gamma_1$ is a self-crossing arc with non-empty local overlap, let $x_{\gamma_1}$ be the corresponding Laurent polynomial and $\calg_1$ be  the snake graph with corresponding self-overlap $\calg.$ In this situation,  if the self-overlap is in the same direction, then the resolution of the crossing consists of a pair $(\calg_3,\calg_4^\circ)$ of a snake and a band graph and a snake graph $\calg_{56}$;
{whereas if the self-overlap is in the opposite direction, then  the resolution of the crossing consists of a snake graph $\calg_{34}$ and a pair $(\calg_5,\calg_6^\circ)$ of a snake and a band graph. If the self-overlap is in the same direction,} the number of tiles in $(\calg_3,\calg_4^\circ)$ is equal to the number of tiles in $\calg_1$, whereas the number of tiles in $\calg_{56}$ is strictly smaller. {On the other hand, if the overlap is in the opposite direction, then the number of tiles in $\calg_{34}$ is equal to the number of tiles in $\calg_1$, whereas the number of tiles in $\calg_5\cup\calg_6^\circ$ is strictly smaller.}

Also in this case, define
$\tcalg$  to be the union of all tiles in $\calg_1 $ which are not in $\calg_{56}$ {or $\calg_5\sqcup\calg_6^\circ$.}

In all cases, under the bijections of section \ref{sect 3} the matchings $P_{56} $ of $\calg_5\sqcup\calg_6$ (respectively $\calg_{56}$ {or $\calg_5\sqcup\calg_6^\circ$}) are completed to matchings of $\calg_1\sqcup\calg_2$ (respectively $\calg_1$ or $\calg_3\sqcup\calg_4^\circ $) in   a unique way which does not depend on $P_{56}$.  
Moreover, the $y$-monomial of the completion is maximal on a connected subgraph of $\tcalg$ and trivial on its complement.
We  denote by $\tcalg_{max}$ the component on which the $y$ monomial is maximal.

 Let $\re 12$ be the resolution of the crossing of $\calg_1$ and $\calg_2$ at $\calg$ and  $\ree 1$ the resolution of the self-crossing of $\calg_1$  at $\calg$. 
Define the {\em Laurent polynomial of the resolutions} by

\begin{equation}\label{laurent12}
 \call (\re 12) = \call (\calg_3 \sqcup \calg_4) + y( {\tcalg_{max}}) \call (\calg_5 \sqcup \calg_6),
\end{equation}
and

\begin{equation} \label{laurentself}
 \call (\ree 1) = \left\{
\begin{array}
 {ll}
\call (\calg_3 \sqcup \calg_4^\circ) + y( {\tcalg_{max}}) \call (\calg_{56}), &\textup{if the self-overlap is in the same direction;}\\ \\
\call (\calg_{34}) + y( {\tcalg_{max}}) \call (\calg_{5}\sqcup\calg_{6})&\textup{if the self-overlap is in the opposite direction,}\\
\end{array}\right.
\end{equation}
where 

\begin{equation*}
 \call (\calg_k ) = \frac{1}{x(\calg_k)} \sum_{P \in \match(\calg_k)} x(P) y(P),  \quad\textup{if $\calg_k$ is positive;}
\end{equation*}
and 
\[ \call(\calg_k)=-\call(-\calg_k), \quad\textup{if $\calg_k$ is negative.}
\]
\begin{thm} \label{laurent}
\begin{enumerate}
\item 
 Let $\gamma_1$ and $\gamma_2$ be two arcs which cross with a non-empty local overlap and let $\calg_1$ and $\calg_2$ be the corresponding snake graphs with local overlap $\calg.$ Then
 
\begin{equation*}
 \call(\calg_1 \sqcup \calg_2) = \call (\re 12)
\end{equation*}

\item
 Let $\gamma_1$ be a self-crossing generalized  arc   a non-empty local overlap and let $\calg_1$ be the corresponding snake graph with local overlap $\calg.$ Then
 
\begin{equation*}
 \call(\calg_1 ) = \call (\ree 1)
\end{equation*}

\end{enumerate}\end{thm}

\begin{proof} (1) This is \cite[Theorem 6.1]{CS}.  The essential step of the proof is to show that the switching operation of section \ref{switching} is weight preserving. {That is, if $\calg$ is a (union of) labeled snake and band graphs coming from an unpunctured surface,  $P\in\match \calg$, and $P'$ is obtained from $P$ by a switching operation, then $x(P)=x(P')$ and $y(P)=y(P')$.} 
Then, since the bijection $\varphi$ of section~\ref{sect 3.3} is defined using switching and restriction, it is also weight preserving. To finish the proof one needs to take care of the missing tiles in $\calg_5\sqcup\calg_6$ and show that the $y(\tcalg_{max})$ is absorbing this discrepancy. 

(2) As usual we use the notation $\calg_{1}[s,t] \cong \calg_{1}[s',t']$
for the overlap. In the case where $s'> t+1$, the proof is the exact analogue of the proof of (1).

If $s'=t+1$, we have two cases. Either the second overlaps $j_1(\calh)$ and $j_2(\calh)$ from the definition of the resolution in section \ref{sect resolution} cross or not. 
If they cross then the proof is the same as in the case $s'\le t$ below.
If they do not cross, then the proof is exactly analogue to the proof of (1), except that we need to check that  the bijection is weight preserving for the one special type of matching $ P_{\tiny\smiley}$ defined in section \ref{sect s'=t+1}.
Observe that on each piece of $\calg_3$, the definition of $\varphi(P_{\tiny\smiley})$ is given by restricting $P_{\tiny\smiley}$ to that piece. Moreover, on $\calg_4^\circ$, $\varphi(P_{\tiny\smiley})$ is also given by restricting $P_{\tiny\smiley}$ except for one edge, namely the boundary edge of $G_{s-1}^{N\!E}$ is mapped to the edge of $\calg_4^\circ$ that corresponds to the edge $e_{t'}$ of $\calg_1$. 

Thus
in order to show that $x(\varphi(P_{\tiny\smiley})) =x(P_{\tiny\smiley})$  
it suffices to show that the boundary edge of $G_{s-1}^{N\!E}$ has the same weight as the edge $e_{t'}$. 
Recall that $i_s=i_{s'}$, $i_t=i_{t'}$ and $i_{s-1}\ne i_{s'-1}$, $i_{t+1}\ne i_{t'+1}$, since we have an overlap.
Moreover, since we are in the case $s'=t+1$, it follows that $i_{s-1}\ne i_t$ and $i_s\ne i_{t'+1}$, and that $\tau_{i_s}$ and $\tau_{i_t}=\tau_{i_{t'}}$ are two distinct sides of a triangle $\Delta$ in the triangulation. Therefore $\tau_{i_{s-1}}=\tau_{i_{t'+1}}$ is the third side of $\Delta$.
Now the edge $e_{t'}$ is the edge shared by the tiles $G_{t'}$ and $G_{t'+1}$ and therefore its weight is given by the side of $\Delta$ different from $\tau_{i_{t'}}$ and $\tau_{i_{t'+1}}$, thus the weight of $e_{t'}$ is equal to $i_s$. On the other hand, the weight of the boundary edge of $G_{s-1}^{N\!E}$ is also equal to $i_s$ since $G_s$ is following $G_{s-1}$ in $\calg_1$.

To show that  $y(\varphi(P_{\tiny\smiley})) =y(P_{\tiny\smiley})$, observe that 
$y(\varphi(P_{\tiny\smiley})|_{\calg_4^\circ}) =y(P_{\tiny\smiley}|_{\calg_1[s',t']})$ by definition and that 
$y(\varphi(P_{\tiny\smiley})|_{\calg_3} =y(P_{\tiny\smiley}|_{\calg_1\setminus\calg_1[s',t'])}$ because even if the edges  on the parts that come from $j_1(\calh)$ and $j_2(\calh)$ are swapped, together they still create the same contribution to the $y$-monomial.

Suppose now that $s'\le t$.
Let 
\[
\begin{array}
 {rcl}
{\varphi } : \match \calg_1 \cup \match \calg_{56}&\longrightarrow 
& \match \calg_3\times \match \calg_4^\circ 
\\
P&\longmapsto &( \varphi_3(p),\varphi_4(p)) 
\end{array}\]
be the bijection defined in section \ref{sect 3.4}.

 First we note that
\begin{equation} \label{eq1}
  x(\calg_1 ) = x(\calg_3 \sqcup \calg_4^\circ) = x(\calg_{56}) x(\tcalg),
\end{equation}
where the first identity holds because $\calg_1 $ and $\calg_3 \sqcup \calg_4^\circ$ have the same set of tiles, and the second identity holds because $\tcalg$ consists of the tiles of $\calg_1  \backslash \calg_{56}$.
Using the equation (\ref{eq1}) on the definition of $\call(\calg_1)$ as well as the bijection $\varphi$, we obtain
\begin{align*}
 \call(\calg_3\sqcup\calg_4^\circ) = &\displaystyle \frac{1}{x(\calg_1)} \sum_{P \in \match
\calg_1} x(\varphi(P))y(\varphi(P))\\&+\displaystyle \frac{1}{x(\tcalg)x(\calg_{56})} \sum_{P\in \match\calg_{56}}  x(\varphi(P))y(\varphi(P)),
\end{align*}
and therefore it suffices to show the following lemma.
\end{proof}

\begin{lem} \label{lem 72} Let $s'\le t$.
\begin{itemize}
 \item[(a)] $x(\varphi(P))\ y(\varphi(P))=x(P) y(P)$ if $P\in \Match(\calg_1)$.
 \item[(b)] $x(\varphi(P))\ y(\varphi(P))= x(\tcalg) y(\tcalg_{max}) x(P)y(P) $ if $P\in \match(\calg_{56})$.
\end{itemize}
\end{lem}

\begin{proof}
 (a) This is \cite[Lemma 6.2]{CS}.\\
(b)  Since $s'<t$, we see that $\tcalg$ consists of two copies of each tile in $\calg_4^\circ$. By definition
of $\varphi$ in this case, we have
\begin{equation} \label{eq2}
\varphi(P) = P\cup P_-\cup P_+
\end{equation}
where $P_-$ is the minimal matching of $\calg_4^\circ$ and $P_+$ is its maximal matching, since $\varphi(P)$ is extended to a matching on $\calg_1$ using boundary edges that are complementary on $\calg_1[s, s'-1]$ and $\calg_1[t+1, t']$.
Now $x(P_-)x(P_+)=x(\calg_4^\circ)x(\calg_4^\circ)$
and thus
\[x(\varphi(P))=x(P)x(P_-)x(P_+)=x(P)(x(\calg_4^\circ))^2=x(P)x(\tcalg).\]

Moreover, $\tcalg_{max}=\calg_4^\circ$ and thus
\[y(\varphi(P))=y(P)y(P_+)=y(P)y(\calg_4^\circ)=y(P)y(\tcalg_{max}).\]
\end{proof}

\subsection{Empty overlaps}
\subsubsection{Two arcs crossing} Now let $\gamma_1$ and $\gamma_2$ be two arcs which cross in a triangle $\Delta$ with an empty overlap. We may assume without loss of generality that $\Delta$ is the first triangle $\gamma_2$ meets. Let $x_{\gamma_1}$ and $x_{\gamma_2}$ be the corresponding cluster variables and $\calg_1$ and $\calg_2$ be their associated snake graphs, respectively. We know from Theorem \ref{smoothing2} that the snake graphs of the arcs obtained by smoothing the crossing of $\gamma_1$ and $\gamma_2$ are given by the resolution $\gt s{\e_3}12$ of the grafting of $\calg_2$ on $\calg_1$ in $G_s,$ where $s$ is such that $\Delta=\Delta_s$ is the triangle $\gamma_1 $ meets after its $s$-th crossing point, and, if $s=0,$ then $\e_3$ is the unique side of $\Delta$ which is not crossed neither by $\gamma_1$ nor $\gamma_2.$

The edge of $G_s$ which is the glueing edge for the grafting is called the {\em grafting edge}. We say that the grafting edge is {\em minimal in $\calg_1$} if it belongs to the minimal matching on $\calg_1.$

Recall that $\gt s {\e_3}12$ is a pair $(\calg_3 \sqcup \calg_4), (\calg_5 \sqcup \calg_6).$ Let
$\tcalg_{34}$ be the union of all tiles in $\calg_1\sqcup\calg_2$ that are not in $\calg_3\sqcup\calg_4$ and
let
$\tcalg_{56}$ be the union of all tiles in $\calg_1\sqcup\calg_2$ that are not in $\calg_5\sqcup\calg_6$.
Define

\begin{align*}
 \call (\gt s{\e_3}12) = y_{34} \call (\calg_3 \sqcup \calg_4) + y_{56} \call (\calg_5 \sqcup \calg_6),
\end{align*}
where

\begin{align} \label{skeincoeff}
\left\{\begin{array}{llll}
 y_{34}=1& \textup{and} &y_{56} = y(\tcalg_{56})
 & \mbox{ if the grafting edge is minimal in $\calg_1$; } \\
  y_{34} = y(\tcalg_{34})
  &\textup{and} &y_{56}=1 
  & \mbox{ otherwise. } 
\end{array}\right.
\end{align}
The following result has been shown in \cite{CS}.
\begin{thm}\cite[Theorem 6.3]{CS} \label{laurent2}
 With the notation above, we have
\begin{align*}
 \la 12 = \call (\gt s{\e_3}12).
\end{align*}
\end{thm}

\subsubsection{Self-crossing} 
Similarly, if $\gamma_1$ is a generalized arc which self-crosses in a triangle $\Delta$ with an empty overlap, let $\calg_1$ be the associated snake graph and $x_{\gamma_1}$ be the corresponding Laurent polynomial. We know from Theorem \ref{smoothingselfgrafting} that the snake graphs of the arcs obtained by smoothing the self-crossing of $\gamma_1$ are given by the resolution $\graft _{s,{\e_3}}(\calg_1)$ of the self-grafting of $\calg_1$  in $G_s,$ where $s$ is such that $\Delta=\Delta_s$ and,  if $s=d,$ then $\e_3$ is the unique side of $\Delta$ which is not crossed by $\gamma_1$.

Let $\tcalg_{34}$ be the union of tiles in $\calg_1$ that are not in $\calg_3
\sqcup\calg^{\circ}_4$ and $\tcalg_{56}$ be the union of tiles in $\calg_1$ that are not in $\calg_{56}.$

If $s<d$,  then  $\e_3 $ is  the north or the east edge in $G_s$, and   we let 
\begin{equation}\label{eq722a}
y_{34} =\left\{\begin{array}{ll}  y(\tcalg_{34})  &\textup{if $\delta_3$ is maximal in $\calg_1$;}\\
           1& \textup{otherwise;}\end{array}\right.
\qquad
y_{56}= \left\{\begin{array}{ll}1 &\textup{if $\delta_3$ is maximal in $\calg_1$}\\
           y(\tcalg_{56}) &\textup{otherwise;}\end{array}\right.
\end{equation}

If $s=d$, then $\calg_1$ and $\calg_{3}\sqcup\calg^{\circ}_4$ have the same tiles, so 
$y_{34}=1.$
On the other hand,  $\calg_1 \setminus\calg_{56}$ has two components $\calg_1[1,k'-1]$ containing the glueing edge $\delta_3'$ and  $\calg_1[k+1,d]$ containing the glueing edge  $\delta_3$. We let
\begin{equation}\label{eq722b}y_{34}=1\qquad y_{56}=y_{56}' y_{56}'' ,\end{equation}
where
\[y_{56}'=\left\{\begin{array}{ll}  y(\calg_1[1\,,k'\!-\!1]) &\textup{if  $\delta_3'$ is minimal in $\calg_1$};\\
           1&\textup{otherwise;}\end{array}\right.\]
\[y_{56}''= \left\{\begin{array}{ll}y(\calg_1[k+1\,,d]) &\textup{if  $\delta_3$ is minimal in $\calg_1$};\\
          1 & \textup{otherwise. }\end{array}\right.\]
%
With this notation define

\begin{align*}
 \call (\graft _{s,{\e_3}}(\calg_1)) = y_{34} \call (\calg_3 \sqcup \calg_4^\circ) + y_{56} \call (\calg_{56}).
\end{align*}

\begin{thm} \label{laurent2self}
 With the notation above, we have
\begin{align*}
 \call (\calg_1) = 
 \call (\graft _{s,{\e_3}}(\calg_1)) 
\end{align*}
\end{thm}

\begin{proof}
As in the proof of Theorem \ref{thm430}, we can introduce a puncture on the segment  of $\gamma_1$
between the two crossing points and complete to a triangulation. Then the two segments of $\gamma_1$ before and after the puncture still have the same crossing with empty overlap. Applying Theorem \ref{laurent2} to this situation and then removing the puncture will complete the proof.
\end{proof}

\subsection{Skein relations}\label{sect skein}
As a corollary we obtain a new proof of the skein relations.
\begin{cor} \label{skein}
Let $C,$ $ C_{+}$, and $C_{-}$ be as in Definition \ref{def:smoothing} but not including any closed loops. 
Then we have the following identity in the cluster algebra $\cala$,
\begin{equation*}
x_C = \pm Y_1 x_{C_+} \pm Y_2 x_{C_-}.
\end{equation*}
Moreover the coefficients $Y_1$ and $Y_2$ are given by
\textup{(\ref{laurent12}), (\ref{laurentself}), (\ref{skeincoeff}), (\ref{eq722a}) and (\ref{eq722b}).}\qed
\end{cor}

\section{An example}
{Let $(S,M)$ be the torus with one boundary and one marked point, and consider the two arcs $\gamma_1$ (in black) and $\gamma_2$ (in red) shown in Figure~\ref{torus}. In the left picture, the arcs are drawn directly on the torus and in the right picture, the arcs are drawn in the standard covering of the torus. In the picture on the right, we also have fixed a triangulation whose arcs carry the labels 1,2,3, and 4. 

\begin{figure}[h]
\begin{center}
\scalebox{0.9}{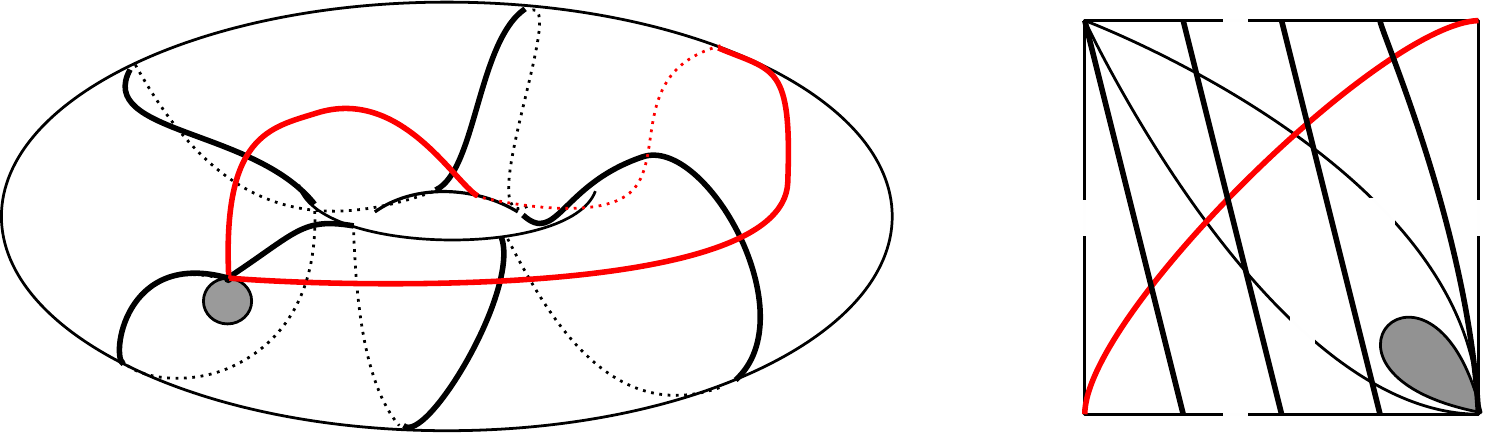}
\caption{Two arcs on the torus}
\label{torus}
\end{center}
\begin{center}
\scalebox{0.9}{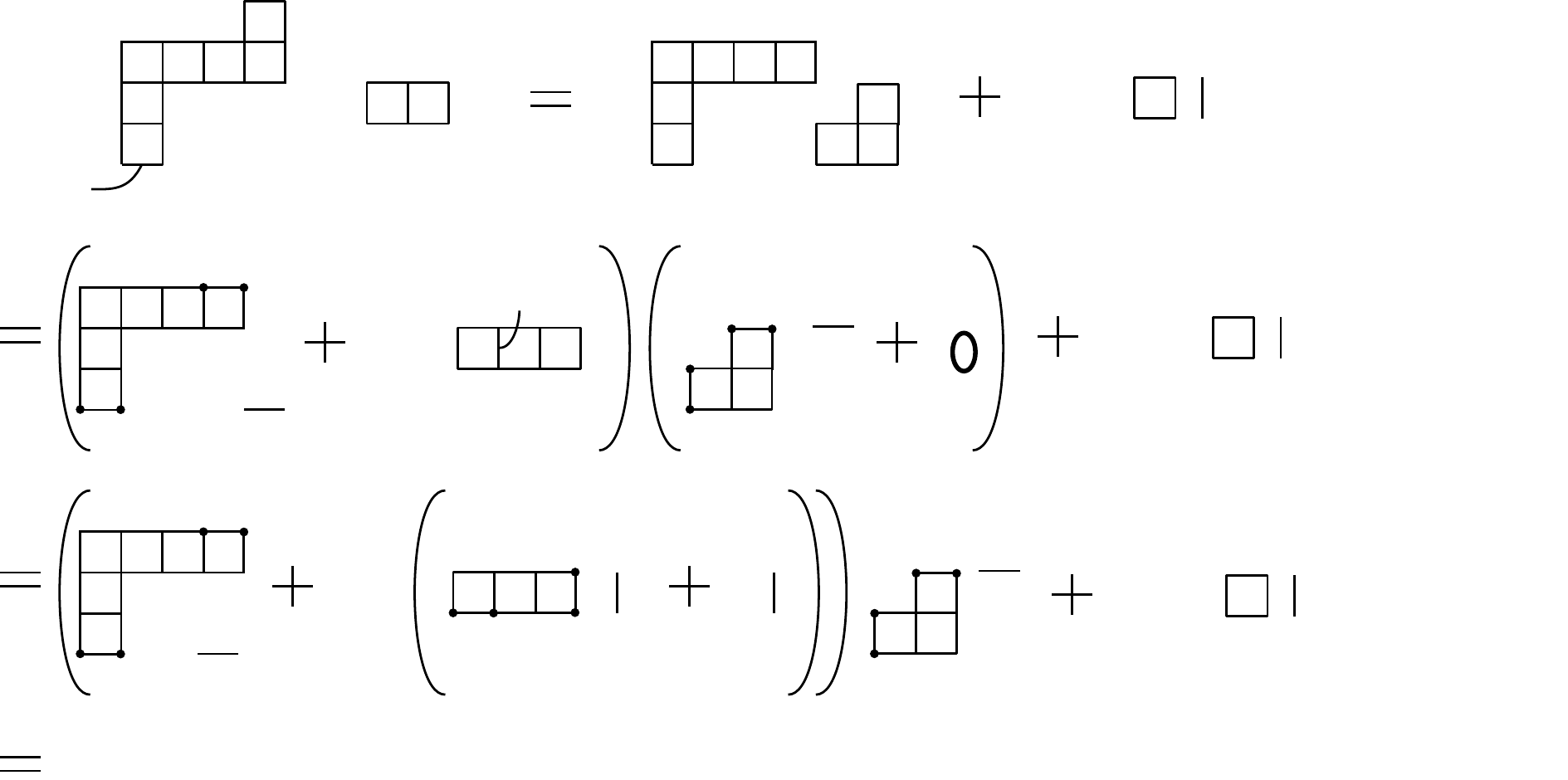}
\caption{Snake graph calculus computing  the product of two cluster variables}
\label{torus2}
\end{center}
\end{figure}

We want to compute the product of the two corresponding cluster variables and express this product in the basis $\B$. This could be done using the smoothing operation 4 times to smooth the 4 crossings of the two arcs. Instead of using the smoothing operation, we do the computation with snake graphs, see Figure~\ref{torus2}.

The first equation in the figure is the resolution of the second crossing overlap given by the tiles with labels 3,4. The graph $\calg_6$ consists of the single edge with label 3, coming from the east edge of the last tile in the snake graph $\calg_1$.
The second equation uses self-grafting with $s=d$ to replace the first snake graph by  a product of the band graph with tiles 1,3,4,1,3,4 and the single edge 2 plus the
selfcrossing snake graph with tiles 4,1,3, and to replace the second snake graph with the product of the band graph with tiles 3,4,1 and the single edge 2 plus zero.
The third equation uses self-grafting with $s=d$ to replace the self-crossing snake graph with tiles 3,4,1 by the product of the corresponding band graph and the single edge $b$ (for boundary). 

For the computation of the coefficients, we choose the orientation such that the south edge of the first snake graph is minimal, see Figure \ref{torus2}. Thus the term $\calg_5\calg_6$ in the first row of the figure carries the coefficient $y(\tcalg_{max})=y_1^2y_3^2y_4^2$.
Recall that the second row is obtained from the first by two self-graftings with $s=d$, thus the two terms $\calg_3\calg_4$ do not carry a coefficient. In the first parenthesis in the second row, the snake graph with tiles 4,1,3 is the term $\calg_{56}$ of the self-grafting. The grafting edge $\zd_3$ is minimal, whereas the grafting edge $\zd_3'$ is maximal, and consequently, the coefficient  $y_{56}'=y_1y_3$ corresponds to the initial segment of the self-crossing snake graph. In the second parenthesis of the second row, the term $\calg_{56}$ is zero and does not carry a coefficient. Recall that the third row is obtained from the second by a self-grafting with $s=d$ on the graph with tiles 4,1,3.
The single edge with label 2 and coefficient $y_4$ is the term $\calg_{56}$ of this self-grafting. The grafting edge $\zd_3$ is minimal, whereas the grafting edge $\zd_3'$ is maximal, and consequently, the coefficient  $y_{56}'=y_4$ corresponds to the initial segment of the self-crossing snake graph.

Finally, in the last row we rewrite the expression as
\[(\textup{Brac}_2\zeta\, x_2 +y_1y_3\,x_\zeta +y_1y_3y_4\,x_2)\,x_\zeta\,x_2+y_1^2y_3^2y_4^2\,x_1'x_3,\]
where $\zeta $ is the loop that crosses 1,3,4 and $x_1'$ is the cluster variable obtained from the initial cluster by mutating in direction $1$.}

{}

\end{document}